\documentclass[11pt,oneside,reqno]{amsart}
\usepackage{amsmath,amsfonts,amssymb,bbm}
\usepackage[active]{srcltx}
\usepackage{dsfont}
\usepackage{xfrac}
\usepackage{xcolor}
\usepackage{tikz}
\usepackage{cancel}
\usepackage{pgfplots}
\usetikzlibrary{intersections}
\usepackage{tikz-3dplot}
\usetikzlibrary{shadings,intersections,patterns}
\usepackage{pgfplots}
\usepgfplotslibrary{fillbetween}
\usetikzlibrary{patterns}
\usetikzlibrary{intersections, calc, angles}
\usepackage[outline]{contour}
\contourlength{0.05em}

\usepackage[shortlabels]{enumitem}

\usepackage[unicode,hypertexnames=false,colorlinks=true,linkcolor=blue,citecolor=blue]{hyperref}

\usepackage{cfr-lm}
\usepackage[osf]{libertine}
\usepackage{hyperref}
\usepackage[capitalise, nameinlink, noabbrev]{cleveref} %Leave it as the last package
\hypersetup{colorlinks=true,
	linkcolor=blue,
	citecolor=red,
}

\hypersetup{colorlinks,breaklinks,
	linkcolor=[rgb]{0,0,0},
	citecolor=[rgb]{0,0,0},
	urlcolor=[rgb]{0,0,0}}

\newcommand{\be}{\begin{equation}}
\newcommand{\ee}{\end{equation}}
\theoremstyle{plain}

\newtheorem{theorem}{Theorem}[section]
\newtheorem{proposition}[theorem]{Proposition}
\newtheorem{corollary}[theorem]{Corollary}
\newtheorem{lemma}[theorem]{Lemma}
\newtheorem{definition}[theorem]{Definition}
\theoremstyle{remark}
\newtheorem{remark}[theorem]{Remark}

\newcommand\ddfrac[2]{\frac{\displaystyle #1}{\displaystyle #2}}
\newcommand{\R}{\mathbb{R}}

\newcommand{\N}{\mathbb{N}}
\newcommand{\eps}{\varepsilon}
\newcommand{\ind}{\mathds{1}}
\newcommand{\HH}{\mathcal{H}}
\newcommand{\loc}{{{\tiny{\mbox{loc}}}}}

\newcommand\intn{- \hspace{-0.42cm} \int}
\newcommand{\meantext}[1]{\,-\hskip-0.88em\int_{#1}} %media integrale nel testo
\newcommand{\abs}[1]{\lvert {#1} \rvert}
\newcommand{\norm}[2]{\Vert {#1} \Vert_{#2}}
\def\O{\Omega}
\DeclareMathOperator{\dive}{div}

\numberwithin{equation}{section}

\title{A capillarity one-phase Bernoulli free boundary problem}

\author[L. Ferreri]{Lorenzo Ferreri}\thanks{}
\address {Lorenzo Ferreri \newline \indent
	Classe di Scienze, Scuola Normale Superiore \newline \indent
	Piazza dei Cavalieri 7, 56126 Pisa - ITALY}
\email{lorenzo.ferreri@sns.it}

\author[G. Tortone]{Giorgio Tortone}\thanks{}
\address {Giorgio Tortone \newline \indent
	Dipartimento di Matematica, Universit\`a di Pisa \newline \indent
	Largo B. Pontecorvo 5, 56127 Pisa - ITALY}
\email{giorgio.tortone@dm.unipi.it}

\author[B. Velichkov]{Bozhidar Velichkov}\thanks{}
\address {Bozhidar Velichkov \newline \indent
	Dipartimento di Matematica, Universit\`a di Pisa \newline \indent
	Largo B. Pontecorvo 5, 56127 Pisa - ITALY}
\email{bozhidar.velichkov@unipi.it}

\begin{document}
	
	\thanks{{\bf Acknowledgments.}
		The authors are supported by the European Research Council (ERC), through the European Union's Horizon 2020 project ERC VAREG - \it Variational approach to the regularity of the free boundaries \rm (grant agreement No. 853404). L.F. and G.T. are members of INDAM-GNAMPA. B.V. is also supported by PRA\_2022\_14 and PRIN 2022
	}
	%\date{\today}
	\subjclass[2010] {
		%49Q10,
		%35R11,
		%47A75,
		%49R05
		%35J70, % degenerate elliptic equations
		% Fractional partial differential equations
		%35B40, % Asymptotic behavior of solutions
		%35B44, % Blow-up
		%35B05, %zero
		%35B53, % Liouville theorems
		%35R35, % free boundary
		%35K67, % Singular parabolic equations
	}
	%\keywords{Capillarity, Free boundary problem, viscosity solution, improvement of flatness.}
	%\tableofcontents
	\begin{abstract}
		We consider a one-phase Bernoulli free boundary 
		problem in a container $D$ -- a smooth open subset of $\R^d$ --  under the condition that on the fixed boundary $\partial D$ the normal derivative of the solutions is prescribed. We study the regularity of the free boundary (the boundary of the positivity set of the solution) up to $\partial D$ and the structure of the {\it wetting region}, which is the contact set between the free boundary and the ($(d-1)$-dimensional) fixed boundary $\partial D$. In particular, we characterize the contact angle in terms of the {\it permeability} of the porous container and we show that the boundary of the wetting region is a smooth $(d-2)$-dimensional manifold, up to a (possibly empty) closed set of Hausdorff dimension at most $d-5$.
	\end{abstract}
	

	\maketitle
	
	\tableofcontents
	\section{Introduction}
	Let $B_1$ be the unit ball in $\R^d$. The classical one-phase Bernoulli free boundary problem consists in finding a non-negative function $u:B_1\to\R$ which solves the following boundary value problem in $B_1$:
	\begin{equation}\label{e:bernoulli}\begin{cases}
	\Delta u=0\quad\text{in}\quad B_1\cap\{u>0\},\medskip\\
	|\nabla u|=\sqrt{Q(x)}\quad\text{on}\quad B_1\cap\partial\{u>0\},
	\end{cases}
	\end{equation}
	where $Q$ is a given positive $C^{0,\alpha}$ function bounded away from zero.
	A variational approach to this problem was introduced by Alt and Caffarelli in \cite{altcaf}, where they showed that solutions can be obtained by minimizing the functional
	\begin{equation}\label{e:alt-caf}
	\int_{B_1}|\nabla u|^2+Q(x)\ind_{\{u>0\}}\,dx\,,
	\end{equation}
	among all functions with prescribed trace on $\partial B_1$. The regularity and the structure of the free boundary $\partial\{u>0\}$ of minimizers $u$ of \eqref{e:alt-caf} were studied in \cite{altcaf,caf1,caf2,caf3,cjk-dim3,js,dj,desilva,esv,ee,w} (see also \cite{velectures}). \\
	
	Free boundary problems of the form \eqref{e:bernoulli} arise in Physics and Engineering, in models of flame propagation, fluid dynamics, thermal insulation, and shape optimization. For instance, in the modelization of jet flows, $\mathbf{v}:=\nabla u$ is interpreted as the velocity field of a fluid; the potential $u$ is the associated streamline function and  the system  \eqref{e:bernoulli} is deduced from the Euler equations associated to the motion of an incompressible, inviscid fluid with constant density (see for example \cite{laminar,friedman}).

	\begin{center}
		\begin{figure}[h]
			\tikzset{
				pattern size/.store in=\mcSize, 
				pattern size = 5pt,
				pattern thickness/.store in=\mcThickness, 
				pattern thickness = 0.3pt,
				pattern radius/.store in=\mcRadius, 
				pattern radius = 1pt}
			\makeatletter
			\pgfutil@ifundefined{pgf@pattern@name@_m7re6w4su}{
				\pgfdeclarepatternformonly[\mcThickness,\mcSize]{_m7re6w4su}
				{\pgfqpoint{0pt}{0pt}}
				{\pgfpoint{\mcSize+\mcThickness}{\mcSize+\mcThickness}}
				{\pgfpoint{\mcSize}{\mcSize}}
				{
					\pgfsetcolor{\tikz@pattern@color}
					\pgfsetlinewidth{\mcThickness}
					\pgfpathmoveto{\pgfqpoint{0pt}{0pt}}
					\pgfpathlineto{\pgfpoint{\mcSize+\mcThickness}{\mcSize+\mcThickness}}
					\pgfusepath{stroke}
			}}
			\makeatother
			
			% Pattern Info
			
			\tikzset{
				pattern size/.store in=\mcSize, 
				pattern size = 5pt,
				pattern thickness/.store in=\mcThickness, 
				pattern thickness = 0.3pt,
				pattern radius/.store in=\mcRadius, 
				pattern radius = 1pt}
			\makeatletter
			\pgfutil@ifundefined{pgf@pattern@name@_hdte5rnxa}{
				\pgfdeclarepatternformonly[\mcThickness,\mcSize]{_hdte5rnxa}
				{\pgfqpoint{0pt}{0pt}}
				{\pgfpoint{\mcSize+\mcThickness}{\mcSize+\mcThickness}}
				{\pgfpoint{\mcSize}{\mcSize}}
				{
					\pgfsetcolor{\tikz@pattern@color}
					\pgfsetlinewidth{\mcThickness}
					\pgfpathmoveto{\pgfqpoint{0pt}{0pt}}
					\pgfpathlineto{\pgfpoint{\mcSize+\mcThickness}{\mcSize+\mcThickness}}
					\pgfusepath{stroke}
			}}
			\makeatother
			\tikzset{every picture/.style={line width=0.75pt}} %set default line width to 0.75pt        
			
			\begin{tikzpicture}[x=0.75pt,y=0.75pt,yscale=-1,xscale=1]
			\draw  [draw opacity=0][fill={rgb, 255:red, 155; green, 155; blue, 155 }  ,fill opacity=0.3 ] (200.88,156.15) .. controls (210.69,162.43) and (216.76,172.22) .. (217.78,182.09) -- (184,181) -- cycle ;
			%uncomment if require: \path (0,300); %set diagram left start at 0, and has height of 300
			
			%Shape: Rectangle [id:dp8999409900561441] 
			\draw  [draw opacity=0][pattern=_m7re6w4su,pattern size=6pt,pattern thickness=0.75pt,pattern radius=0pt, pattern color={rgb, 255:red, 155; green, 155; blue, 155},fill opacity=0.4 ] (17,17.49) -- (384.33,17.49) -- (384.33,63.49) -- (17,63.49) -- cycle ;
			%Shape: Rectangle [id:dp09347058739871783] 
			\draw  [draw opacity=0][pattern=_hdte5rnxa,pattern size=6pt,pattern thickness=0.75pt,pattern radius=0pt, pattern color={rgb, 255:red, 155; green, 155; blue, 155},fill opacity=0.4 ] (19,182.49) -- (386.33,182.49) -- (386.33,221.49) -- (19,221.49) -- cycle ;
			%Straight Lines [id:da08384491249953974] 
			\draw [line width=1.5]    (181.33,63.49) -- (384.33,63.49) ;
			%Straight Lines [id:da41540128284143085] 
			\draw [line width=1.5]  [dash pattern={on 5.63pt off 4.5pt}]  (17,63.49) -- (181.33,63.49) ;
			%Straight Lines [id:da6293466210139147] 
			\draw [line width=1.5]    (190.33,181.49) -- (384.33,181.49) ;
			%Straight Lines [id:da26091283285883105] 
			\draw [line width=1.5]  [dash pattern={on 5.63pt off 4.5pt}]  (21,181.49) -- (190.33,181.49) ;
			%Curve Lines [id:da187243941748934] 
			\draw [line width=1.5]    (181.33,63.49) .. controls (180.33,79.49) and (205.33,68.49) .. (207.33,97.49) .. controls (209.33,126.49) and (235.33,127.49) .. (233.33,141.49) .. controls (231.33,155.49) and (209.33,141.49) .. (184.33,180.49) ;
			%Straight Lines [id:da028853900882760186] 
			\draw    (184.33,180.49) -- (208.33,143.49) ;
			%Shape: Pie [id:dp08846327496671957] 
			%Shape: Rectangle [id:dp21276048501470135] 
			%Shape: Rectangle [id:dp6827874118406492] 
			%Straight Lines [id:da47431718388507416] 
			\draw    (371.33,174.49) -- (371.33,127.49) ;
			\draw [shift={(371.33,124.49)}, rotate = 90] [fill={rgb, 255:red, 0; green, 0; blue, 0 }  ][line width=0.09]  [draw opacity=0] (8.93,-4.29) -- (0,0) -- (8.93,4.29) -- cycle    ;
			
			% Text Node
			\draw (43,101.4) node [anchor=north west][inner sep=0.75pt]    {$\mathrm{div}\,\mathbf{v} =0\ $};
			% Text Node
			\draw (224,92.4) node [anchor=north west][inner sep=0.75pt]    {$|\mathbf{v} |^{2} =1\ \text{on}\ \partial \Omega \cap D$};
			% Text Node
			\draw (25,194.4) node [anchor=north west][inner sep=0.75pt]    {$\mathbf{v} \cdot e_{d} =\beta \ \text{on}\ \partial \Omega \cap\partial P $};
			% Text Node
			\draw (23,38.89) node [anchor=north west][inner sep=0.75pt]    {$\mathbf{v} \cdot e_{d} =0\ \text{on}\ \partial \Omega \cap \partial S$};
			% Text Node
			\draw (376,157.4) node [anchor=north west][inner sep=0.75pt]    {$e_{d}$};
			% Text Node
			\draw (41,130.4) node [anchor=north west][inner sep=0.75pt]    {$( D\mathbf{v}) \ \mathbf{v} =-\nabla p\ $};
			% Text Node
			\draw (152,111.4) node [anchor=north west][inner sep=0.75pt]    {$\text{in}\ \Omega \ $};
			% Text Node
			\draw (209,161.4) node [anchor=north west][inner sep=0.75pt]    {$\theta $};
			\end{tikzpicture}
			\caption{This picture represents a streamline associated with an incompressible, inviscid fluid with constant density in a tube.  In such case, the fluid is underlying an impermeable solid wall $S$ (homogeneous Neumann condition on $\partial S$) and a permeable porous media $P$ (non-homogeneous Neumann condition on $\partial P$).\vspace{-0.3cm}}
			\label{fig2}
		\end{figure}
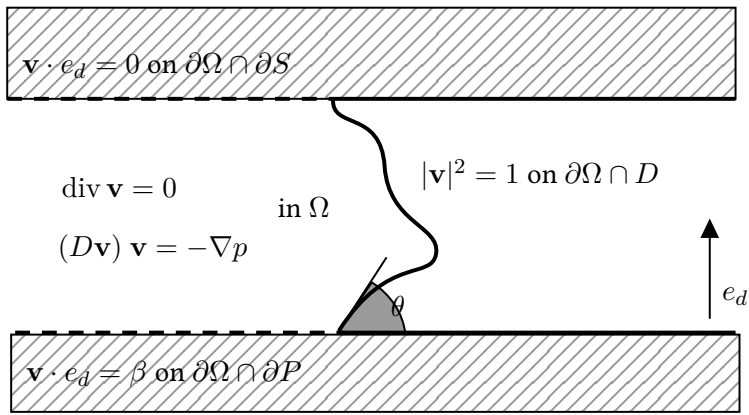
	\end{center}

	When the fluid is confined in a container $D$, the stream function $u$ still satisfies the Bernoulli equations \eqref{e:bernoulli} in $D\cap B_1$, while the interaction between the fluid and the domain wall $\partial D$ is encoded by the boundary condition on $\partial D\cap B_1$. \medskip
	
	One of the possible boundary conditions is the \emph{no-slip} condition:
	$$u=0\quad\text{on}\quad B_1\cap\partial D,$$
	which leads to the following system for $u$ 
	\begin{equation}\label{e:bernoulli-no-slip}\begin{cases}
	\Delta u=0 &\text{in}\quad B_1\cap D\cap\{u>0\},\medskip\\
	|\nabla u|=\sqrt{Q(x)} &\text{on}\quad B_1\cap D\cap\partial\{u>0\},\\
	u=0\quad\text{and}\quad |\nabla u|\ge\sqrt{Q(x)}&\text{on}\quad B_1\cap\partial D.
	\end{cases}
	\end{equation}
	The regularity of the free boundary $\partial\{u>0\}$ and its interaction with $\partial D$ is already well-understood and it is well-known that the variational counterpart of this problem is the classical Alt-Caffarelli problem with an obstacle $\partial D$. Indeed, solutions of this boundary value problem can be obtained by minimizing the functional \eqref{e:alt-caf} among the functions $u$ satisfying the additional constraint $\{u>0\}\subset D$ (or equivalently, $u\equiv 0$ on $B_1\setminus D$). The regularity of the free boundary of the solutions was studied in \cite{cls}; we also refer to \cite{rtv} and \cite{stv} for the case of almost-minimizers and operators with variable coefficients, and to \cite{dsv} for an analysis of the fine structure of the free boundary in dimension two. In particular, in \cite{cls}  it was shown that if $u$ is a solution to \eqref{e:bernoulli-no-slip}, then the free boundary  $\partial\{u>0\}$ approaches $\partial D$ tangentially and can be described by the graph of a $C^{1,\sfrac12}$ function over $\partial D$. \medskip
	
	Other boundary conditions are widely studied in fluid dynamics as they naturally arise in thermal convection problems in industry and geophysics, in particular, in problems with \emph{fluid-porous} interaction (see for instance \cite{ChangChangChen2017:StabilityPoiseuillePorous,HillStraughan2008:PoiseuillePorousNumerical,patne2023stability,phygasdin,generalboundary,generalboundary2} and the references therein). In many cases the permeability of the porous interface is modelled through the following non-homogeneous Neumann condition: 
	$$\partial_{\nu}u=\beta\quad\text{on}\quad B_1\cap\partial D,$$
	where $\nu$ is the exterior normal to $\partial D$. In particular, it is known that (see for instance \cite{HillStraughan2008:PoiseuillePorousNumerical}), in the presence of non-homogeneous fluid-porous interaction, an angle $\theta$ is formed between the streamline and the permeable interface (see Figure \ref{fig2}), which is completely different from the no-slip behavior. Here, in this paper, we develop a free boundary regularity theory for boundary conditions of this type.

	\subsection{The free boundary problem} In the present paper we consider the problem
	\begin{equation}\label{e:bernoulli-D}\begin{cases}
	\Delta u=0 &\text{in}\quad B_1\cap D\cap \{u>0\},\medskip\\
	|\nabla u|=\sqrt{Q(x)} &\text{on}\quad B_1\cap D\cap \partial\{u>0\},\medskip\\
	\partial_\nu u=\beta(x) &\text{on}\quad B_1\cap \partial D\cap \{u>0\},
	\end{cases}
	\end{equation}
	where the container $D$ is a fixed open set with smooth boundary in $\R^d$,  $\nu$ is the exterior normal to $\partial D$ and $\beta:\partial D\to\R$ is a fixed smooth (not necessarily positive) function. In the system \eqref{e:bernoulli-D} the solution $u$ is allowed to be positive on $\partial D$ and the Neumann boundary condition on $\partial D$ corresponds to the case of permeable boundary. This condition gives rise to a new phenomenon with respect to the problem from \cite{cls,dsv}, as it turns out that the free boundaries subjected to this permeability condition do not approach the wall of the container tangentially. Precisely, we will show that the solutions of \eqref{e:bernoulli-D} satisfy a Young-type law, that is, the free boundary $\partial\{u>0\}$
	meets the fixed wall $\partial D$ at an angle depending explicitly on $\beta$ and $Q$.
	We  notice that the non-homogeneous Neumann condition on $\partial D$ in \eqref{e:bernoulli-D} is the free boundary counterpart of the capillarity term in the context of sessile liquid drops \cite{CM1,CM2,CFdrops, dephimaggi}.
	\begin{center}
		\begin{figure}[h]\label{fig:1}
			\begin{tikzpicture}[scale=1.8]
			\draw[looseness=.6,rotate=-3] (-1.95,-1) node[left] {$\partial D$}
			to[bend left] coordinate (msx) (-1.25,0.25)
			to[bend left] coordinate (mp) (2.22,0.28)
			to[bend right] coordinate (mdx) (1.65,-1)
			to[bend right] coordinate (mm) (-1.95,-1)
			-- cycle;
			%\filldraw[white] (1,-0.01) arc (-35:219:1.247cm);
			\filldraw [white] plot [smooth] coordinates {(1,-0.01) (1.22,0.5) (1,1.21) (0.4,1.3) (0,1.8) (-1,1) (-1.13,0.54) (-1,-0.03)};
			\draw[looseness=.6,dashed,rotate=-3] (-1.25,0.25)
			to[bend left] coordinate (mp) (2.22,0.28);
			\filldraw [gray!50, rotate=1] (1,-0.01) arc (0:360:1cm and 0.2cm);
			%  \draw[-latex,gray!50] (1,0) .. controls (1.2,0) and (1.4,-0.1) .. (1.58,-0.17);
			\filldraw[gray] (0.7,0) -- (1,0) -- (1.14,0.235) --(0.7,0);%angolo1
			\draw[fill=gray] (0.7,0) arc (180:56.5:0.3cm);%angolo2
			%  \draw[-latex] (0.8,1.12) -- (1.4,1.5);
			%  \draw (1.45,1.42) node[anchor=south east] {$\nabla u$};
			\draw (0.2,0.75) node {$\Delta u=0$};
			\draw (1.9,0.85) node {$u\equiv0$};
			\node[label={[rotate=80]:\small $|\nabla u|=\sqrt{Q}$}] at (-1.1,0.85) {};
			%\draw[black] (1,-0.01) arc (-35:90:1.247cm);
			\draw[black,dashed, rotate=1] (1,-0.01) arc (0:180:1cm and 0.2cm);
			\draw[black, fill=gray!50,rotate=1] (-1,-0.01) arc (180:360:1cm and 0.3cm);
			\draw[-latex, black] (1,-0) -- (1.57,-0.3);
			\draw (1.3,-0.65) node [anchor =south] {$-\nabla u$};
			\draw[-latex, black, dashed] (0,0) -- (0,-0.63);
			\draw (0.45,-0.55) node {$\nu_{D}$};
			\draw[-latex, black] (1,-0) -- (1.6,0);
			\draw (1.6,-0.05) node [anchor =south] {$-\nabla_T u$};
			\draw (0.6,0)--(1.01,0) -- (1.25,0.4); %angolo3
			% \draw [black] plot [smooth] coordinates {(1,-0.01)  (1.22,0.5) (1,1.21) (0.4,1.3) (0,1.8) (-1,1) (-1.13,0.54) (-1,-0.03)};
			\draw [black] plot [smooth] coordinates {(1,-0.01)  (1.22,0.5) (1.27,0.9)  (1.25,1.3) (1.22, 1.5)};
			\draw [black] plot [smooth] coordinates {(-1,1.4) (-1.1,1.1) (-1.15,0.8) (-1.13,0.54)  (-1,-0.03)};
			\draw [black] plot [smooth] coordinates {(-1.1,1.1) (-1.0,1) (-0.8,0.92) (-0.6,1.1-0.2) (0,1.2-0.2) (0.7,1.05) (1.15,1.2) (1.25,1.3)};
			\draw [black, dashed] plot [smooth] coordinates {(-1.1,1.1) (-1.0,1.25) (-0.6,1.4) (-0.1,1.47) (0.5,1.3+0.2) (1.15,1.4) (1.25,1.3)};
			% \draw [black] plot [smooth] coordinates {(-1.15,1.01) (-1.13,0.54) (-1,-0.03)};
			\end{tikzpicture}
			\caption{This picture represents a  solution to \eqref{e:bernoulli-D} with smooth free boundary.}
			\label{fig1}
		\end{figure}
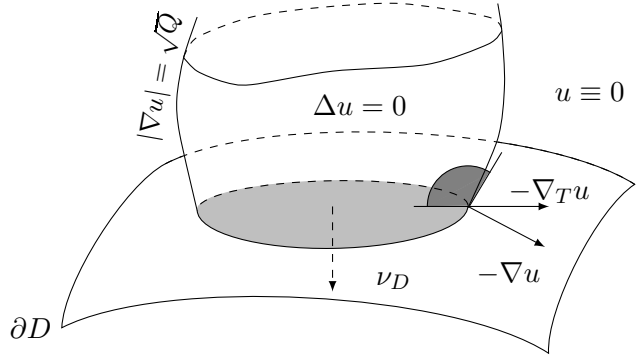
	\end{center}

	\subsection{Variational formulation} The variational problem associated to \eqref{e:bernoulli-D} is the following:
	\be\label{min.toy}
	\!\!\!\min\left\{\int_{D}{|\nabla u|^2 + Q\ind_{\{u>0\}}\mathrm{d}x} + \int_{B_1\cap\partial D}2\beta u\,\mathrm{d}\HH^{d-1}\,:\,u \in H^1(B_1\cap \overline{D}),\, u=\phi \text{ on }\partial B_1 \cap D\right\},
	\ee
	where $\phi \in H^{1/2}(B_1)$ is the boundary datum, and $Q$ and $\beta$ are given H\"{o}lder continuous functions. In order to study the regularity of the minimizers and their free boundary around a point $x_0\in\partial D\cap B_1$ (without loss of generality we can assume that $x_0=0$), we consider a smooth change of coordinates
	$$\Phi:\R^d\to\R^d$$
	such that $\Phi(0)=0$ and in a neighborhood of the origin:
	$$\Phi(D)=\{x_d>0\}\qquad\text{and}\qquad \Phi(\partial D)=\{x_d=0\}.$$
	Setting
	$$E:=\Phi(B_1)\ ,\quad E^+:=E\cap\{x_d>0\}\quad\text{and}\quad  E':=E\cap\{x_d=0\},$$
	we get that if $w\in H^1(D\cap B_1)$ is a solution of the variational problem \eqref{min.toy}, then the function
	$$u\in H^1(E^+)\ ,\quad u=w\circ \Phi^{-1},$$
	is a minimizer of the functional
	\be\label{e:def-F}
	\mathcal F(u,E):=\int_{E^+}\Big(\nabla u\cdot A(x)\nabla u+\ind_{\{u>0\}}Q\Big)\,dx+\int_{E'}2\beta\,u\,dx',
	\ee
	in the sense that
	\begin{equation}\label{e:def-minimality-F}
	\mathcal F(u,E)\le \mathcal F(v,E)\quad\text{for every}\quad v\in H^1(E^+)\quad\text{with}\quad \begin{cases}
	v\ge 0\quad\text{on}\quad E^+\\
	v= u\quad\text{on}\quad\partial E\cap\{x_d>0\}\ ,
	\end{cases}
	\end{equation}
	where $A$ is a symmetric elliptic matrix with variable H\"older coefficients, and $Q$ and $\beta$ are  H\"older continuous functions depending on $\Phi$ and on the corresponding functions in \eqref{min.toy} (see \cref{cor:main} and its proof for a concrete choice of $\Phi$ and the precise expressions of $A$, $Q$ and $\beta$).\medskip
	\begin{definition}\label{d:minimizer}
		In what follows, given an open set $E\subset\R^d$ and a non-negative function $u\in H^1(E^+)$, we will say that $u$ "minimizes the functional $\mathcal F$ in $E$" if \eqref{e:def-minimality-F}  holds.
	\end{definition}	
	
	In the rest of the paper we will work directly with minimizers (in the sense of \cref{d:minimizer}) of the functional $\mathcal F$ (defined in \eqref{e:def-F} above) on an open subset $E$ of $\R^d$. Since we are interested in the local behaviour of the minimizers and the free boundaries, we will assume that $E$ is the ball $B_1$.
	
	\subsection{Notations and main assumptions on the functional}\label{sub:intro-assumptions}
	Before we state our main regularity results, we list the main assumptions on the matrix $A$ and the functions $Q$ and $\beta$ from \eqref{e:def-F}:
	\begin{enumerate}
		\item[$\qquad(\mathcal H1)$] $A$ is a symmetric matrix with variable H\"older continuous coeffients defined on $B_1\cap \{x_d\ge 0\}$,
		$$A \in C^{0,\delta_A}\big(B_1\cap \{x_d\ge 0\};\R^{d\times d}\big),$$
		and $A$ is elliptic, that is there are constants $\Lambda_A\ge 1$ such that
		\begin{equation}\label{e:lambda-A-Lambda-A}
		\Lambda_A^{-2} |\xi|^2 \leq \xi\cdot A(x)\xi \leq \Lambda_A^2 |\xi|^2,\quad\text{for every $\xi \in \R^d$ and every $x\in B_1\cap \{x_d\ge 0\}$};
		\end{equation}
		we also define the function
		\begin{equation}\label{e:def-of-a}
		a:B_1\cap \{x_d\ge 0\}\to\R\ ,\qquad a(x):=\sqrt{e_d\cdot A(x)e_d}\ ,
		\end{equation}
		which is clearly a H\"older continuous bounded from below by $\Lambda_A^{-1}$;\medskip
		\item[$\qquad(\mathcal H2)$] $Q:B_1\cap \{x_d\ge 0\}\to\R$ is a positive H\"older continuous function,
		$$Q\in C^{0,\delta_Q}\big(B_1\cap \{x_d\ge 0\}\big),$$ and there is a constant $\lambda_Q>0$ such that
		\begin{equation}\label{e:lambda-Q}
		\lambda_Q\le Q(x)\quad\text{for every}\quad x\in B_1\cap \{x_d\ge 0\}\,;\smallskip
		\end{equation}
		\item[$\qquad(\mathcal H3)$] $\beta:B_1\cap \{x_d= 0\}\to\R$ is H\"older continuous; there is $\delta_\beta>0$ for which
		$$\beta\in C^{0,\delta_\beta}\big(B_1\cap \{x_d= 0\}\big).$$
	\end{enumerate}	
	\begin{remark}
		We notice that the function $\beta$ might change sign. Moreover, the sign of $\beta$ affects the angle between the free boundary $\partial\{u>0\}$ and the fixed boundary $\{x_d=0\}$; for instance, when $\beta=0$, these two boundaries are orthogonal. The angle between $\partial\{u>0\}$ and $\{x_d=0\}$ is determined by all the quantities $a$, $Q$ and $\beta$; for instance, we will show that when $-a\sqrt{Q}<\beta< a\sqrt{Q}$, they always meet transversally (see \eqref{def.theta} for the precise expression of the angle).
	\end{remark}	
	
	\subsection{On the topology of $\{x_d\ge 0\}$ and the definition of the free boundary}\label{sub:topology} In what follows, given a set $\Omega\subset\R^d\cap\{x_d\ge 0\}$, we will say that $\Omega$ is {\it open} if it is {\it relatively open in $\{x_d\ge 0\}$} (for instrance, according to this convention, the set $H:=\{(x_1,x'',x_d)\in\R^d\ :\ x_d\ge0,\ x_1+x_d>0\}$ is open). In particular, this means that if $\Omega\subset\{x_d\ge 0\}$ is open, then also $\Omega\cap\{x_d=0\}$ is open inside $\{x_d=0\}$. Moreover, we will denote by $\partial\Omega$ {\it the boundary of $\Omega$ with respect to the topology of $\{x_d\ge 0\}$}. This means that the boundary of the set $H$, defined above, is $\partial H=\{(x_1,x'',x_d)\in\R^d\ :\ x_d\ge0,\ x_1+x_d=0\}$.
	\subsection{Regularity of the free boundary}
	\noindent The main result of the present paper is the following.
	\begin{theorem}\label{t:main}
		Let $A$, $Q$ and $\beta$ be as in \cref{sub:intro-assumptions}, and suppose that
		\begin{equation}\label{e:beta-small}	
		|\beta(x)| < a(x)\sqrt{Q(x)}\qquad\text{on}\qquad B_1':=B_1\cap\{x_d=0\}.
		\end{equation}
		Let
		$$u:B_1\cap\{x_d\ge 0\}\to\R\,,\quad u\in H^1(B_1^+)\,,\quad u\ge 0\ \,\text{on}\,\ B_1^+\,,$$
		be a minimizer of $\mathcal{F}$ in $B_1$. Then, $u$ is Lipschitz continuous on $B_1\cap\{x_d\ge 0\}$, the set  $\{u>0\}\cap B_1'$ is an open subset of $B_1'$ and the free boundary can be decomposed as
		$$B_1'\cap\partial\{u>0\}=\mathrm{Reg}(u)\cup \mathrm{Sing}(u),$$
		where:
		\begin{enumerate}
			\item[\rm(i)] $\mathrm{Reg}(u)$ is a relatively open subset of $B_1'\cap\partial\{u>0\}$. Moreover, in a $\R^d$-neighborhood of any point $x_0\in \mathrm{Reg}(u)$, the set $$\partial\{u>0\}\cap \{x_d\ge 0\}$$ is a $(d-1)$-dimensional $C^{1,\alpha}$ manifold with boundary; the boundary of this manifold is precisely $\partial\{u>0\}\cap \{x_d=0\}$ and is, as well, a $(d-2)$-dimensional $C^{1,\alpha}$ manifold.
			\vspace{0.2cm}
			\item[\rm(ii)]  $\mathrm{Sing}(u)$ is a closed set and there is a universal critical dimension $d^* \geq 5$ such that:
			\begin{itemize}
				\item if $d < d^*$, then $\ \mathrm{Sing}(u)$ is empty;
				\item if $d = d^*$, then $\ \mathrm{Sing}(u)$ is a locally finite set of isolated points;
				\item if $d > d^*$, then $\ \mathrm{Sing}(u)$ has Hausdorff dimension at most $d - d^*$.
			\end{itemize}\vspace{0.2cm}
		\end{enumerate}
	\end{theorem}
	
	In \cref{cor:main} below we obtain the regularity of the minimizers of \eqref{min.toy} as a corollary of \cref{t:main}. In particular, \cref{cor:main} implies that, in dimensions $d=2$, $d=3$ and $d=4$, the minimizers of \eqref{min.toy} are classical solutions of the free boundary system \eqref{e:bernoulli-D}.
	\begin{theorem}\label{cor:main}
		Suppose that $D$ is a $C^{1,\alpha}$-regular domain in $\R^d$; let
		$$Q:D\cap B_1\to\R$$ be a H\"older continuous function bounded from below by a positive constant $\lambda_Q>0$; let
		$$\beta:\partial D\cap B_1\to\R$$
		be a H\"older continuous function such that $|\beta|<\sqrt{Q}$ on $\partial D\cap B_1.$\\Let
		$u:B_1\cap \overline D\to\R$ be a non-negative function minimizing the functional
		$$\mathcal F_{\text{\tiny D}}(u,B_1):=\int_{B_1\cap D}\big(|\nabla u|^2+\ind_{\{u>0\}}Q(x)\big)\,dx+\int_{\partial D\cap B_1}2\beta\,u\,d\HH^{d-1},$$
		in $B_1$. Then, $u$ is Lipschitz continuous in $B_1\cap \overline D$ and the positivity set $\{u>0\}$ is a relatively open subset of
		$\overline D$. Let $\partial\{u>0\}$ be the boundary of $\{u>0\}$ with respect to relative topology of $\overline D$. Then, 
		$$\partial\{u>0\}\cap \partial D$$
		can be decomposed into a disjoint union of
		a (relatively open) regular part $\ \mathrm{Reg}(u)$ and a singular part $\ \mathrm{Sing}(u)$ such that:
		\begin{enumerate}
			\item[\rm (i)]  in a neighborhood of any point $x_0=(x_0',0)\in \mathrm{Reg}(u)$, $\partial\{u>0\}\cap \overline D$ is a $(d-1)$-dimensional $C^{1,\alpha}$ manifold with boundary in $B_1$ that meets the hyperplane $\partial D$ at an angle
			\be\label{def.theta}
			\theta=\arctan\left(\frac{1}{|\beta|}\sqrt{Q-{\beta^2}}\right).
			\ee
			In particular, $\partial\{u>0\}\cap \partial D$ is a $(d-2)$-dimensional $C^{1,\alpha}$ embedded submanifold of $\partial D$.
			\vspace{0.2cm}
			\item[\rm (ii)] there is a universal critical dimension $d^* \geq 5$ such that:
			\begin{itemize}
				\item if $d < d^*$, then for any minimizer $u$, the singular set $\ \mathrm{Sing}(u)$ is empty;
				\item if $d = d^*$, then $\ \mathrm{Sing}(u)$ is a locally finite set of isolated points;
				\item if $d > d^*$, then $\ \mathrm{Sing}(u)$ has Hausdorff dimension at most $d - d^*$.	
			\end{itemize}\vspace{0.2cm}
		\end{enumerate}
	\end{theorem}
	\begin{proof}
		Let $x_0\in\partial D$ be a point on the free boundary $\partial\{u>0\}\cap\partial D$. Without loss of generality, we may suppose that $x_0=0$ and that there is a $C^{1,\alpha}$ function
		$$\eta:B_r'\to(-r,r)$$
		such that $\eta(0')=|\nabla_{x'}\eta|(0')=0$ and
		$$D\cap\Big\{B_r'\times(-r,r)\Big\}=\Big\{(x',x_d)\in B_r'\times(-r,r)\ :\ x_d>\eta(x')\Big\}.$$
		Then, the function
		$$v(x',x_d):=u\big(x',x_d+\eta(x')\big),$$
		defined in $C:=B_\delta\cap\{x_d\ge 0\}$ for some $\delta$ small enough is a local minimizer of the functional
		\be\label{F}
		\mathcal F(v,B_\delta):=\int_{B_\delta^+}\big(\nabla v\cdot A(x)\nabla v+\ind_{\{v>0\}}\widetilde Q\big)\,dx+\int_{B_\delta'}2\widetilde \beta\,v\,dx',
		\ee
		where
		$$A(x',x_d):=\begin{pmatrix}1&\dots&0&-\partial_1\eta(x')\\
		\vdots&\ddots&\vdots&\vdots\\
		0&\dots&1&-\partial_{d-1}\eta(x')\smallskip\\
		-\partial_1\eta(x')&\dots&-\partial_{d-1}\eta(x')&1+|\nabla_{x'}\eta|^2
		\end{pmatrix}\ ,$$
		$$\widetilde Q(x',x_d)=Q(x',x_d+\eta(x'))\qquad\text{and}\qquad \widetilde \beta(x',0)=\beta(x',\eta(x'))\sqrt{1+|\nabla_{x'}\eta(x')|^2}.$$
		Since by construction
		$$|\widetilde\beta(0)|=|\beta(0)|< \sqrt{Q(0)}=\sqrt{\widetilde Q(0)},$$
		we can choose $\delta$ small enough in such a way that
		$$|\widetilde\beta|<\sqrt{\widetilde Q}\quad\text{on}\quad B_\delta,$$
		so the claim follows from \cref{t:main}.
	\end{proof}	

	\section{Main steps of the proof of \cref{t:main} and further results}
	We recall the notation
	\begin{equation}\label{e:def-E+E'}
	E^+:=E\cap\{x_d>0\}\qquad\text{and}\qquad E':=E\cap\{x_d=0\},
	\end{equation}
	for any set $E\subset \R^d$, thus $B_1\cap \{x_d\ge 0\}=B_1^+\cup B_1'$. From now on we consider a function
	$$u:B_1^+\cup B_1'\to\R\ ,\quad u\in H^1(B_1^+)\ ,\quad u\ge 0\quad \text{on}\quad B_1^+\cup B_1',$$
	which minimizes $\mathcal F$ in $B_1$ (in the sense of \cref{d:minimizer}), where $\mathcal F$ is given by \eqref{e:def-F} with $A$, $Q$ and $\beta$ satisfying the conditions from \cref{sub:intro-assumptions}.
	\subsection{On the regularity of the free boundary away from the hyperplane}\label{rem:above} In the open set $B_1^+$, away from the hyperplane $\{x_d=0\}$, any minimizer $u$ of the functional $\mathcal F$ from \eqref{e:def-F} is a minimizer of the one-phase Bernoulli problem with variable coefficients, that is,
	$$\int_{B_1^+}\big(\nabla u\cdot A(x)\nabla u+\ind_{\{u>0\}}Q\big)\,dx\le \int_{B_1^+}\big(\nabla v\cdot A(x)\nabla v+\ind_{\{v>0\}}Q\big)\,dx,$$
	for any $v\in H^1(B_1^+)$ such that $u-v\in H^1_0(B_1^+)$. Thus, the regularity of $u$ and of the free boundary $\partial\{u>0\}$ in the interior of $B_1^+$ follows from the regularity results for solutions of the one-phase problem with variable coefficients: we refer to \cite{dets} and \cite{trey1,trey2}; and also to \cite{stv} for the $2$-dimensional case and to \cite{rtv} for the special case of operators of the form $A=a(x)\text{\rm Id}$.\smallskip

	Thus, the focus of the present paper is on:
	\begin{itemize}
		\item the regularity of the free boundary $\partial\{u>0\}\cap (B_1^+\cup B_1')$ up to the hyperplane $\{x_d=0\}$; \smallskip
		\item the structure of the contact sets $\{u>0\}\cap B_1'$ and its boundary $\partial\big(\{u>0\}\cap B_1'\big)$ as subsets of the $(d-1)$-dimensional hyperplane $\{x_d=0\}$.
	\end{itemize}

	\subsection{Lipschitz continuity of the minimizers and its consequences}\label{sub:intro-lipschitz}
	In \cref{s:lipschitz} we will prove the following Theorem (we stress that here we do not assume \eqref{e:beta-small}).
	\begin{theorem}[Lipschitz continuity of $u$]\label{t:lipschitz-continuity}
		Let $A$, $Q$ and $\beta$ be as in \cref{sub:intro-assumptions}, and let the function
		$$u:B_1\cap\{x_d\ge 0\}\to\R\,,\ \, u\in H^1(B_1^+)\,,\ \, u\ge 0\ \,\text{on}\,\ B_1^+\,,$$
		be a minimizer of the functional $\mathcal{F}$ (from \eqref{e:def-F}) in $B_1$. Then, $u$ is Lipschitz continuous in  $B_{\sfrac12}\cap\{x_d\ge 0\}$.
	\end{theorem}
	
	The key steps in the proof of \cref{t:lipschitz-continuity} are the following. First, in \cref{s:change}, we prove that the minimizers of \eqref{e:def-F} satisfy an almost-minimality condition for the corresponding functional with "freezed" coefficients and we use this information in \cref{l:holder0} to prove that $u$ is continuous; then, in \cref{l:lap.est1}, we estimate the total variation of the measure $\text{div}(A\nabla u)$ on balls $B_r$; in these preliminary lemmas, we use methods, which were already known (see for instance \cite{velectures}). \smallskip
	
	\noindent 
	The crux of the proof of \cref{t:lipschitz-continuity} lies in \cref{l:iteration-interior} and \cref{l:iteration-boundary}, where we show that $\meantext{B_r^+}{u}$ controls $\meantext{B_{2r}^+}{u}$ up to a summable error term; the key point in  this estimate is an almost-monotonicity-type formula for the mean value of $u$ over the level sets of $A$-capacitary functions on the ring $B_{2r}\setminus B_r$. We then iterate this estimate (\cref{p:mean-value-estimate}) to prove that $\meantext{B_r^+}{u}\lesssim r$ and we conclude the proof of \cref{t:lipschitz-continuity} by  interior and boundary gradient estimates (\cref{l:gradest-ball} and \cref{l:gradest-half-ball}).

	\begin{remark}[The positivity set is open]
		In particular, \cref{t:lipschitz-continuity} implies that the set
		$$\Omega_u:=\Big\{x\in B_1\cap\{x_d\ge0\}\ :\ u(x)>0\Big\},$$
		is relatively open in the half-ball (with boundary) $B_1\cap\{x_d\ge0\}$; and also that the set
		$$\Omega_u':=\Omega_u\cap\{x_d=0\},$$
		is an open subset of the hyperplane $B_1':=B_1\cap \{x_d=0\}$.
	\end{remark}
	
	As in the classical Alt-Caffarelli's problem, the Lipschitz continuity of $u$ implies the compactness of the blow-up sequences. Precisely, let $x_0\in \partial\Omega_u\cap B_1'$ be fixed and, for every $r>0$ small enough, let
	$$u_{x_0,r}(x):=\frac1ru(x_0+rx),$$
	where we notice that if
	$$\overline B_{\rho}(x_0)\subset B_1\qquad\text{and}\qquad \|\nabla u\|_{L^\infty(B_\rho(x_0))}\le L\,,$$
	for some $L>0$, then $u_{x_0,r}$ is defined and $L$-Lipschitz on $B_{\sfrac{\rho}r}$. Thus, for every sequence $r_k\to0$, there are a subsequence $r_{k_n}\to0$ and a non-negative $L$-Lipschitz function
	$$u_0:\R^d\cap\{x_d\ge0\}\to\R\,$$
	such that $u_{x_0,r_{n_k}}$ converges to $u_0$ locally uniformly in $\R^d\cap\{x_d\ge0\}$. As usual,  we  will say that
	\begin{center}$u_0$ is a {\it blow-up limit} of $u$ at $x_0$. \end{center}

	\subsection{Basic results for the blow-up limits} The analysis of the blow-up limits is essential for the study of the structure of the free boundary $\partial\Omega_u$; in particular, in \cref{s:blow-up} we will prove the following basic facts about the blow-up limits at a free boundary point
	$$x_0\in B_1'\cap\partial\{u>0\}.$$
	As above, we suppose that $A$, $Q$ and $\beta$ satisfy the assumptions from \cref{sub:intro-assumptions}.
	\begin{enumerate}[\rm (B1)]
		\item If $\,\beta(x_0)+ a(x_0)\sqrt{Q(x_0)}>0\,$,
		then every blow-up limit $u_0$ of $u$ at $x_0$ is not identically zero. This follows from the {\it non-degeneracy} estimate on the solution $u$ (\cref{p:non-degeneracy}), which we prove by combining the interior non-degeneracy with a clean-up argument based on viscosity sliding method. Finally, the Lipschitz continuity and the non-degeneracy provide a lower density estimate for $\O$ (\cref{cor:lower.H}).\medskip
		\item For every blow-up limit $u_0$, there is a $d\times d$ real matrix $M_{x_0}$ such that the function
		\begin{equation}\label{e:def-tilde-u-0-intro}
		\widetilde u_0:\R^d\cap\{x_d\ge0\}\to\R\ ,\quad \widetilde u_0(x):=u_{0}(M_{x_0}(x)),
		\end{equation}
		is a {\it global minimizer} of the one-phase functional
		\begin{equation}\label{e:def-F-x-0-intro}
		\mathcal F_{x_0}(v,B_R):=\int_{B_R^+}|\nabla v|^2\,\mathrm{d}x + Q(x_0)\big|\{v>0\}\cap B_R^+\big|+
		2\,\frac{\beta(x_0)}{a(x_0)}\int_{B_R'}\,|v| \,\mathrm{d}\HH^{d-1}\,,
		\end{equation}
		in the sense that, for every ball $B_R\subset\R^d$, we have
		$$\mathcal F_{x_0}(\widetilde u_0,B_R)\le \mathcal F_{x_0}(v,B_R)\quad\text{for every}\quad v\in H^1(B_R^+)\quad\text{with}\quad v=\widetilde u_0\quad\text{on}\quad (\partial B_R)^+.$$
		Notice that this follows essentially from the fact $\widetilde u_0$ is a blow-up of the function
		$$\widetilde u:=u\circ M_{x_0}\quad\text{at the point}\quad\widetilde x_0:=M_{x_0}^{-1}(x_0).\medskip$$
		\item In \cref{p:weiss} we prove a monotonicity formula for $\widetilde u$ in the spirit of \cite{w}, from which we deduce that the blow-up limit $\widetilde u_0$ is one-homogeneous.
		In particular, this allows to reduce the classification of the blow-up limits to the classification of the one-homogeneous global minimizers of $\mathcal F_{x_0}$, which we discuss in the next subsection.
	\end{enumerate}	
	
	\subsection{Analysis of the one-homogeneous global minimizers}\label{sub:intro-one-homogeneous-global-minimizers}
	Given two constants
	$$q>0\quad\text{and}\quad m\in(-q,q),$$
	a set $E\subset\R^d$ and a function $v:\R^d\cap\{x_d\ge0\}\to\R$ in $H^1_{loc}(\R^d)$, we define
	\begin{equation}\label{e:def-J-intro}
	J(v,E):=\int_{E^+}|\nabla v|^2\,\mathrm{d}x + q^2\big|\{v>0\}\cap E^+\big|+
	2\,m\int_{E'}\,|v| \,\mathrm{d}\HH^{d-1}\,,
	\end{equation}
	where $E^+$ and $E'$ are as in \eqref{e:def-E+E'}. We say that a function $u:\R^d\cap\{x_d\ge0\}\to\R$ is a global minimizer of $J$ if $u$ is non-negative and
	$$J(u,B_R)\le J(v,B_R)\quad\text{for every}\quad v\in H^1(B_R^+)\quad\text{with}\quad v= u\quad\text{on}\quad (\partial B_R)^+,$$
	for every ball $B_R\subset\R^d$. Then, the blow-up limit $\widetilde u_0$ defined in \eqref{e:def-tilde-u-0-intro} is a global minimizer of the functional
	\be\label{equivalenteJF}J=\mathcal F_{x_0}\quad\text{with}\quad q=\sqrt{Q(x_0)}\quad\text{and}\quad m=\frac{\beta(x_0)}{a(x_0)}\ .\ee
	We notice that the global minimizers of $J$ are Lipschitz continuous and that solve the PDE
	\begin{equation}\label{e:bernoulli-J}\begin{cases}
	\Delta u=0\quad\text{in}\quad \{x_d>0\}\cap \{u>0\},\medskip\\
	|\nabla u|=q\quad\text{on}\quad \{x_d>0\}\cap \partial\{u>0\},\medskip\\
	e_d\cdot\nabla u=m\quad\text{on}\quad \{x_d=0\}\cap \{u>0\},
	\end{cases}
	\end{equation}
	where the first and the third equations hold in the classical sense, while the second equation is intended in viscosity sense (see for instance \cite{desilva}), which becomes classical where $\{x_d>0\}\cap \partial\{u>0\}$ is smooth. 
	
	\begin{remark}
		We stress that in \eqref{e:bernoulli-J} it is not specified what are the conditions satisfied by $u$ at points  $\partial\{u>0\}\cap \{x_d=0\}$. One consequence of \cref{t:main} is that $\partial\{u>0\}\cap \{x_d=0\}$ has codimension two and zero capacity, but this is not an information that we have a priori for minimizers. Thus, we show that $u$ satisfies a suitable boundary condition in viscosity sense also at these points (see \cref{sub:viscosity} and \cref{s:visco}); this boundary condition will be crucial in the proof of \cref{t:epsilon-regularity}. 
	\end{remark}
	
	\begin{definition}[Half-plane solutions]\label{def:half-plane}
		Let $d\ge 2$. Given the real constants  $q>0$ and $m\in(-q,q)$, and a unit vector $\nu\perp e_d$, we consider the function
		$$h_{q,m,\nu}:\R^d\cap\{x_d\ge0\}\to\R\ ,$$
		defined as
		$$h_{q,m,\nu}(x):=\Big(\sqrt{q^2-m^2}\,(x\cdot \nu) + m\,(x\cdot e_d)\Big)^+\ ,$$
		and we notice that $h_{q,m,\nu}$ satisfies \eqref{e:bernoulli-J}.
	\end{definition}	
	
	The main results concerning the one-homogeneous global minimizers of $J$ are the following.
	
	\begin{theorem}[A Caffarelli-Jerison-Kenig type stability inequality]\label{t:stability-inequality-CJK-type}
		Let $u:\R^d\cap \{x_d\geq 0\}\to\R$ be a non-negative one-homogeneous global minimizer of $J$ with
		$q>0$ and
		$$
		-q<m<q.
		$$
		Suppose that $\R^{d-1}\setminus\{0\}\subset \text{\rm Reg}(u)$. Then $\{u>0\}$ supports the following stability inequality:
		\be\label{e:stability-inequality}
		\int_{\{u>0\}^+}|\nabla v|^2\,dx\ge\int_{(\partial\{u>0\})^+}|H| v ^2\,d\HH^{d-1}\,,
		\ee
		for every $v \in C^\infty_c(\R^d\setminus \{0\})$, where $H$ is the mean curvature of $\partial\{u>0\}$.
	\end{theorem}

	\begin{theorem}[Non-existence of singular cones in low dimension]\label{t:stability}
		Suppose that $d\in\{2,3,4\}$ and that $q>0$, $m\in\R$ are constants such that
		$$-q<m< q.$$
		Let $ u:\R^d\cap\{x_d\ge0\}\to\R$ be a (non-negative) one-homogeneous global minimizer of the functional $J$ defined in \eqref{e:def-J-intro}. Then, $u$ is necessarily of the form $h_{q,m,\nu}$, for some unit vector $\nu$ orthogonal to $e_d$.
	\end{theorem}
	
	We will first prove \cref{t:stability} in dimension $d=2$ (see \cref{s:global2}). We will then use \cref{p:classification2D} to show that the minimizers are viscosity solutions (as we will explain in  \cref{sub:viscosity} below). Then,
	using this result, we will prove an epsilon-regularity Theorem (see \cref{t:epsilon-regularity}). At this point, we can apply the classical Federer's dimension reduction principle to reduce the analysis to the cones which are smooth outside the origin (see for instance \cite{velectures}). Finally, in \cref{s:stability.sect}, we will prove the validity of the  stability inequality of \cref{t:stability-inequality-CJK-type} and we complete the proof of \cref{t:stability} for $d=3$ and $d=4$.
	
	\subsection{Decomposition of the free boundary}\label{s:decomposition}
	Let $u$ be a minimizer of $\mathcal F$ with $A$, $Q$ and $\beta$ as in \cref{sub:intro-assumptions}. Let $x_0\in B_1'\cap\partial\{u>0\}$ be a point on the free boundary such that
	$$|\beta(x_0)|<a(x_0)\sqrt{Q(x_0)}.$$
	In view of the analysis carried out in \cref{s:global2} we decompose the free boundary $B_1'\cap\partial \{u>0\}$ as follows. Given $x_0\in B_1'\cap\partial \{u>0\}$, we say that:
	\begin{enumerate}[$\bullet$]
		\item $x_0$ is a regular point ($x_0\in\text{\rm Reg}(u)$) if
		there exists a blow-up limit $u_0$ of $u$ at $x_0$ for which $\widetilde u_0$ is a half-plane solution (as in \cref{def:half-plane}) with 
		$$q=\sqrt{Q(x_0)}\qquad\text{and}\qquad\displaystyle m=\frac{\beta(x_0)}{a(x_0)}.$$
		\item $x_0$ is a singular point ($x_0\in\text{\rm Sing}(u)$), if $x_0\notin \text{\rm Reg}(u).$
		\end{enumerate}

	\subsection{Viscosity solutions and epsilon-regularity}\label{sub:viscosity}
	In what follows, we will write every point $x\in\R^d$ as $x=(x',x_d)$ with $x'\in\R^{d-1}$ and $x_d\in\R$. Moreover, for functions $\varphi:\R^d\to\R$ we will indicate with $\nabla_{x'}\varphi$ the part of the gradient $\nabla\varphi$ tangent to $\{x_d=0\}$, that is,
	$$\nabla_{x'}\varphi(x):=\nabla \varphi(x)-\big(e_d\cdot \nabla \varphi(x)\big)e_d\,.$$
	\begin{definition}
		Given a set $\Omega\subset\R^d$ and two functions $u:\Omega\to\R$ and $v:\Omega\to\R$, we say that
		$u$ touches from below (res. from above) $v$ at $x_0\in\Omega$ if
		$$u(x_0)=v(x_0)\quad\text{and}\quad u\le v\ \big(\text{resp.}\ u\ge v\big)\quad \text{in}\quad \Omega.$$
	\end{definition}	
	\begin{definition}[Viscosity solutions]\label{def:solutionnew}
		Let $A$, $Q$ and $\beta$ be as in \cref{sub:intro-assumptions} and suppose that
		$$|\beta|<a\sqrt{Q}\quad\text{on}\quad B_1'.$$
		Let $u:B_1^+\cup B_1'\to\R$ be a non-negative continuous function. We say that $u$ satisfies
		\be\label{e:interno}
		\text{\rm div}\big(A\nabla u\big)=0\,\mbox{ in }\, B_1^+\cap\{u>0\},\qquad A \nabla u \cdot e_d = \beta\,\mbox{ on }\,B_1'\cap \{u>0\}\\
		\ee
		and
		\begin{align}\label{FB}
		\begin{split}
		|A^{1/2}\nabla u|= \sqrt{Q} &\quad \mbox{on}\quad  B_1^+\cap \partial\{u>0\},\\
		|A^{\sfrac12}\nabla_{x'} u |= \sqrt{Q-\frac{\beta^2}{a^2}} &\quad\mbox{on}\quad B_1'\cap \partial\{u>0\},
		\end{split}
		\end{align}
		in a viscosity sense if the following conditions are fulfilled.
		\begin{enumerate}
			\item[\rm(VS1)] \label{interno} $u$ satisfies \eqref{e:interno} in a weak-$H^1$ sense in $(B_1^+\cup B_1')\cap\{u>0\}$.
			\item[\rm(VS2)]\label{prima} Let $\psi \in C^2$ be defined in a neighborhood $\mathcal N\subset \R^d$ of $x_0 \in B_1^+\cap \partial\{u>0\}$ and let $|\nabla \psi|(x_0) \neq 0$.\\ Suppose that, in $\mathcal N$, $\psi^+$ touches $u$ from below (resp. from above) at $x_0$. Then
			$$
			|A^{\sfrac12}\nabla \psi|(x_0)\leq \sqrt{Q(x_0)}\quad \left(\text{resp. $|A^{\sfrac12}\nabla \psi|(x_0)\geq \sqrt{Q(x_0)}$}\right).
			$$
			\item[\rm(VS3)]\label{terza} Let $\psi \in C^2$ be defined in a neighborhood $\mathcal N\subset \R^d$ of $x_0 \in B_1'\cap \partial\{u>0\}$ and let $|\nabla_{x'} \psi|(x_0) \neq 0$. Suppose that, in $\mathcal N\cap\{x_d\ge 0\}$, $\psi^+$ touches $u$ from below (resp. from above) at $x_0$. Then, at least one of the following conditions holds:
			\begin{itemize}
				\item[(i)] 
				$|A^{\sfrac12}\nabla_{x'} \psi|(x_0)\leq \sqrt{Q(x_0)-\frac{\beta(x_0)^2}{a(x_0)^2}}\qquad \left(\text{resp. $|A^{\sfrac12}\nabla_{x'} \psi|(x_0)\geq \sqrt{Q(x_0)-\frac{\beta(x_0)^2}{a(x_0)^2}}$}\ \right).$
				
				\item[(ii)]$A \nabla \psi(x_0) \cdot e_d \le \beta(x_0) \qquad \left(\text{resp. $A \nabla \psi(x_0) \cdot e_d \ge \beta(x_0)$} \right).$ 
			\end{itemize}
		\end{enumerate}
	\end{definition}
	\begin{remark}
		We notice that these the conditions in {\rm(VS3)} mean that if $\psi$ is a function such that
		$$|A^{\sfrac12}\nabla_{x'} \psi|(x_0)> \sqrt{Q(x_0)-\frac{\beta(x_0)^2}{a(x_0)^2}}\qquad \text{and}\qquad A \nabla \psi(x_0) \cdot e_d > \beta(x_0),$$
		which in particular means that also 
		$$|A^{\sfrac12}\nabla \psi|(x_0)> \sqrt{Q(x_0)},$$
		then $\psi^+$ cannot touch $u$ from below at $x_0$. Conversely, if $\psi$ is such that
		$$|A^{\sfrac12}\nabla_{x'} \psi|(x_0)< \sqrt{Q(x_0)-\frac{\beta(x_0)^2}{a(x_0)^2}}\qquad \text{and}\qquad A \nabla \psi(x_0) \cdot e_d < \beta(x_0),$$
		so that, in particular,
		$$|A^{\sfrac12}\nabla \psi|(x_0)< \sqrt{Q(x_0)},$$
		then $\psi^+$ cannot touch $u$ from above at $x_0$.
	\end{remark}
	In \cref{s:visco} we will show that the minimizers of $\mathcal F$ are viscosity solutions. 
	\begin{theorem}[Minimizers are viscosity solutions]\label{t:viscosity}
		Let $A$, $Q$ and $\beta$ be as in \cref{sub:intro-assumptions} and suppose that
		$$|\beta|<a\sqrt{Q}\quad\text{on}\quad B_1'.$$
		Let $u$ be a local minimizer of $\mathcal{F}$ in $B_1$. Then, $u$ satisfies \eqref{e:interno}-\eqref{FB} in viscosity sense.
	\end{theorem}
	
	\begin{remark}
		We notice that, due to the capillarity term in the functionals $\mathcal F$ and $J$, it is not known if the presence of a one-sided tangent ball at some point $x_0\in\partial\{u>0\}\cap \{x_d=0\}$ implies that the blow-up is a half-plane solution. Thus, in \cref{s:visco}, we develop a different strategy, which is based on a dimension reduction argument and the classification of the blow-up limits in dimension two. This idea is general and can also find application in other contexts, for instance, in minimal surfaces. 
	\end{remark}

	The following is our main epsilon-regularity result for viscosity solutions, which follows by standard iteration arguments from an improvement-of-flatness result (see \cref{l:IMPF}).
	\begin{theorem}[Epsilon-regularity for viscosity solutions]\label{t:epsilon-regularity}
		Let $A$, $Q$ and $\beta$ be as in \eqref{sub:intro-assumptions}. Suppose that
		$$A(0)=\text{\rm Id}\qquad\text{and}\qquad |\beta|<a\sqrt{Q}\quad\text{on}\quad B_1'.$$
		There are constants $\eps,\sigma>0$ for which the following holds. Let $u:B_1^+\cup B_1'\to\R$ be a non-negative continuous function such that:
		\begin{enumerate}[$\bullet$]
			\item $u$ is a viscosity solution to \eqref{e:interno}-\eqref{FB} in $B_1$ and $0\in \partial\{u>0\}$;
			\item $u$ is $\eps$-flat in $B_1$, that is, there is a unit vector $\nu\in\R^d$ orthogonal to $e_d$ and such that
			$$h_{q,m,\nu}(x-\eps\nu)\le u(x)\le h_{q,m,\nu}(x+\eps\nu)\quad\text{for every}\quad x\in B_1^+\cap B_1',$$
			where $\ \displaystyle q:=\sqrt{Q(0)}\ $ and $\ \displaystyle m:=\frac{\beta(0)}{a(0)}\,.$
		\end{enumerate}
		Suppose that
		$$
		|A(x)-\text{\rm Id}|\leq \eps^\sigma,\qquad |Q(x)-Q(0)|\leq Q(0)\eps^\sigma,\qquad |\beta(x)-\beta(0)|\leq \beta(0)\eps^\sigma
		$$
		then $\partial\{u>0\}\cap\{x_d\ge0\}$ is a  $C^{1,\alpha}$ regular manifold with boundary in $B_{\sfrac12}$.
	\end{theorem}

	\subsection{Proof of \cref{t:main}} The Lipschitz continuity of a minimizer $u$ is proved in \cref{t:lipschitz-continuity}. We decompose the free boundary $\partial\{u>0\}\cap\{x_d=0\}$ into a regular and singular parts according to the blow-up limits as in \cref{s:decomposition}. When a point $x_0\in \partial\{u>0\}\cap\{x_d=0\}$ is regular, then a rescaling of $u$ satisfies the hypotheses of \cref{t:epsilon-regularity}. This implies the regularity of the free boundary $\partial\{u>0\}\cap\{x_d\ge0\}$ in a neighborhood of $x_0$. It remains to estimate the dimension of the singular set. Since, the blow-ups are homogeneous, by \cref{t:stability}, we know that there are no singular points in dimension $d\in \{2,3,4\}$. Finally, the dimension estimates in (ii) of \cref{t:main} follow again by \cref{t:stability} and a standard dimension reduction argument.

	\section{Change of variables}\label{s:change}
	%In $\R^d$ we use the coordinates $x=(x',x_d)=(x_1,x'',x_d)$, with $x'\in\R^{d-1}$ and $x''\in\R^{d-2}$.
	\subsection{Notations}\label{sub:notations} In $\R^d$ we use the coordinates $x=(x',x_d)$, with $x'\in\R^{d-1}$ and $x_d\in\R$.\\
	For any fixed $x_0 \in B_1\cap \{x_d\geq 0\}$, we define the function
	\be\label{T}
	T_{x_0} \colon \R^d \to \R^d\ ,\qquad T_{x_0}(\xi) = x_0+ A^{\sfrac12}(x_0)\xi .
	\ee
	Given a function $f$, we will often use the notation $f^{x_0} := f \circ T_{x_0}$, thus
	\be\label{e:def-of-u_x_0}
	u^{x_0}(x) := u\big(T_{x_0}(x)\big)=u\big(x_0+ A^{\sfrac12}(x_0)x\big).
	\ee
	Similarly, for any $r>0$, we set
	\begin{equation}\label{e:def-E-r}
	E_r(x_0):=T_{x_0}(B_r).
	\end{equation}
	Notice that $E_r(x_0)$ is an ellipsoid centered in $x_0$ and that we have
	\be\label{ellip}
	B_{\Lambda_A^{-1} r}(x_0)\subset E_r(x_0)\subset B_{\Lambda_A r}(x_0),
	\ee
	where $\Lambda_A$ is the constant from \eqref{e:lambda-A-Lambda-A}.\medskip
	
	By $H_{x_0}'$ we denote the hyperplane such that
	$$\{x_d=0\}=T_{x_0}(H_{x_0}').$$
	When $x_0=(x_0',0)$, the hyperplane $H'_{x_0}$ is given by
	$$H'_{x_0}=\Big\{x\in\R^d\ :\ x\cdot e_d^{x_0}=0\Big\},$$
	where $e_d^{x_0}$ is the unit vector defined by
	\be\label{e.a}
	e_d^{x_0}:= \frac{A^{\sfrac12}(x_0)e_d}{a(x_0)} \qquad\mbox{with}\qquad a(x_0):=\sqrt{e_d\cdot A(x_0)e_d}\ \,.
	\ee
	Finally, by $H_{x_0}^+$ we denote the half-space
	$$H_{x_0}^+:=T_{x_0}^{-1}\big(\{x_d>0\}\big)\ ,$$
	and we notice that by construction
	$$H_{x_0}^+:=\{x\in\R^d\ :\ x\cdot e_d^{x_0}>0\}\qquad\text{and}\qquad\partial \big(H_{x_0}^+\big)=H_{x_0}'\ .$$
	
	\subsection{Almost-minimality}
	Fixed a point $x_0=(x_0',0)\in\R^d$, we consider the following functionals defined for every open set $E\subset\R^d$ and every function $\varphi\in H^1\big(E\cap H_{x_0}^+\big)$\ :
	\be\label{e:def-G-inside}
	\mathcal{G}^+_{x_0}(\varphi,E):=\int_{E\cap H_{x_0}^+}|\nabla\varphi|^2 \,\mathrm{d}x + Q(x_0)\big|\{\varphi>0\}\cap E\cap H_{x_0}^+\big|\,,
	\ee
	\be\label{e:def-G-boundary}
	\mathcal{G}_{x_0}'(\varphi,E):=2\,\frac{\beta(x_0)}{a(x_0)}\int_{E\cap H_{x_0}'}\,\varphi \,\mathrm{d}\HH^{d-1}\,,
	\ee
	where we use the notation from the previous Section. We also notice that, when $H_{x_0}'=\{x_d=0\}$,
	\be\label{e:FeG}
	\mathcal{F}_{x_0}(\varphi,E)=\mathcal{G}_{x_0}^+(\varphi,E)+\mathcal{G}_{x_0}'(\varphi,E),
	\ee
	where $\mathcal{F}_{x_0}$ is the functional from \eqref{e:def-F-x-0-intro}. In the next proposition, we prove an almost-minimality condition for the function $u^{x_0}$ from \eqref{e:def-of-u_x_0} which we will use to prove the continuity of $u$; once we prove that $u$ is Lipschitz continuous (\cref{t:lipschitz-continuity}), we will improve the almost-minimality result in \cref{l:freezing}.
	\begin{proposition}\label{l:freezing_0}
		Let $A$, $Q$ and $\beta$ be as in $(\mathcal H1)$, $(\mathcal H2)$ and $(\mathcal H3)$. Then, there are constants
		$$C_A =C(d,A),\quad C_Q=C(A,d,Q),\quad C_\beta=C(A,d,\beta),$$
		such that the following holds. Given a local minimizer $u$ of $\mathcal{F}$ in $B_1'$, a point $x_0 \in B_1$ and a radius$r\in (0,1)$ such that $B_{\Lambda_A r}(x_0)\subset B_1$, we have that the function $u^{x_0}$ from \eqref{e:def-of-u_x_0} satisfies the following almost-minimality condition:
		\begin{align}\label{almost0}
		\begin{aligned}
		\mathcal{G}^+_{x_0}(u^{x_0},B_r) &+ \Big(1+C_A r^{\delta_A}\Big)\mathcal{G}_{x_0}'(u^{x_0},B_r)\\
		&\leq\,\, \Big(1+3C_A r^{\delta_A}\Big)\mathcal{G}^+_{x_0}(\tilde{u},B_r) + \Big(1+C_A r^{\delta_A}\Big)\mathcal{G}_{x_0}'(\tilde u,B_r)\\
		&\qquad\qquad\qquad\qquad\qquad\,+{C}_Q r^{d+\delta_Q} + {C}_\beta r^{\delta_\beta}\norm{u^{x_0}-\tilde{u}}{L^1(B_r\cap H_{x_0}')}\ ,
		\end{aligned}
		\end{align}
		for every $\tilde{u}\in H^1(B_r\cap H^+_{x_0})$ such that $\tilde{u}=u^{x_0}$ on $\partial B_r\cap H^+_{x_0}$.
	\end{proposition}
	\begin{proof}
		By hypothesis $B_{\Lambda_A r}(x_0)\subset B_1$. Thus, from \eqref{ellip} it follows that $E_r(x_0)\subset B_1$.\\ Consider the competitor $v$ defined through the identity $\tilde{u}=v\circ T_{x_0}$ (in the notations from \cref{sub:notations} $\tilde u$ is precisely the function $v^{x_0}$). Since $u^{x_0}=u\circ T_{x_0}$ we get that $v = u$ on $\partial E_r(x_0)\cap \{x_d>0\}$. Thus, by the minimality of $u$ we have
		\begin{equation}\label{e:35gradicazzo}
		\mathcal{F}(u,E_r(x_0))\leq \mathcal{F}(v,E_r(x_0)),
		\end{equation}
		where $\mathcal F$ is the functional from \eqref{e:def-F}. By the H\"{o}lder continuity of $Q$ and $\beta$, and from the fact that $E_r(x_0)\subset B_{\Lambda_Ar}$ (see \eqref{ellip}), we get
		$$
		|Q(x)-Q(x_0)|\leq C(A,Q) r^{\delta_Q}\quad\text{and}\quad
		|\beta(x)-\beta(x_0)|\leq C(A,\beta) r^{\delta_\beta}\quad\mbox{in}\quad E_r^+(x_0).
		$$
		Thus, we can find constants $C_Q$ and $C_\beta$ such that
		\begin{align*}
		\int_{E_r^+(x_0)}Q\left(\ind_{\{v>0\}}-\ind_{\{u>0\}}\right)\mathrm{d}x &\leq   Q(x_0)\int_{E_r^+(x_0)}\left(\ind_{\{v>0\}}-\ind_{\{u>0\}}\right)\,\mathrm{d}x + C_Q r^{d+\delta_Q}\ ,\\
		\int_{E_r'(x_0)}\beta(v-u)\,\mathrm{d}x' &\leq   \beta(x_0) \int_{E_r'(x_0)}(v-u)\,\mathrm{d}x' + C_\beta r^{\delta_Q}\norm{v-u}{L^1(E_r'(x_0))}\ .
		\end{align*}
		Now, using the change of variables formula, we have
		\begin{align*}
		\int_{E_r^+(x_0)}\left(\ind_{\{v>0\}}-\ind_{\{u>0\}}\right)\,\mathrm{d}x &=\mathrm{det}\big(A^{\sfrac12}(x_0)\big)\int_{B_r\cap H_{x_0}^+}\left(\ind_{\{v^{x_0}>0\}}-\ind_{\{u^{x_0}>0\}}\right)\,\mathrm{d}x\ ,\\
		\int_{E_r'(x_0)}(v-u)\,\mathrm{d}x' &=\frac{\mathrm{det}\big(A^{\sfrac12}(x_0)\big)}{a(x_0)}\int_{B_r\cap H_{x_0}'}\left(\ind_{\{v^{x_0}>0\}}-\ind_{\{u^{x_0}>0\}}\right)\,\mathrm{d}\HH^{d-1}\ ,\\
		\norm{v-u}{L^1(E_r'(x_0))}&=\frac{\mathrm{det}\big(A^{\sfrac12}(x_0)\big)}{a(x_0)}\,\norm{v^{x_0}-u^{x_0}}{L^1(B_r\cap H_{x_0}')}\ .
		\end{align*}
		Thus, we get
		\begin{align}\label{e:Q-beta-100metri}
		\begin{aligned}
		\int_{E_r^+(x_0)}Q&\left(\ind_{\{v>0\}}-\ind_{\{u>0\}}\right)\mathrm{d}x+2\int_{E_r'(x_0)}\beta(v-u)\,\mathrm{d}x' \\
		&\le Q(x_0)\mathrm{det}\big(A^{\sfrac12}(x_0)\big)\int_{B_r\cap H_{x_0}^+}\left(\ind_{\{v^{x_0}>0\}}-\ind_{\{u^{x_0}>0\}}\right)\,\mathrm{d}x\\
		&\qquad+\mathrm{det}\big(A^{\sfrac12}(x_0)\big)\,2\,\frac{\beta(x_0)}{a(x_0)}\int_{B_r\cap H_{x_0}'}\left(\ind_{\{v^{x_0}>0\}}-\ind_{\{u^{x_0}>0\}}\right)\,\mathrm{d}\HH^{d-1}\\
		&\qquad\qquad+ C_Q r^{d+\delta_Q}+ C_\beta r^{\delta_Q}\norm{v^{x_0}-u^{x_0}}{L^1(H'_{x_0})}\ .
		\end{aligned}
		\end{align}	
		Similarly, using the H\"{o}lder continuity and the ellipticity of $A$, we can estimate
		\begin{align}\label{e:cambio}
		\begin{aligned}
		\int_{B_r\cap H_{x_0}^+}|\nabla u^{x_0}|^2 \mathrm{d}x &= \mathrm{det}\big(A^{-\sfrac12}(x_0)\big)\int_{E_r^+(x_0)}\nabla u\cdot A(x_0)\nabla u\,\mathrm{d}x\\
		&\leq \mathrm{det}\big(A^{-\sfrac12}(x_0)\big)\,\big(1+C_Ar^{\delta_A}\big)\int_{E_r^+(x_0)}\nabla u\cdot A(x)\nabla u\,\mathrm{d}x\ ,
		\end{aligned}
		\end{align}
		and similarly
		\begin{align}\label{e:cambio2}
		\begin{aligned}
		\int_{B_r\cap H_{x_0}^+}|\nabla v^{x_0}|^2 \mathrm{d}x  &= \mathrm{det}\big(A^{-\sfrac12}(x_0)\big)\int_{E_r^+(x_0)}\nabla v\cdot A(x_0)\nabla v\,\mathrm{d}x\\
		&\geq \frac{\mathrm{det}(A^{-\sfrac12}(x_0))}{1+C_A r^{\delta_A}}\int_{E_r^+(x_0)}\nabla v\cdot A(x)\nabla v \,\mathrm{d}x\ .
		\end{aligned}
		\end{align}
		Now, by rearranging the terms in \eqref{e:35gradicazzo} and using \eqref{e:Q-beta-100metri}, \eqref{e:cambio} and \eqref{e:cambio2}, we get
		\begin{align}\label{e:almost-versione00}
		\begin{aligned}
		&\frac1{1+C_A r^{\delta_A}}\int_{B_r\cap H_{x_0}^+}|\nabla u^{x_0}|^2 \mathrm{d}x +Q(x_0)\big|B_r\cap H_{x_0}^+\cap\{u^{x_0}>0\}\big| +\mathcal{G}_{x_0}'(u^{x_0},B_r)\\
		&\qquad\qquad\leq\,\, 	\Big(1+C_A r^{\delta_A}\Big)\int_{B_r\cap H_{x_0}^+}|\nabla v^{x_0}|^2 \mathrm{d}x +Q(x_0)\big|B_r\cap H_{x_0}^+\cap\{v^{x_0}>0\}\big|\\
		&\qquad\qquad\qquad\qquad\qquad\,+\mathcal{G}_{x_0}'(v^{x_0},B_r)+C_Q r^{d+\delta_Q} + C_\beta r^{\delta_\beta}\norm{v^{x_0}-u^{x_0}}{L^1(B_r\cap H_{x_0}')}\ .
		\end{aligned}
		\end{align}
		Finally, multiplying both sides by $1+C_A r^{\delta_A}$, we obtain \eqref{almost0}.
	\end{proof}

	\section{Lipschitz continuity}\label{s:lipschitz}
	In this section we study the optimal regularity of the functions $u$ that minimize $\mathcal{F}$ and we provide a detailed proof of \cref{t:lipschitz-continuity}.
	
	\subsection{H\"older continuity of the minimizers}
	We start by providing a non-optimal regularity result. For the sake of completeness, we mention that the argument in the proof of \cref{l:holder0} could be easily extended by showing $C^{0,\alpha}$ regularity, for every $\alpha \in (0,1)$. Nevertheless, we avoid this technical generalization since it would not simplify the proof of \cref{t:lipschitz-continuity}.
	\begin{lemma}\label{l:holder0}
		Let $u\in H^1(B_1^+)\cap L^\infty(B_1^+)$ be a local minimizer of $\mathcal{F}$ in $B_1$ in the sense of \cref{d:minimizer}, the functional $\mathcal F$ being given by \eqref{e:def-F}. Then, there are a radius $R_0$ and a constant $C$, both depending on $d,A,Q$ and $\beta$, such that $u\in C^{0,\sfrac12}\big(B_{\sfrac{R_0}2}\cap \{x_d\geq 0\}\big)$.
	\end{lemma}
	\begin{proof}
		By the translation invariance of the problem, we can assume that $x_0=0, r\in (0,R_0)$ (with $R_0$ to be defined later). Set $u^0:=u \circ T_0$ with $T_0(x)=A^{\sfrac12}(0)x$. Let now $v_r$ be the following half-space harmonic replacement of $u^0$
		$$
		\begin{cases}
		\Delta v_r=0 & \mbox{in } B_r\cap H_0^+ \\
		\nabla v_r \cdot e_d^{0}=0 & \mbox{on } B_r\cap H_{0}'\\
		v_r=u^0 & \mbox{on } \partial B_r\cap H_0^+\ ,
		\end{cases}
		$$
		where $e_d^{0}$, $H_0^+ $ and $H_0'$ were defined in \cref{sub:notations}.
		Integration by parts, we get that
		\be\label{h}
		\int_{B_r^+}\nabla v_r\cdot\nabla (v_r-u^0)\mathrm{d}x=0,
		\ee
		and so we get
		\begin{align*}
		\mathcal{G}_{0}^+(u^{0},B_r) - \mathcal{G}_{0}^+(v_r,B_r) &= \int_{B_r\cap H_0^+}|\nabla (u^0-v_r)|^2\mathrm{d}x - Q(x_0)|\{u^0=0\}\cap B_r\cap H_0^+|.\\
		\mathcal{G}_{0}^+(v_r,B_r) &\leq \int_{B_r\cap H_0^+}|\nabla u^0|^2\mathrm{d}x + Q(x_0)|B_r\cap H_0^+|.
		\end{align*}
		On the other hand, by considering $v_r$ as a competitor in the almost-minimality condition \eqref{almost0}, we get
		\begin{align}\label{utili}
		\begin{aligned}
		\int_{B_r\cap H_0^+}|\nabla (u^0-v_r)|^2\mathrm{d}x \leq&\, C_A r^{\delta_A}\int_{B_r\cap H_0^+}|\nabla u^0|^2\mathrm{d}x +  (1+C_A r^{\delta_A})Q(x_0)|B_r\cap H_0^+|\\
		&+{C}_Q r^{d+\delta_Q} + \left(\frac{\beta(0)}{a(0)}+ \frac{\beta(0)}{a(0)}C_A r^{\delta_A}+{C}_\beta r^{\delta_\beta}\right)\norm{u^{0}-v_r}{L^1(B_r\cap H_{0}')}\ .
		\end{aligned}
		\end{align}
		Set $\eps_0=1/2$, then
		\be\label{1/2}
		\norm{u^{0}-v_r}{L^1(B_r\cap H_{0}')} \leq
		2\norm{u^0}{L^\infty(B_{R_0}\cap H_0^+)} r^{d-2\eps_0}
		\ee
		and for $r\in (0,R_0)$ it holds
		\begin{align*}
		\int_{B_r\cap H_0^+}|\nabla (u^0-v_r)|^2 \mathrm{d}x \leq &\,
		C_A r^{\delta_A}\int_{B_r\cap H_0^+}|\nabla u^0|^2\mathrm{d}x + (1+C_A r^{\delta_A})Q(x_0)\omega_d r^d\\
		&\qquad+2\left( \frac{\beta(0)}{a(0)}+C_A r^{\delta_A}\frac{\beta(0)}{a(0)}+{C}_\beta r^{\delta_\beta} \right)\norm{u^0}{L^\infty(B_{R_0}\cap H_0^+)} r^{d-2\eps_0}+ {C}_Q r^{d+\delta_Q}\\
		\leq&\, C_Ar^{\delta_A}\int_{B_r\cap H_0^+}|\nabla u^0|^2\mathrm{d}x + C(A,d,Q,\beta)\Big(1+\norm{u^0}{L^\infty(B_{R_0}\cap H_0^+)}\Big) r^{d-2\eps_0}.
		\end{align*}
		Now pick a radius $\rho > r$. Since $|\nabla v_\rho|^2$ is subharmonic we first get
		$$
		\int_{B_r\cap H_0^+}|\nabla v_\rho|^2\mathrm{d}x \leq \left(\frac{r}{\rho} \right)^{d} \int_{B_\rho\cap H_0^+}|\nabla v_\rho|^2\mathrm{d}x,
		$$
		from which we deduce
		\begin{align*}
		\int_{B_r\cap H_0^+}{\abs{\nabla u^0}^2\mathrm{d}x} & \leq 2 \bigg(\int_{B_\rho\cap H_0^+}{\abs{\nabla (u^0-v_\rho)}^2\mathrm{d}x} + \int_{B_r\cap H_0^+}{\abs{\nabla v_\rho}^2\mathrm{d}x}\bigg)\\
		& \leq 2\int_{B_\rho\cap H_0^+}{\abs{\nabla (u^0-v_\rho)}^2\mathrm{d}x} + 2\left(\sfrac{r}{\rho}\right)^{d} \int_{B_\rho\cap H_0^+}{\abs{\nabla v_\rho}^2\mathrm{d}x}\\
		& \leq \tilde{C} \rho^{d-2\eps_0} + \frac{\tilde{C}}2\Big( \rho^{\delta_A}+\left(\sfrac{r}{\rho}\right)^{d}\Big)\int_{B_\rho\cap H_0^+}{\abs{\nabla u^0}^2\mathrm{d}x}  ,
		\end{align*}
		where we choose the constant $\tilde{C}$ to be greater than $4$, $2C_A$ and ${C}(A,d,Q,\beta)\big(1+\norm{u^0}{L^\infty(B_{R_0})}\big)$.
		
		Now, fixed $t<1/2$ such that $\tilde{C} t^{2\eps_0}=\tilde{C} t=1/2$, we choose $R_0>0$ such that
		\be\label{rzero}R_0^{\delta_A} \leq t^d .\ee
		Then, for $r=t\rho< \rho \leq R_0$
		$$
		\int_{B_{t\rho}\cap H_0^+}{\abs{\nabla u^0}^2\mathrm{d}x} \leq \tilde{C} \rho^{d-2\eps_0} + \tilde{C}t^{d} \int_{B_\rho\cap H_0^+}{\abs{\nabla u^0}^2\mathrm{d}x},
		$$
		and  by iterating we get
		\begin{align*}
		\int_{B_{\rho t^{k}}\cap H_0^+} \abs{\nabla u^0}^2 \mathrm{d}x &\leq \tilde{C}(t^{k-1}\rho)^{d-2\eps_0} + \tilde{C}t^d \int_{B_{\rho t^{k-1}}\cap H_0^+} \abs{\nabla u^0}^2 \mathrm{d}x\\
		&\leq \tilde{C}(t^{k-1}\rho)^{d-2\eps_0} + \tilde{C}(t^{k-2}\rho)^{d-2\eps_0}\tilde{C}t^d +\big(\tilde{C}t^d\big)^2\int_{B_{\rho t^{k-2}}\cap H_0^+} \abs{\nabla u^0}^2 \mathrm{d}x\\
		&\dots\\
		&\leq \tilde{C}(t^{k-1}\rho)^{d-2\eps_0}\sum_{i=0}^{k-1}(\tilde{C}t^{2\eps_0})^i + \big(\tilde{C} t^{d}\big)^k \int_{B_{\rho}\cap H_0^+} \abs{\nabla u^0}^2 \mathrm{d}x.
		\end{align*}
		So, given $0<r\leq \rho\leq R_0$, such that $t^{k+1}\rho<r\leq t^k \rho$, we get
		\begin{align*}
		\int_{B_{r}\cap H_0^+} \abs{\nabla u^0}^2 \mathrm{d}x &\leq 2\tilde{C} r^{d-2\eps_0}t^{-2(d-2\eps_0)} + \left(\frac{r}{\rho}\right)^{d-2\eps_0} t^{-d} \int_{B_{\rho}\cap H_0^+} \abs{\nabla u^0}^2 \mathrm{d}x\\
		&\leq (2\tilde{C})^{1/(2\eps_0)}\bigg( 1 + \rho^{2\eps_0}\intn_{B_{\rho}\cap H_0^+} \abs{\nabla u^0}^2 \mathrm{d}x\bigg)r^{d-2\eps_0}.
		\end{align*}
		By changing the variable, we obtain the same estimate for the function $u$. Thus, by the Morrey's embedding with exponent $\alpha= 1-\eps_0 = 1/2$ (see \cite{morrey}), we get the claim.
	\end{proof}
	
	\subsection{The set $\{u>0\}$}\label{sub:open}
	As an immediate consequence of \cref{l:holder0} we obtain that if
	$$u:B_1\cap\{x_d\ge0\}\to\R$$
	is a local minimizer of $\mathcal{F}$ in $B_1$, then the set $\{u>0\}$ is a relatively open subset of $B_1\cap\{x_d>0\}$.\\ In particular, the sets
	$$\{u>0\}\cap B_1^+\qquad\text{and}\qquad \{u>0\}\cap B_1'$$
	are open subset respectively of $\R^d$ and $\{x_d=0\}$. Moreover, $u$ is a weak solution of the equation
	\begin{equation}\label{equation-in-u>0}
	\begin{cases}
	\mathrm{div}(A\nabla u) =0 & \mbox{in }\ B_1^+\cap \{u>0\}\,,\\
	A \nabla u \cdot e_d = \beta & \mbox{on }\ B_1'\cap\{u>0\}\,,
	\end{cases}
	\end{equation}
	that is, for every $C^\infty$ function $\varphi:\overline B_1\cap\{x_d\ge 0\}\to\R$ with support contained in the (relatively) open set $\{u>0\}\cap B_1$, 	we have:
	$$\int_{B_1^+}\nabla \varphi\cdot A\nabla u\,dx+\int_{B_1'\cap\{u>0\}}\beta(x')\varphi(x',0)\,dx'=0\,.$$
	
	\begin{remark}\label{oss:equation-in-u>0}
		Notice that \eqref{equation-in-u>0} can be written also as	
		$$
		\mathrm{div}(A\nabla u ) = {\beta} \ind_{\{x_d=0\}} \quad\mbox{in}\quad \{u>0\}\cap B_1\cap\{x_d\ge 0\}.
		$$
	\end{remark}

	\subsection{Generalized Laplacian estimate}
	We first notice that since $u$ is positive, the equation from \cref{oss:equation-in-u>0} becomes an inequality on $B_1\cap\{x_d\ge 0\}$. Precisely, we have the following lemma.
	
	\begin{lemma}[Definition of $\text{\rm div}(A\nabla u)$]\label{l:radon-measure-divergence-A}
		Let $E\subset\R^d$ be an open set and let
		$$u:E\cap\{x_d\ge0\}\to\R\ ,$$
		be a local minimizer of $\mathcal{F}$ in $E$. Then,
		$$\mathrm{\rm div}(A\nabla u)- \beta\ind_{\{x_d=0\}\cap\{u>0\}}\HH^{d-1}\ge 0 \quad\mbox{in}\quad E\cap \{x_d\geq0\},$$
		in sense of distributions, that is,
		\begin{equation}\label{e:lemma1-eq1}
		-\int_{E_+} \nabla\varphi\cdot A(x)\nabla u\,dx-  \int_{E'} \beta\varphi\ind_{\{u>0\}}\,dx'\ge 0\ ,
		\end{equation}
		for every non-negative function $\varphi\in C^\infty_c(E)$. In particular, there is a Radon measure
		$\mu$ on the set
		$$E\cap\{x_d\ge 0\}=E_+\cup E'$$
		such that
		$$\int_{E_+\cup E'}\varphi\,d\mu=	-\int_{E_+} \nabla\varphi\cdot A(x)\nabla u\,dx-  \int_{E'} \beta\varphi\ind_{\{u>0\}}\,dx'$$
		As a consequence, $\mathrm{\rm div}(A\nabla u)$ is also a Radon measure on $E\cap\{x_d\ge 0\}$, defined as
		$$\int_{E_+\cup E'}\varphi\,d(\mathrm{\rm div}(A\nabla u))=\int_{E_+\cup E'}\varphi\,d\mu+\int_{E'} \beta\varphi\ind_{\{u>0\}}\,dx'=-\int_{E_+} \nabla\varphi\cdot A(x)\nabla u\,dx.$$
	\end{lemma}
	\begin{proof}
		Consider the function
		\begin{equation}\label{e:p-epsilon}
		p_\eps:\R\to\R\ ,\qquad p_\eps(t)=\begin{cases}
		0\quad \text{if}\quad t\le \eps\ ,\\
		\frac1{\eps}(t-\eps)\quad \text{if}\quad \eps\le t\le 2\eps\ ,\\
		1\quad \text{if}\quad t\ge 2\eps\ .
		\end{cases}
		\end{equation}
		Let $\varphi\in C^\infty_c(E)$. For any $t\in\R$ we define the competitor $u-tp_\eps(u)\varphi$ and we notice that if $|t|$ is small enough, then $u-tp_\eps(u)\varphi\ge 0$ in $E$ and $\{u-tp_\eps(u)\varphi>0\}=\{u>0\}$. Thus, the minimality condition
		$$\mathcal{F}(u,E)\leq \mathcal{F}(u-tp_\eps(u)\varphi,E),$$
		for every $t$ (with $|t|$ small enough), we get
		\begin{align*}
		0&= -\int_{E_+}\nabla \big(p_\eps(u)\varphi\big)\cdot A\nabla u\,dx-\int_{E'}\beta\,p_\eps(u)\varphi\,dx'\\
		&= -\int_{E_+}\varphi p_\eps'(u)\nabla u\cdot A\nabla u\,dx-\int_{E_+} p_\eps(u)\nabla\varphi\cdot A\nabla u\,dx-\int_{E'}\beta\,p_\eps(u)\varphi\,dx'\ge0.
		\end{align*}
		Since $\nabla u\cdot A\nabla u\ge 0$, $\varphi\ge 0$ and $p_\eps'\ge 0$, we get that
		\begin{align*}
		-\int_{E_+} p_\eps(u)\nabla\varphi\cdot A\nabla u\,dx-\int_{E'}\beta\,p_\eps(u)\varphi\,dx'\ge0.
		\end{align*}
		Sending $\eps$ to $0$, we get the claim.
	\end{proof}

	\begin{lemma}[Generalized Laplacian estimate for balls centered on the hyperplane]\label{l:lap.est1}
		Let $u$ be a local minimizer of $\mathcal{F}$ in $B_1$. Then,
		for every $x_0 \in B_1'$ and every $r \in (0,1)$ such that $B_{2r}(x_0)\subset B_1$, we have
		\be\label{laplacian.est}
		|\mathrm{div}(A\nabla u)|(B_r^+(x_0))\leq C r^{d-1},
		\ee
		with $C$ depending only on $d,\Lambda_A,\norm{Q}{L^\infty}$ and $\norm{\beta}{L^\infty}$.
	\end{lemma}
	\begin{proof}
		Recall that, by the definition of $\mathrm{div}(A\nabla u)$, we have
		$$
		\mathrm{div}(A\nabla u)(\varphi) = -\int_{B_1^+} \nabla \varphi \cdot A \nabla u \,\mathrm{d}x \quad\mbox{for }\varphi \in C^\infty_c(B_1).
		$$
		In the rest of the proof, for simplicity, we suppose that $x_0=0$. Given $\psi \in C^\infty_c(B_r)$ the minimality condition $\mathcal{F}(u,B_1)\leq \mathcal{F}((u+\psi)_+,B_1)$ implies
		$$
		0\leq \int_{B_r^+}\nabla \psi \cdot A\nabla \psi\,\mathrm{d}x + 2\int_{B_r^+}\nabla \psi \cdot A\nabla u\,\mathrm{d}x + \frac12\omega_d \norm{Q}{L^\infty}r^d + 2\int_{B_r'}|\beta|\, |\psi|\, \mathrm{d}x',
		$$
		which leads to
		$$
		-\int_{B_r^+}\nabla \psi\cdot A\nabla u\,\mathrm{d}x\leq \frac12 \Lambda_A^2 \norm{\nabla \psi}{L^2(B_r^+)}^2+ \frac14{\omega_d}\norm{Q}{L^\infty}r^d + C_d\norm{\beta}{L^\infty}r^{\sfrac{d}2}\norm{\nabla \psi}{L^2(B_r^+)}.
		$$
		Finally, given $\varphi \in C^\infty_c(B_r)$ let us consider $\psi = r^{\sfrac{d}2}\norm{\nabla \varphi}{L^2(B_r^+)}^{-1}\varphi$, so that $\norm{\nabla\psi}{L^2(B_r^+)}=r^{\sfrac{d}2}$. Then
		$$
		-\int_{B_r^+}\nabla \varphi \cdot A\nabla u\,\mathrm{d}x\leq \left(\frac12 \Lambda_A^2 + \frac14{\omega_d} \norm{Q}{L^\infty} +C_d\norm{\beta}{L^\infty}\right)r^{\sfrac{d}2}\norm{\nabla \varphi}{L^2(B_r^+)},
		$$
		and since we have the same argument with $-\varphi$ in place of $\varphi$, we get that
		$$
		\left|\int_{B_r^+}\nabla \varphi \cdot A\nabla u\,\mathrm{d}x\right|\leq \left(\frac12 \Lambda_A^2 + \frac14{\omega_d} \norm{Q}{L^\infty} +C_d\norm{\beta}{L^\infty}\right)r^{\sfrac{d}2}\norm{\nabla \varphi}{L^2(B_r^+)},
		$$
		Let now $\varphi \in C^\infty_c(B_{2r})$ be such that
		$$
		0\le\varphi \le 1\,\text{ in $B_{2r}$},\quad
		\varphi \equiv 1\,\text{ in $B_{r}$}\quad\mbox{and}\quad\norm{\nabla \varphi}{L^\infty(B_{2r})}\leq \frac{2}{r},
		$$
		and let
		$$\mu:=\mathrm{div}(A\nabla u)+\beta\ind_{\{x_d=0\}}\HH^{d-1}.$$
		Then, the total variation $|\mathrm{div}(A\nabla u)|$ can be estimated as follows:
		\begin{align*}
		|\mathrm{div}(A\nabla u)|=\big|\mu+\beta\ind_{\{x_d=0\}}\HH^{d-1}\big|&\le \mu +|\beta|\ind_{\{x_d=0\}}\HH^{d-1}\\
		&=\mathrm{div}(A\nabla u) +(|\beta|-\beta)\ind_{\{x_d=0\}}\HH^{d-1}\\
		&=\mathrm{div}(A\nabla u) +2\beta_-\ind_{\{x_d=0\}}\HH^{d-1}\\
		&\le\mathrm{div}(A\nabla u) +2\|\beta\|_{L^\infty}\ind_{\{x_d=0\}}\HH^{d-1},
		\end{align*}
		so
		\begin{align*}
		|\mathrm{div}(A\nabla u)|(B_r^+)&\leq
		|\mathrm{div}(A\nabla u)|(\varphi)\\
		&\leq
		\Big(\mathrm{div}(A\nabla u) +2\|\beta\|_{L^\infty}\ind_{\{x_d=0\}}\HH^{d-1}\Big)(\varphi)\\
		&\leq
		\left(\frac12 \Lambda_A^2 + \frac14{\omega_d} \norm{Q}{L^\infty} +C_d\norm{\beta}{L^\infty}\right)(2r)^{\sfrac{d}2}\norm{\nabla \varphi}{L^2(B_{2r}^+)}+2\|\beta\|_{L^\infty}(2r)^{d-1}\\
		&\leq
		C\big(d,\Lambda_A,\norm{Q}{L^\infty},\norm{\beta}{L^\infty}\big) r^{d-1},
		\end{align*}
		as claimed.
	\end{proof}
	
	We next show that the Laplacian estimate holds also for balls centered in $B_1^+$.
	
	\begin{proposition}[Generalized Laplacian estimate]\label{p:laplacian-estimate-general}
		Let $u$ be a local minimizer of $\mathcal{F}$ in $B_1$. Then,
		for every $x_0 \in B_1^+\cup B_1'$ and every $r \in (0,1)$ such that $B_{8r}(x_0)\subset B_1$, we have
		\be\label{e:laplacian-estimate-general}
		|\mathrm{div}(A\nabla u)|(B_r^+(x_0))\leq C r^{d-1},
		\ee
		with $C$ depending only on $d,\Lambda_A,\norm{Q}{L^\infty}$ and $\norm{\beta}{L^\infty}$.
	\end{proposition}
	\begin{proof}
		For every $B_{8r}(x_0)\subset B_1$ with $x_0\in B_1^+\cup B_1'$, we have one of the following two possibilities:
		\begin{itemize}
			\item[$\text{Case 1.}$]  $B_{2r}(x_0)\subset B_1^+$;
			\item[$\text{Case 2.}$]  $B_{2r}(x_0)\subset B_{4r}(y_0)$, where $y_0$ is the projection of $x_0$ on $B_1'$.
		\end{itemize}	
		We notice that in the second case $B_{8r}(y_0)\subset B_1$, so by \cref{l:lap.est1} we get
		$$|\mathrm{div}(A\nabla u)|(B_{r}^+(x_0))\le |\mathrm{div}(A\nabla u)|(B_{4r}^+(y_0))\leq C (4r)^{d-1},$$
		while the estimate in the first case follows by the same argument we used in \cref{l:lap.est1} (the detailed proof can be also found in \cite[Section 3.2]{velectures}).
	\end{proof}
	
	\subsection{Growth estimates on the boundary} Before to provide the main regularity result, we need some growth estimate of the solutions in a neighborhood of the free boundary.
	\begin{lemma}[Growth estimate for the mean over balls centered in $B_1^+$]\label{l:iteration-interior}
		Let $u:B_1\cap\{x_d\ge0\}\to\R$ be a local minimizer of $\mathcal{F}$ in $B_1$, where $\mathcal F$ is as in \eqref{e:def-F} and where $A$, $Q$ and $\beta$ satisfy the conditions from \cref{sub:intro-assumptions}. Suppose that $x_0 \in B_{\sfrac12}^+$ is such that
		$$u(x_0)=0\qquad\text{and}\qquad A(x_0)=\text{\rm Id}\,.$$ Then, there are constants $C$ and $R\in(0,1)$ such that
		\begin{equation}\label{e:lip-main-iter-est-inside}
		\frac{1}{(2r)^{d}}\int_{B_{2r}^+(x_0)}{u}\,dx\le \Big(1+Cr^\sigma\Big)\frac1{r^{d}}\int_{B_{r}^+(x_0)}{u}\,dx+Cr,
		\end{equation}
		for every radius $r<\min\{R,4\delta\},$
		where $\delta$ is the distance from $x_0$ to the hyperplane $\{x_d=0\}$.
	\end{lemma}	
	\begin{proof}
		We fix a radius $r<\min\{R,2\delta\}$ and set
		$$h:\overline B_{2r}(x_0)\setminus B_r(x_0)\to\R,$$
		to be the solution to the problem
		$$\begin{cases}
		\begin{aligned}\text{\rm div}(A\nabla h)=0&\quad\text{in}\quad B_{2r}(x_0)\setminus \overline B_r(x_0)\,,\\
		h=r^{2-d}&\quad\text{on}\quad \partial B_{r}(x_0)\,,\\
		h=(2r)^{2-d}&\quad\text{on}\quad \partial B_{2r}(x_0)\,.
		\end{aligned}
		\end{cases}$$
		We consider the rescaling
		$$w:\overline B_2\setminus B_1\to\R\ ,\qquad  w(x)=r^{d-2}h(x_0+rx),$$
		and we notice that $w$ solves
		$$\text{\rm div}(A_r\nabla w)=0\quad\text{in}\quad B_{2}\setminus \overline B_1\ ,\qquad w=1\quad\text{on}\quad \partial B_{r}\ ,\qquad w=2^{2-d}\quad\text{on}\quad \partial B_{2}\ ,$$
		where $A_r(x):=A(rx)$. Since $A$ satisfies the hypotheses $(\mathcal H1)$, there are constants $C$ and $\alpha$ such that
		\begin{equation}\label{e:lip-estimate-A-r}
		\frac1{1+Cr^\alpha}\text{\rm Id}\le A_r(x)\le (1+Cr^\alpha)\text{\rm Id}\qquad\text{for every}\qquad x\in\overline B_2\setminus B_1.
		\end{equation}
		Thus, for $r$ small enough, $A_r$ satisfies the hypotheses of \cref{t:app-schauder}, so there are constants $K$ and $\sigma$ (depending only on the dimension and the constants involved in $(\mathcal H1)$), for which
		$$\frac1{1+Kr^\sigma}\frac{d-2}{|x|^{d-1}}\le |\nabla w(x)|\le (1+Kr^\sigma)\frac{d-2}{|x|^{d-1}}\quad\text{for every}\quad x\in\overline B_2\setminus B_1,$$
		where $\frac{d-2}{|x|^{d-1}}$ is the norm of the gradient of the harmonic function in $\overline B_2\setminus B_1$  with the same boundary datum as $w$ on $\partial B_2\cup\partial B_1$. Rescaling everything back to scale $r$, we get
		\begin{equation}\label{e:lip-estimate-gradient-h}
		\frac1{1+Kr^\sigma}\frac{d-2}{|y|^{d-1}}\le |\nabla h(x_0+y)|\le (1+Kr^\sigma)\frac{d-2}{|y|^{d-1}}\quad\text{for every}\quad y\in\overline B_{2r}\setminus B_r\,.
		\end{equation}
		We set $t:=(2r)^{2-d}$ and $T:=r^{2-d}$, and we notice that the above estimate on $|\nabla h|$ implies that the level sets $\{h=s\}$ are $C^{1,\alpha}$ manifolds for every $s\in(t,T)$. We will use the notation
		$$\Omega_s:=\{h>s\}\cap\{x_d>0\}\,,\quad \Gamma_s:=\{h=s\}\cap\{x_d>0\} \quad\text{and}\quad\Omega_s':=\{h>s\}\cap\{x_d=0\}\,,$$
		for every $s\in[t,T]$. Notice that by construction
		$$\Omega_t=B_{2r}^+(x_0)\qquad\text{and}\qquad \Omega_T=B_{r}^+(x_0)\,.$$
		Moreover, for any $s\in(t,T)$, the normal to $\Gamma_s$ pointing outwards $\Omega_s$ is precisely $\frac{-\nabla h}{|\nabla h|}$, while the normal to $\Omega_s'$ pointing outwards $\Omega_s$ is $-e_d$.
		By integrating by parts, applying the co-area formula and integrating by parts again, we get that
		\begin{align*}
		\int_t^T\bigg(\int_{\Omega_s^+}\text{\rm div}(A\nabla u)\,dx\bigg)\,ds&=\int_t^T\bigg(\int_{\Gamma_s}\nabla u\cdot A\frac{-\nabla h}{|\nabla h|}\,d\mathcal{H}^{d-1}+\int_{\Omega_s'}\nabla u\cdot A(-e_d)\,dx'\bigg)\,ds\\
		&=\int_{\Omega_t\setminus\Omega_T}\nabla u\cdot A(-\nabla h) \,dx+\int_t^T\bigg(\int_{\Omega_s'}\nabla u\cdot A(-e_d)\,dx'\bigg)\,ds\\
		&=\int_{\Gamma_T} u\frac{\nabla h}{|\nabla h|}\cdot A(-\nabla h)\,d\mathcal{H}^{d-1}+\int_{\Gamma_t} u\frac{-\nabla h}{|\nabla h|}\cdot A(-\nabla h)\\
		&\qquad+\int_{\Omega_t'\setminus\Omega_T'} u(-e_d)\cdot A(-\nabla h)\,dx'+\int_t^T\bigg(\int_{\Omega_s'}\nabla u\cdot A(-e_d)\,dx'\bigg)\,ds\,.
		\end{align*}	
		Now, in $B_R(x_0)$, $A$ is close to the identity matrix and (by \cref{t:app-schauder}) $h$ is close to the harmonic function in the annulus $B_{2r}(x_0)\setminus B_r(x_0)$, we get that $(-e_d)\cdot A\nabla h>0$ on $\{x_d=0\}$. Thus
		\begin{align*}
		\int_{\partial B_{2r}(x_0)\cap\{x_d>0\}} u\frac{\nabla h\cdot A \nabla h}{|\nabla h|}\,d\mathcal{H}^{d-1}&\le\int_{\partial B_{r}(x_0)\cap\{x_d>0\}} u\frac{\nabla h\cdot A \nabla h}{|\nabla h|}\,d\mathcal{H}^{d-1}\\
		&\qquad +\int_t^T\bigg(\int_{\Omega_s^+}\text{\rm div}(A\nabla u)\,dx+\int_{\Omega_s'}\nabla u\cdot A e_d\,dx'\bigg)\,ds\\
		&=\int_{\partial B_{r}(x_0)\cap\{x_d>0\}} u\frac{\nabla h\cdot A \nabla h}{|\nabla h|}\,d\mathcal{H}^{d-1}\\
		&\qquad +\int_t^T\bigg(\int_{\Omega_s^+}\text{\rm div}(A\nabla u)\,dx+\int_{\Omega_s'}\beta\ind_{\{u>0\}}\,dx'\bigg)\,ds\,,
		%	&\le\int_{\partial B_{r}} \varphi\frac{\nabla h\cdot A \nabla h}{|\nabla h|}+(T-t)\int_{B_{2r}}|\text{\rm div}(A\nabla \varphi)|\,dx\,.
		\end{align*}	
		Notice that $\Omega_s\subset B_{2r}(x_0)$, for every $s$. Thus, by \eqref{laplacian.est} and the choice $T=r^{2-d}$, we get that
		\begin{align*}
		\int_{\partial B_{2r}(x_0)\cap\{x_d>0\}} u\frac{\nabla h\cdot A \nabla h}{|\nabla h|}\,d\mathcal{H}^{d-1}\le \int_{\partial B_{r}(x_0)\cap\{x_d>0\}} u\frac{\nabla h\cdot A \nabla h}{|\nabla h|}\,d\mathcal{H}^{d-1}+Cr,
		\end{align*}
		for some $C>0$.	Finally, using the estimates \eqref{e:lip-estimate-gradient-h} and \eqref{e:lip-estimate-A-r}, we get
		\begin{equation*}%\label{e:lip-main-iter-est}
		\frac{1}{(2r)^{d-1}}\int_{\partial B_{2r}(x_0)\cap\{x_d>0\}}{u}\,d\mathcal{H}^{d-1}\le \Big(1+Cr^\sigma\Big)\frac1{r^{d-1}}\int_{\partial B_{r}(x_0)\cap\{x_d>0\}}{u}\,d\mathcal{H}^{d-1}+Cr,
		\end{equation*}
		which by the area formula gives the claim.
	\end{proof}

	\begin{lemma}[Growth estimate for the mean over balls centered on $B_1'$]\label{l:iteration-boundary}
		Let $u:B_1\cap\{x_d\ge0\}\to\R$ be a local minimizer of $\mathcal{F}$ in $B_1$, where $\mathcal F$ is as in \eqref{e:def-F} and where $A$, $Q$ and $\beta$ satisfy the conditions from \cref{sub:intro-assumptions}. Suppose that $x_0 \in B_{\sfrac12}'$ is such that
		$$u(x_0)=0\qquad\text{and}\qquad A(x_0)=\text{\rm Id}\,.$$ Then, there are constants $C$ and $R\in(0,1)$ such that
		\begin{equation}\label{e:lip-main-iter-est-boundary}
		\frac{1}{(2r)^{d}}\int_{B_{2r}^+(x_0)}{u}\,dx\le \Big(1+Cr^\sigma\Big)\frac1{r^{d}}\int_{B_{r}^+(x_0)}{u}\,dx+Cr,
		\end{equation}
		for every radius $r<R$.
	\end{lemma}	
	\begin{proof}
		Suppose for simplicity that $x_0=0$. Fixed a radius $r>0$, we set
		$$h:\overline B_{2r}^+\setminus B_r^+\to\R,$$
		to be the solution to the problem
		$$\begin{cases}
		\begin{aligned}
		\text{\rm div}(A\nabla h)=0&\quad\text{in}\quad B_{2r}^+\setminus B_r^+\,,\\
		h=r^{2-d}&\quad\text{on}\quad \partial B_{r}\cap\{x_d>0\}\,,\\
		h=(2r)^{2-d}&\quad\text{on}\quad \partial B_{2r}\cap\{x_d>0\}\,,\\
		e_d\cdot A\nabla h=0&\quad\text{on}\quad (B_{2r}\setminus B_r)\cap\{x_d=0\}\,.
		\end{aligned}\end{cases}$$
		Reasoning as in \cref{l:iteration-interior} and using \cref{t:app-schauder-half-ball}, we can choose $R>0$ small enough such that
		\begin{equation}\label{e:lip-estimate-gradient-h2}
		\frac1{1+Kr^\sigma}\frac{d-2}{|y|^{d-1}}\le |\nabla h(x_0+y)|\le (1+Kr^\sigma)\frac{d-2}{|y|^{d-1}}\quad\text{for every}\quad y\in\overline B_{2r}\setminus B_r\,,
		\end{equation}
		for every $r\le R$, where $K$ and $\sigma$ are universal constants depending on the constants from \cref{sub:intro-assumptions} and \cref{t:app-schauder-half-ball}.
		Setting $t:=(2r)^{2-d}$, $T:=r^{2-d}$ and
		$$\Omega_s:=\{h>s\}\cap\{x_d>0\}\,,\quad \Gamma_s:=\{h=s\}\cap\{x_d>0\} \quad\text{and}\quad\Omega_s':=\{h>s\}\cap\{x_d=0\}\,,$$
		for $s\in[t,T]$, we get that
		\begin{align*}
		\int_t^T\bigg(\int_{\Omega_s}\text{\rm div}(A\nabla u)\,dx+\int_{\Omega_s'}\nabla u\cdot Ae_d\,dx'\bigg)\,ds&=\int_t^T\bigg(\int_{\Gamma_s}\nabla u\cdot A\frac{-\nabla h}{|\nabla h|}\,d\mathcal{H}^{d-1}\bigg)\,ds\\
		&=\int_{\Omega_t\setminus\Omega_T}\nabla u\cdot A(-\nabla h) \,dx\\
		&=\int_{\Gamma_T} u\frac{\nabla h}{|\nabla h|}\cdot A(-\nabla h)\,d\mathcal{H}^{d-1}+\int_{\Gamma_t} u\frac{-\nabla h}{|\nabla h|}\cdot A(-\nabla h)\,d\mathcal{H}^{d-1},
		%&\qquad+\int_t^T\bigg(\int_{\Omega_s'}\nabla u\cdot A(-e_d)\bigg)\,ds\,,
		\end{align*}	
		which by \eqref{laplacian.est} gives
		\begin{align*}
		\int_{\partial B_{2r}\cap\{x_d>0\}} u\frac{\nabla h\cdot A \nabla h}{|\nabla h|}\,d\mathcal{H}^{d-1}\le \int_{\partial B_{r}\cap\{x_d>0\}} u\frac{\nabla h\cdot A \nabla h}{|\nabla h|}\,d\mathcal{H}^{d-1}+Cr.
		\end{align*}
		and implies that
		\begin{equation*}
		\frac{1}{(2r)^{d-1}}\int_{\partial B_{2r}\cap\{x_d>0\}}{u}\,d\mathcal{H}^{d-1}\le \Big(1+Cr^\sigma\Big)\frac1{r^{d-1}}\int_{\partial B_{r}\cap\{x_d>0\}}{u}\,d\mathcal{H}^{d-1}+Cr,
		\end{equation*}
		which by the area formula gives the claim.
	\end{proof}

	\begin{proposition}[Lipschitz estimate on the free boundary]\label{p:mean-value-estimate}
		Let the functional $\mathcal F$ be as in \eqref{e:def-F}, where $A$, $Q$ and $\beta$ satisfy the conditions from \cref{sub:intro-assumptions}. Then, there are constants $L$ and $R\in(0,1)$ such that for any function $u:B_1\cap\{x_d\ge0\}\to\R$ minimizing $\mathcal{F}$ in $B_1$ in the sense of \cref{d:minimizer}, we have
		$$\frac{1}{r^{d}}\int_{B_r^+(x_0)}u\,dx\le Lr\quad\text{for every}\quad r\in(0,R)\quad\text{and every}\quad x_0 \in \Big(B_{\sfrac12}^+\cup B_{\sfrac12}'\Big)\cap\{u=0\}\,.$$
	\end{proposition}	
	\begin{proof}
		Suppose first that $x_0\in B_{\sfrac12}^+$.	
		Up to a change of variables, we may suppose that $A(x_0)=\text{\rm Id}$. We next take $r\in(0,1)$ and we set
		$$r_n:=2^{-n}r\qquad\text{and}\qquad M_n=\frac1{r_n^d}\int_{B_{r_n}^+}{u(x)}\,dx\,.$$
		Let $\delta$ be the distance from $x_0$ to the hyperplane $\{x_d=0\}$ and let $N$ be the smallest natural number for which $r_N< 4\delta$.
		Then, for every $n\ge N$, we can apply \cref{l:iteration-interior}; the estimate \eqref{e:lip-main-iter-est-inside} reads as
		\begin{equation}\label{e:lip-iteration-M_n}
		M_{n}\le \Big(1+C2^{-n\sigma}\Big)M_{n+1}+Cr2^{-n}.
		\end{equation}
		Setting
		$$P_{N}=1\qquad\text{and}\qquad P_n:=\prod_{k=N}^{n-1}\Big(1+C2^{-k\sigma}\Big)\quad\text{for}\quad n\ge N+1\,,$$
		it is immediate to prove that
		\begin{equation}\label{e:lip-induction}
		M_N\le P_nM_n+P_{n-1}Cr\sum_{k=N}^{n-1}2^{-k}\quad\text{for every}\quad n\ge N+1.
		\end{equation}
		In fact, for $n=N+1$ the estimate is an immediate consequence of \eqref{e:lip-iteration-M_n}, while if \eqref{e:lip-induction} holds for some $n\ge N+1$, then \eqref{e:lip-iteration-M_n} implies
		\begin{align*}
		M_N&\le P_nM_n+P_{n-1}Cr\sum_{k=N}^{n-1}2^{-k}\\
		&\le P_n\Big(\Big(1+C2^{-n\sigma}\Big)M_{n+1}+Cr2^{-n}\Big)+P_{n-1}Cr\sum_{k=N}^{n-1}2^{-k}\\
		&\le P_{n+1}M_{n+1}+P_nCr\sum_{k=N}^{n}2^{-k},
		\end{align*}
		which concludes the proof of \eqref{e:lip-induction}. Finally, using the continuity of $u$ (\cref{l:holder0}),
		$$\lim_{n\to\infty}M_n=0\qquad\text{and}\qquad\lim_{n\to\infty}P_n\le \exp\bigg(C\sum_{k=0}^{+\infty}2^{-k\sigma}\bigg)=\exp\bigg(\frac{C}{1-2^{-\sigma}}\bigg)\,,$$
		we get that
		\begin{equation}\label{e:lipschitz-prop-estimate-for-M-n}
		M_N\le \exp\bigg(\frac{C}{1-2^{-\sigma}}\bigg)Cr_N\,.
		\end{equation}
		Now, if $r<4\delta$, then $N=0$ and \eqref{e:lipschitz-prop-estimate-for-M-n} concludes the proof. Conversely, if $r\ge 4\delta$, then we cannot make further use of the estimate \eqref{e:lip-main-iter-est-inside}.
		Notice that, at this point, we have
		$$2\delta\le r_N<4\delta\,.$$
		We will continue with the boundary estimate \eqref{e:lip-main-iter-est-boundary}. In order to do so, we consider the projection $y_0$ of $x_0$ on $\{x_d=0\}$. Without loss of generality, we can suppose that $y_0=0$. By the choice of $N$, we have
		$$B_{\sfrac{\sqrt{3}r_N}2}^+\subset B_{r_N}^+(x_0).$$
		Now, we consider the function
		$$\widetilde u(x)=u\Big(A^{\sfrac12}(0)\mathcal R x\Big),$$
		where $\mathcal R$ is a suitably chosen rotation of the coordinate system. Then, $\widetilde u$ is still a minimizer of $\mathcal F$ this time with $\widetilde A$, $\widetilde Q$ and $\widetilde \beta$ that still satisfy the assumptions from \cref{sub:intro-assumptions}, but this time with $\widetilde A(0)=\text{\rm Id}$. Now, by choosing the radius $R$ small enough, we can suppose that $A^{\sfrac12}(0)$ is arbitrarily close to $A^{\sfrac12}(x_0)$. Then, for some choice of $R$, we will have that $A^{\sfrac12}(0)\mathcal R\big[B_{r_{N}/2}^+(x_0)\big]\subset B_{\sfrac{\sqrt{3}r_N}2}^+$. Thus, using \eqref{e:lipschitz-prop-estimate-for-M-n}, we can find a constant $\widetilde C$ depending only on the dimension and the constants from \cref{sub:intro-assumptions}, such that
		\begin{align*}
		\frac{1}{r_{N+1}^d}\int_{B_{r_{N+1}}^+}\widetilde u(x)\,dx\le \widetilde C r_{N+1}\,.
		\end{align*}
		We now repeat once more the iteration argument from above. We set
		$$\widetilde M_n=\frac1{r_n^d}\int_{B_{r_n}^+}{\widetilde u(x)}\,dx\,,\qquad\widetilde P_{0}=1\qquad\text{and}\qquad \widetilde P_n:=\prod_{k=0}^{n-1}\Big(1+C2^{-k\sigma}\Big)\quad\text{for}\quad n\ge 1\,.$$
		Reasoning as above, we get that
		\begin{align*}
		\widetilde M_0\le \widetilde P_{N+1}\widetilde M_{N+1}+\widetilde P_NCr\sum_{k=0}^N2^{-k}
		&\le \exp\bigg(\frac{C}{1-2^{-\sigma}}\bigg)\Big(\widetilde M_{N+1}+Cr\sum_{k=0}^N2^{-k}\Big)\\
		&\le  \exp\bigg(\frac{C}{1-2^{-\sigma}}\bigg)\Big(\widetilde C r_{N+1}+2Cr\Big)\,,
		\end{align*}
		which proves that (recall that the constant $C$ is changing from line to line)
		$$\frac1{r^d}\int_{B_{r}^+}{\widetilde u(x)}\,dx\le Cr\,,$$
		for some $r>0$. Now, by the estimate from \cref{l:iteration-boundary}, we get that
		$$\frac1{(4r)^d}\int_{B_{4r}^+}{\widetilde u(x)}\,dx\le Cr\,.$$
		and changing back the coordinates to the orginal ones,
		$$\frac1{(2r)^d}\int_{B_{2r}^+} u(x)\,dx\le Cr\,.$$
		Finally, we notice that by the choice of $\delta$, the ball $B_{2r}^+$ contains $B_r^+(x_0)$. This concludes the proof in the case $x_0\in B_{\sfrac12}^+$. The case $x_0\in B_{\sfrac12}'$ is analogous.
	\end{proof}

	\subsection{Proof of \cref{t:lipschitz-continuity}}
	We will use the following gradient estimates for $A$-harmonic functions. The results are classical, but we give here the precise statements for the sake of completeness.
	\begin{lemma}[Gradient estimate for $A$-harmonic functions in a ball]\label{l:gradest-ball}
		Suppose that $A$ is a matrix with variable coefficients in $B_1$ and that $A$ satisfies the estimates from \cref{sub:intro-assumptions}.
		Suppose that $B_r(x_0)\subset B_1$ and that $u:B_r(x_0)\to\R$ is such that
		$$\text{\rm div}(A\nabla u)=0\quad\text{in}\quad B_r(x_0)\ ;\qquad \frac1r\int_{B_r(x_0)}|u(x)|\,dx\le L\ ,$$
		for some constant $L>0$. Then, there is a constant $C$ depending on $L$, the dimension $d$ and the constants from \cref{sub:intro-assumptions} such that $\|\nabla u\|_{L^\infty(B_{r/2}(x_0))}\le C$.
	\end{lemma}

	\begin{lemma}[Gradient estimate for $A$-harmonic functions in a half-ball]\label{l:gradest-half-ball}
		Suppose that $A$ is a matrix with variable coefficients in $B_1$, that $\beta$ is a H\"older continuous function on $B_1'$, and that $A$ and $\beta$ satisfy the estimates from \cref{sub:intro-assumptions}. Suppose, moreover, that
		$$\|\beta\|_{L^\infty(B_1')}\le M\,,$$
		for some constant $M>0$.
		Consider a ball $B_r(x_0)\subset B_1$ with $x_0\in B_1'$ and a function
		$$u:B_r^+(x_0)\cup B_r'(x_0)\to\R\,$$
		satisfying
		$$\text{\rm div}(A\nabla u)=0\quad\text{in}\quad B_r^+(x_0)\ ; \qquad e_d\cdot A\nabla u=\beta\quad\text{on}\quad B_r'(x_0)\ ;\qquad \frac1{r^d}\int_{B_r^+(x_0)}|u(x)|\,dx\le Lr\ ,$$
		for some constant $L>0$. Then, there is a constant $C$ depending on $L$, $M$, the dimension $d$ and the constants from \cref{sub:intro-assumptions} such that $\|\nabla u\|_{L^\infty(B_{r/2}^+(x_0))}\le C$.
	\end{lemma}

	\begin{proof}[\bf Proof of \cref{t:lipschitz-continuity}]
		Let	$x_0\in B_{\sfrac18}^+\cup B_{\sfrac18}'$. Let $y_0$ be the projection of $x_0$ on the closed set
		$$\Big(B_{1}^+\cup B_{1}'\Big)\cap\{u=0\}\,,$$
		and let $r:=|x_0-y_0|\,.$ We notice that \cref{p:mean-value-estimate} implies that
		\begin{equation}\label{e:proof-lip-cont-theorem-1}
		\frac1{(2r)^d}\int_{B_{2r}(y_0)}u(x)\,dx\le Cr,
		\end{equation}
		for some constant $C$ depending on the dimension and the constants from \cref{sub:intro-assumptions}.
		
		Let $z_0\in B_{\sfrac18}'$ be the projection of $x_0$ on the hyperplane $\{x_d=0\}$ and let
		$$\delta:=|x_0-z_0|.$$
		We consider two cases. In the case $\delta\ge r/8$, we have that $B_{r/8}(x_0)$ is entirely contained in $B_{2r}(y_0)$, in the half-ball $B_1^+$, and in the positivity set $\{u>0\}$. Thus, \eqref{e:proof-lip-cont-theorem-1} together with \cref{l:gradest-ball} allow to estimate $|\nabla u(x_0)|$. Conversely, if $\delta< r/8$, then we have that $x_0\in B_{\sfrac{r}{8}}^+(z_0)$ and $B_{\sfrac{r}{4}}^+(z_0)\subset B_r^+(x_0)\subset B_{2r}^+(y_0)$. In particular, since $B_r^+(x_0)$ is contained in the positivity set $\{u>0\}$, $u$ solves the PDE
		$$\text{\rm div}(A\nabla u)=0\quad\text{in}\quad B_{\sfrac{r}{4}}^+(z_0)\ ; \qquad e_d\cdot A\nabla u=\beta\quad\text{on}\quad B_{\sfrac{r}{4}}'(z_0)\ .$$
		Thus, by \cref{l:gradest-half-ball}, we get that $|\nabla u(x_0)|$ can be estimated by a constant depending on the dimension, the constants from \cref{sub:intro-assumptions} and the norm $\|\beta\|_{L^\infty(B_1')}$.
	\end{proof}

	\subsection{Improved almost-minimality}
	In this subsection, we show that by exploiting the Lipschitz continuity of local minimizer of $\mathcal{F}$, we can simplify the almost-minimality condition introduced in Proposition \ref{l:freezing_0} and write it in a form which is easier to handle along blow-up sequences.
	\begin{lemma}\label{l:freezing}
		Let $A$, $Q$ and $\beta$ be as in $(\mathcal H1)$, $(\mathcal H2)$ and $(\mathcal H3)$. Given a minimizer $u$ of $\mathcal{F}$ in $B_1$ and a point $x_0 \in B_1$, we consider the function $u^{x_0}$ from \eqref{e:def-of-u_x_0} and the functionals  $\mathcal{G}_{x_0}^+$ and $\mathcal{G}_{x_0}'$ from \eqref{e:def-G-inside} and \eqref{e:def-G-boundary}. Then, for every $r\in (0,1)$ such that $B_{\Lambda_A r}(x_0)\subset B_1$, we have
		\begin{align}\label{almost}
		\begin{aligned}
		\mathcal{G}_{x_0}^+(u^{x_0},B_r)+\mathcal{G}_{x_0}'(u^{x_0},B_r) &\leq\,\, \mathcal{G}_{x_0}^+(\tilde{u},B_r)+\mathcal{G}_{x_0}'(\tilde{u},B_r)\\
		&\qquad+ {C}_\beta r^{\min\{\delta_\beta,\delta_A\}}\norm{u^{x_0}-\tilde{u}}{L^1(B_r\cap H_{x_0}')}\\
		&\qquad\qquad+{C}_Q r^{d+\delta_Q}+ {C}_A r^{d+\delta_A}\norm{\nabla u}{L^\infty}^2,
		\end{aligned}
		\end{align}
		for every $\tilde{u}\in H^1(B_r^+\cap H^+_{x_0})$ such that $\tilde{u}=u^{x_0}$ on $\partial B_r\cap H^+_{x_0}$, where
		$$C_A =C(d,A), C_Q=C(A,d,Q), C_\beta=C(A,d,\beta).$$
	\end{lemma}
	\begin{proof}
		By the Lipschitz continuity of $u$, we have that
		\be\label{half.caccio}
		\mathcal{G}^+_{x_0}(u^{x_0},B_r) \leq \omega_d(\lambda^2\norm{\nabla u}{L^\infty}^2+Q(x_0))r^d\,,
		\ee
		so the claim follows from \cref{l:freezing_0}.
		\end{proof}

	\section{Non-degeneracy}\label{s:non-degeneracy}
	
	In this section we prove that the minimizers of the one-phase functional $\mathcal F$ are non-degenerate. In order to prove the main result of the section  (see \cref{p:non-degeneracy}), we will need the following lemma.
	
	\begin{lemma}\label{l:non-degeneracy}
		In any dimension $d\ge 2$, suppose that $\alpha>0$, $C>0$ and $\lambda>1$ are given constants and that $A=(a_{ij})_{i,j}$ is a real symmetric matrix with variable coefficients $a_{ij}: B_2\to\R$ satisfying
		\begin{enumerate}
			\item[($\mathcal A$1)] $\lambda^{-1}\text{\rm Id}\le A(x)\le \lambda\text{\rm Id}\ $ for every $\ x\in\overline D\,;$\medskip
			\item[($\mathcal A$2)]  for every couple of indices $1\le i,j\le d$, we have
			$$|a_{ij}(y)-a_{ij}(x)|\le C|x-y|^\alpha\qquad\text{for every}\qquad x,y\in\overline D\,,$$
		\end{enumerate}	
		Then, for every $\delta\in(0,1)$ there are $\rho\in(0,1)$ and $\eps>0$ such that the following holds.\\
		Suppose that $\varphi:\overline B_{1+\rho}\to\R$ is a solution to
		$$\text{\rm div}(A\nabla \varphi)=0\quad\text{in}\quad B_{1+\rho}\setminus \overline B_1\ ,\qquad \varphi=0\quad\text{on}\quad \partial  B_1\ ,\qquad \varphi=\rho\quad\text{on}\quad \partial  B_{1+\rho}\ ,$$
		and that
		$$\frac1{1+\eps}\text{\rm Id}\le A(x)\le (1+\eps)\text{\rm Id}\quad\text{for every}\quad x\in B_{2}\,.$$
		Then
		$$|\nabla \varphi|\le 1+\delta\quad\text{in}\quad B_{1+\rho}\qquad\text{and}\qquad |e_d\cdot A\nabla \varphi|\ge 1-\delta\quad\text{in}\quad B_{1+\rho}\cap\{x_d\le -1\}. $$	
	\end{lemma}
	\begin{proof}
		For the sake of simplicity we suppose that $d\ge 3$ and we consider the solution $h$ to the problem
		$$\Delta h=0\quad\text{in}\quad B_{1+\rho}\setminus \overline B_1\ ,\qquad h=0\quad\text{on}\quad \partial  B_1\ ,\qquad h=\rho\quad\text{on}\quad \partial  B_{1+\rho}\ ,$$
		that is
		$$h(x)=\frac{\rho}{1-(1+\rho)^{2-d}}\Big(1-|x|^{2-d}\Big)\quad\text{and}\quad \nabla h(x)=\frac{(d-2)\rho |x|^{1-d}}{1-(1+\rho)^{2-d}}\frac{x}{|x|}\,.$$
		Then, there is a dimensional constant $C_d$ such that, for $\rho$ small enough,
		$$1-C_d\rho\le |\nabla h|\le 1+C_d\rho\qquad\text{in}\qquad B_{1+\rho}\setminus \overline B_1\,.$$
		Moreover, since $\nabla h$ is radial,
		$$\frac1{1+\rho}|\nabla h|\le -e_d\cdot \nabla h\le |\nabla h|\qquad\text{in}\qquad B_{1+\rho}\cap \{x_d\le -1\}\,.$$
		Finally, we first choose $\rho$ and then $\eps$ small enough, and we apply \cref{t:app-schauder} to $h$ and $\varphi$.
	\end{proof}

	\begin{proposition}\label{p:non-degeneracy}
		Let $u$ be a local minimizer of $\mathcal{F}$ in $B_1$ in the sense of \cref{d:minimizer}, where $\mathcal F$ is defined in \eqref{e:def-F} with $A$, $Q$ and $\beta$ satisfying the hypotheses $(\mathcal H1)$, $(\mathcal H2)$ and $(\mathcal H3)$ from \cref{sub:intro-assumptions}. Suppose moreover that there is a constant $\delta>0$ such that
		\begin{equation}\label{e:gap-non-degeneracy}
		\beta+a\sqrt{Q}\ge \delta\quad\text{in}\quad B_1\,,
		\end{equation}
		where $a$ is given by \eqref{e.a}. Then, there is a constant $c>0$ depending on the dimension $d$, and on $\Lambda_A$, $Q$, $\beta$ and $\delta$, such that if
		if $B_{2r}(x_0)\subset B_1$, then the following implication holds:
		\begin{equation}\label{e:non-degeneracy}
		\norm{u}{L^\infty(B_{2r}(x_0)\cap\{x_d\ge 0\})}\leq c\,r\qquad \Rightarrow\qquad u\equiv 0 \ \text{ in }\  B_{r}(x_0)\cap\{x_d\ge 0\}.
		\end{equation}
	\end{proposition}
	\begin{proof}
			We first notice that it is sufficient to consider the following two cases: when $B_r(2x_0)$ is entirely contained in $\{x_d>0\}$ and when $x_0$ lies on the hyperplane $\{x_d=0\}$. Precisely,  \cref{p:non-degeneracy} is a consequence of the statements (a) and (b) below.
		\begin{enumerate}[\rm(a)]
			\item There is a constant $\eta_a>0$ depending on $d$, $\Lambda_A$ and $Q$ such that:
			\begin{equation}\label{e:non-degeneracy-a}
			\text{If $B_{2r}(x_0)\subset \{x_d>0\}$ and $\norm{u}{L^\infty(B_{2r}(x_0))}\leq \eta_a r$, then $u\equiv 0$ in $B_{r}(x_0)$.}
			\end{equation}	
			\item There is a constant $\eta_b>0$ depending on $d$, $\Lambda_A$, $Q$, $\beta$ and $\delta$ such that:
			\begin{equation}\label{e:non-degeneracy-b}
			\text{If $x_0\in\{x_d=0\}$ and $\norm{u}{L^\infty(B_{2r}^+(x_0))}\leq \eta_b r$, then $u\equiv 0$ in $B_{r}^+(x_0)$.}
			\end{equation}	
		\end{enumerate}		
		\noindent{\bf Proof of (a).} We consider the solutions $\varphi:B_{2r}(x_0)\setminus B_{r}(x_0) \to\R$ and $\psi:B_{2}\setminus B_{1}\to \R$ to
		$$\begin{cases}
		\mathrm{div}(A\nabla \varphi)=0\quad\text{in}\quad B_{2r}(x_0)\setminus B_{r}(x_0) \\
		\varphi=0 \quad\text{on}\quad \overline{B_{r}}(x_0)\\
		\varphi=r \quad\text{on}\quad \partial B_{2r}(x_0)
		\end{cases}\qquad\text{and}\qquad \begin{cases}
		\mathrm{div}(A_{r,x_0}\nabla \psi)=0\quad\text{in}\quad B_{2}\setminus B_{1} \\
		\psi=0 \quad\text{on}\quad \overline{B_{1}}\\
		\psi=1\quad\text{on}\quad \partial B_{2}\,,
		\end{cases}$$
		where $A_r(x):=A(x_0+rx)$. Now, by \eqref{e:lambda-A-Lambda-A}, we have that the matrix $A_{r,x_0}$ has H\"older continuos coefficients and satisfies
		$$\Lambda_A^{-1}\text{\rm Id}\le A_{r,x_0}(x)\le \Lambda_A\text{\rm Id}\qquad\text{for every}\qquad x\in \overline B_2.$$
		In particular, $\psi$ is $C^{1,\alpha}$ and satisfies the Lipschitz bound
		$$\|\nabla \psi\|_{L^\infty(\overline B_2\setminus B_1)}\le L\,,$$
		where $L$ depends only on the dimension and the constants from $(\mathcal H1)$. We notice that by construction $\psi(x)=\frac1r\varphi(x_0+rx)$, so $\varphi$ satisfies the uniform (in $x_0$ and $r$) bound
		$$\|\nabla \varphi\|_{L^\infty(\overline B_{2r}(x_0)\setminus B_{r}(x_0))}\le L\,.$$
		Now we proceed following step-by-step the argument from \cite{altcaf}. Precisely, using the optimality of $u$ and the fact that $B_{2r}(x_0)\subset \{x_d>0\}$, we get that
		$$\int_{B_{2r}(x_0)}\Big(\nabla u\cdot A(x)\nabla u+Q\ind_{\{u>0\}}\Big)\,dx\le \int_{B_{2r}(x_0)}\Big(\nabla (u\wedge (\eta_a\varphi))\cdot A(x)\nabla (u\wedge (\eta_a \varphi))+Q\ind_{\{u\wedge (\eta_a\varphi)>0\}}\Big)\,dx\,,$$
		which, since $\varphi=u\wedge(\eta_a\varphi)=0$ in $B_{r}(x_0)$, can be written as
		$$\int_{B_{r}(x_0)}\Big(\nabla u\cdot A(x)\nabla u+Q\ind_{\{u>0\}}\Big)\,dx\le \int_{B_{2r}(x_0)\setminus B_r(x_0)}\Big(\nabla (u\wedge (\eta_a\varphi))\cdot A\nabla (u\wedge(\eta_a\varphi))-\nabla u\cdot A\nabla u\Big)\,dx\,.$$
		Now, using the pointwise identity
		\begin{align*}
		\nabla (u\wedge \varphi)\cdot A\nabla (u\wedge\varphi)&=\nabla \big(u-(u-\varphi)_+\big)\cdot A\nabla \big(u-(u-\varphi)_+\big)\\
		&=\nabla u\cdot A\nabla u-2\nabla u\cdot A\nabla(u-\varphi)_++\nabla (u-\varphi)\cdot A\nabla(u-\varphi)_+\\
		&=\nabla u\cdot A\nabla u-\nabla (u-\varphi)\cdot A\nabla(u-\varphi)_+-2\nabla \varphi\cdot A\nabla(u-\varphi)_+\,,
		\end{align*}
		we obtain
		\begin{align*}
		\int_{B_{r}(x_0)}\Big(\nabla u\cdot A(x)\nabla u+Q\ind_{\{u>0\}}\Big)\,dx\le -2\int_{B_{2r}(x_0)}\eta_a\nabla \varphi\cdot A\nabla(u-\eta_a\varphi)_+\,dx\,.
		\end{align*}
		From the facts that $\varphi$ is $A$-harmonic, that $(u-\eta_a\varphi)_+=0$ on $\partial B_{2r}(x_0)$ and $(u-\eta_a\varphi)_+=u$ on $\partial B_r(x_0)$, and using again the bounds on $\nabla\varphi$ from \cref{t:app-schauder} (with $\rho=1$) we obtain
		\begin{align*}
		-\int_{B_{2r}(x_0)}\eta_a\nabla \varphi\cdot A\nabla(u-\eta_a\varphi)_+\,dx=\eta_a\int_{\partial B_r(x_0)}u(x)\,\frac{x-x_0}{r}\cdot A\nabla\varphi\,d\sigma\le \eta_a\,C_A\int_{\partial B_r(x_0)}u\,d\sigma\,
		\end{align*}
		where $C_A$ is a constant depending on $A$. Finally, just as in \cite{altcaf}, we use the trace inequality
		\begin{align*}
		\int_{\partial B_r(x_0)}u&\le C_d\bigg(\int_{B_r(x_0)}|\nabla u|\,dx+\frac1r \int_{B_r(x_0)}u\,dx\bigg)\\
		&\le C_d\frac{1}2\int_{B_r(x_0)}|\nabla u|^2\,dx+C_d\bigg(\frac12+\frac1r\|u\|_{L^\infty(B_{r}(x_0))}\bigg)|\{u>0\}\cap B_r(x_0)|\\
		&\le C_d\Lambda_A^{-1}\frac{1}2\int_{B_r(x_0)}\nabla u\cdot A\nabla u\,dx+C_d\bigg(\frac12+\eta_a\bigg)\frac1{\lambda_Q}\int_{B_r(x_0)}\ind_{\{u>0\}}Q(x)\,dx\\
		&\le C_d\Big(\Lambda_A^{-1}+(1+\eta_a)\lambda_Q^{-1}\Big)\int_{B_{r}(x_0)}\Big(\nabla u\cdot A(x)\nabla u+Q\ind_{\{u>0\}}\Big)\,dx,
		\end{align*}
		which finally implies that
		$$C_d\Big(\Lambda_A^{-1}+(1+\eta_a)\lambda_Q^{-1}\Big)\le \eta_a\,C_A\,,$$
		which is not possible if $\eta_a$ is small enough. This concludes the proof of \eqref{e:non-degeneracy-a}. \bigskip
		
		\noindent{\bf Proof of (b).}  Up to translation and rotation, we can suppose that
		$$x_0=0\quad\text{and}\quad A(0)=\text{\rm Id}\,.$$
		Moreover, by choosing $r$ small enough, we can suppose that the oscillation of $A$, $Q$ and $\beta$ is small on $B_{2r}$. In particular, we can assume that there is a constant $C>0$ such that
		\begin{equation}\label{e:choice-of-C}
		\sqrt{Q(x)-\frac{\beta(x)^2}{a(x)^2}\,}\le C-\frac{\delta}{3}\le C+\frac{\delta}{3}\le \sqrt{Q(x)\,}\quad\text{for every}\quad x\in B_{2r}\,,
		\end{equation}
		and that
		$$\quad\frac{1}{1+\eps}\text{\rm Id}\le A(x)\le (1+\eps)\text{\rm Id}\quad\text{for every}\quad x\in B_{2r}\,,$$
		where the constant $\eps$ will be chosen later and will depend not only on $A$, $\beta$ and $Q$, but also on the lower bound $\delta$ from \eqref{e:gap-non-degeneracy}. We will complete the proof in several steps. \medskip
		
		\noindent{\it Step 1.\,\it Clean-up in $B_{2r}^+$.} We notice that by   \eqref{e:non-degeneracy-a} we can clean up some space (where $u$ is identically zero) in $B_{2r}^+$.
		In fact, by choosing the constant $\eta_b$ (from \eqref{e:non-degeneracy-b}) to be small enough, we have that
		$$u\le \eta_a\frac{r}{20}\quad\text{on}\quad B_{\sfrac{r}{10}}(y_0),$$
		for every ball $B_{\sfrac{r}{10}}(y_0)$ contained in $B_{2r}^+$. Thus, by \eqref{e:non-degeneracy-a}, we have that
		$$u\equiv 0\quad\text{in}\quad B_{\sfrac{r}{20}}(y_0).$$
		In other words,
		$$\text{for every}\quad y_0\in Y:=\Big\{y\in B_{2r}^+\ :\ \mathrm{dist}\big(y,\partial( B_{2r}^+)\big)\geq \frac{3}{20} r\Big\}\,,$$
		we have that
		$$u\equiv 0\quad\text{in}\quad B_{\sfrac{r}{20}}(y_0).$$
		\begin{figure}[h]\label{fig:non-degeneracy}
			\begin{tikzpicture}[rotate=0, scale= 1.3]
			\draw[very thick, name path=1] (-2,0) -- (2,0);
			\draw[very thick, name path=2] (2,0) arc[start angle=0, end angle=180,radius=2cm];
			\draw[dashed, thick, name path=A] (25:1.7cm) arc[start angle=25, end angle=155,radius=1.7cm];
			\draw[draw=none, name path=a] (-1.8,0) -- (1.8,0);
			\draw[draw=none, name path=b] (-1.8,0.3) -- (1.8,0.3);
			\tikzfillbetween[of=a and b] {color=white};
			\draw[very thick] (-2,0) -- (2,0);
			\draw[dashed, thick, name path=B] (-1.325,0.3) -- (1.325,0.3);
			\tikzfillbetween[of=1 and 2] {pattern=north west lines};
			\filldraw[color=white] (1.265,0.6) circle (0.3cm);
			\filldraw[color=white] (-1.265,0.6) circle (0.3cm);
			\draw[dashed] (1.265,0.6) circle [radius=0.3cm];
			\draw[dashed] (-1.265,0.6) circle [radius=0.3cm];
			\tikzfillbetween[of=A and B] {color=white};
			\draw[thick] (0.5,0.6) circle [radius=0.2cm];
			\draw[thick] (0.5,0.6) circle [radius=0.1cm];
			\draw node at (-0.8,0.6) {$u\equiv 0$};
			\draw node at (0.5,1) {$\Omega_t$};
			\draw[very thick,->] (0.5,0.6) -- (0.5,0.1);
			\end{tikzpicture}
			\caption{Clean-up in $B_{2r}^+$. On the picture $\Omega_t:=B_{(1+\rho)\frac{r}{40}}(y_t)\setminus B_{\frac{r}{40}}(y_t)$.
			}
		\end{figure}
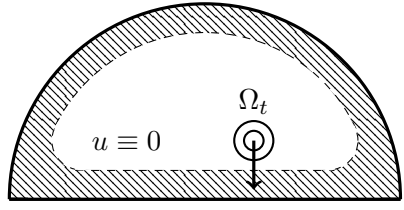
		Now, for every $y_0$ as above, and every $t>0$ we consider the point
		$$y_t:=y_0-te_d\,,$$
		and the solution $\varphi_t$ to
		$$\begin{cases}
		\mathrm{div}(A\nabla \varphi_t)=0\quad\text{in}\quad B_{(1+\rho)r/40}(y_t)\setminus B_{r/40}(y_t) \\
		\varphi_t=0 \quad\text{on}\quad \overline{B_{r/40}}(y_t)\\
		\varphi_t=C\rho r/40 \quad\text{on}\quad \partial B_{(1+\rho)r/40}(y_t),
		\end{cases}$$
		where $C$ is the constant from \eqref{e:choice-of-C}  and $\rho$ will be chosen below. By \cref{l:non-degeneracy}, if we choose  $\rho$ and $\eps$ to be small enough, then
		\begin{equation}\label{e:choice-of-C2}
		C-\frac{\delta}{4}\le |A(x)^{\sfrac12}\nabla \varphi_t|\le C+\frac{\delta}{4}\quad\text{for every}\quad x\in B_{2r}\,.
		\end{equation}
		Finally, we choose $\eta_b$ to be smaller than $C\rho/40$. Suppose now that $T$ is the largest possible value of the parameter $t$ for which $\varphi_t\ge u$. Then, there is a point $z_T$ at which $\varphi_T(z_T)=u(z_T)$. By \eqref{e:choice-of-C} and \eqref{e:choice-of-C2}, we can exclude the following possibilities:
		\begin{enumerate}
			\item[$\bullet$] $z_T\in\{x_d=0\}$ and $u(z_T)>0$\,;
			\item[$\bullet$] $z_T\in\{x_d>0\}$ and $z_T\in\partial\{u>0\}$\,.
		\end{enumerate}
		Moreover, by the maximum principle $z_T\notin\{x_d>0\}\cap\{u>0\}$. Thus, the only possibility is that $z_t\in \{x_d=0\}$ and $u(z_t)=0$. But then also $\varphi_T(z_T)=0$, so this means that $z_T$ is the projection of the starting point $x_0$ on the hyperplane $\{x_d=0\}$ and that $u\equiv0$ on the segment connecting $x_0$ to $z_T$. This concludes the proof.
	\end{proof}
	
	\subsection{Consequences of the non-degeneracy and the Lipschitz continuity}
	As in the classical one-phase case, \cref{t:lipschitz-continuity} and \cref{p:non-degeneracy} provide a lower density estimate for the positivity set.
	\begin{corollary}\label{cor:lower.H}
		Suppose that $A$, $Q$ and $\beta$ are as in \cref{sub:intro-assumptions}, and that there is a constant $\delta>0$ such that
		\begin{equation*}
		\beta+a\sqrt{Q}>0\quad\text{in}\quad B_1\,,
		\end{equation*}
		where $a$ is given by \eqref{e.a}.
		Let $u$ be a local minimizer of $\mathcal{F}$ in $B_1$ and assume $0\in \partial\{u>0\}$. Then, for every $r \in (0,\sfrac12)$, there exists $x_r \in B_r\cap \{x_d\geq 0\}$ such that
		$$
		B_{C_0r}^+(x_r) \subset B_r^+\cap \{u>0\}\quad\mbox{and}\quad B_{C_0r}'(x_r) \subset B_r'\cap \{u>0\}
		$$
		for some constant $C_0>0$ depending on $\delta,d$ and the constants from \cref{sub:intro-assumptions}. As a consequence,
		$$
		|B_r^+\cap \{u>0\}| \geq \eps_0 \omega_d r^d,
		$$
		where $\eps_0:=\frac12C_0^d$.
	\end{corollary}
	\begin{proof}
		By \cref{p:non-degeneracy}, for $r$ small enough there exists $x_r\in B_r\cap \{u>0\}\cap \{x_d\geq 0\}$ such that $u(x_r)\geq C r$. On the other one, since $u$ is Lipschitz continuous (\cref{t:lipschitz-continuity}), by setting
		$$
		C_0 = \min\left\{1,\frac{C}{\norm{\nabla u}{L^\infty}}\right\},
		$$
		we have that $u>0$ in $B_{C_0 r}(x_r)\cap \{x_d\geq 0\}$, which proves the claimed lower bound.
	\end{proof}
	
	\section{Blow-up sequences and blow-up limits}\label{s:blow-up}
	This section is dedicated to the convergence of the blow-up sequences and the analysis of the blow-up limits via a monotonicity formula.
	\subsection{Monotonicity formula}\label{s:weiss.sect}
	In this section we prove two Weiss-type monotonicity formulas for the boundary adjusted energy naturally associated to the functional $\mathcal F$. The first one is an almost-monotonicity formula for the function $u^{x_0}$, obtained from a minimizer $u$ of $\mathcal F$ by the change of coordinates from  \eqref{e:def-of-u_x_0}. The second one is a (proper) monotonicity formula that we will apply to (suitable rotations of) the blow-up limits of $u^{x_0}$. We will use both formulas in the classification of the blow-up limits of minimizers $u$ at free boundary points $x_0\in\partial\{u>0\}\cap \{x_d=0\}$.
	\begin{proposition}\label{p:weiss}
		Let $A$, $Q$ and $\beta$ be as in \cref{sub:intro-assumptions} and let $u$ be a local minimizer of $\mathcal{F}$ in $B_1$. Fixed a point $x_0 \in\partial\{u>0\}\cap \{x_d=0\}$, we consider the energy
		\be\label{weiss.expre}
		W_{x_0}(u^{x_0},r):= \frac{1}{r^d}\Big(\mathcal{G}_{x_0}^+(u^{x_0},B_r)+\mathcal{G}_{x_0}'(u^{x_0},B_r) \Big) - \frac{1}{r^{d+1}}\int_{\partial B_r\cap H^+_{x_0}} (u^{x_0})^2\mathrm{d}\HH^{d-1}\,,
		\ee
		where $u^{x_0}$ is the function from \eqref{e:def-of-u_x_0} and the functionals $\mathcal G_{x_0}^+$ and $\mathcal G_{x_0}'$ are given by \eqref{e:def-G-inside} and \eqref{e:def-G-boundary}. Then, there are constants $\tilde{C}$ and $\eps$, depending on $\lambda,A,d,Q$ and $\beta$, such that
		%for every $r \in (0,\mathrm{dist}(x_0,\partial B_1)/2)$ such that $B_{\lambda r}(x_0)\subset B_1$, we have
		$$
		\frac{d}{d r}W_{x_0}(u^{x_0},r) \geq  \frac{1}{r^{d}}\int_{\partial B_r\cap H_{x_0}^+}{\left( \partial_r u^{x_0}- \frac{1}{r} u^{x_0} \right)^2\mathrm{d}\HH^{d-1}} -\tilde{C}r^{\eps-1}.
		$$
		In particular, the limit $\displaystyle \lim_{r\to 0^+}W_{x_0}(u^{x_0},r)$ exists and is finite.
	\end{proposition}
	\begin{proof}
		The claim follows from \cref{l:freezing} and the Lipschitz continuity of $u$ as in the case of the classical one-phase functional. 	
	\end{proof}	
	
	\begin{proposition}\label{p:weiss2}
		Let $q$ and $m$ be given real constants and let $J$ be the functional from \eqref{e:def-J-intro}:
		\begin{equation*}
		J(\varphi,B_r):=\int_{B_r^+}|\nabla \varphi|^2\,\mathrm{d}x + q^2\big|\{\varphi>0\}\cap B_r^+\big|+
		2\,m\int_{B_r'}\,|\varphi| \,\mathrm{d}\HH^{d-1}\,,
		\end{equation*}
		defined for every ball $B_r$ and every function $\varphi\in H^1_{loc}(\{x_d\ge0\})$.
		Let
		$v:\{x_d\ge0\}\to\R$
		be a non-negative Lipschitz continuous global minimizer of $J$ (in the sense of \cref{sub:intro-one-homogeneous-global-minimizers}) and let $0\in\partial\{v>0\}$.
		Then, setting
		\begin{equation*}
		W(v,r):= \frac{1}{r^d}J(v,B_r) - \frac{1}{r^{d+1}}\int_{(\partial B_r)^+} v^2\mathrm{d}\HH^{d-1}\,,
		\end{equation*}
		we have:
		$$
		\frac{d}{d r}W(v,r) \geq  \frac{1}{r^{d}}\int_{(\partial B_r)^+}{\left( \partial_r v- \frac{1}{r} v \right)^2\mathrm{d}\HH^{d-1}}\qquad\text{for every}\qquad r>0.
		$$
		Moreover, if $r\mapsto W(v,r)$ is constant in $r$, then $v$ is one-homogeneous.
	\end{proposition}	
	\begin{proof}
		The proof is the same as in the case of the classical one-phase functional. We only point out that the energy $J$ is $d$-homogeneous. Precisely, if we set $v_\rho(x):=\frac1{\rho}v(x\rho)$, then
		$$J(v_{rs}, 1)=J(v_r,s)\quad\text{ for every }\quad r,s>0,$$
		and so, we also have $W(v_{rs}, 1)=W(v_r,s)$.
	\end{proof}

	\subsection{Compactness and convergence of blow-up sequences}\label{s:blow.analysis}
	We fix $x_0\in \partial\{u>0\}\cap B_1'$ and, for every $r>0$ small enough  such that $B_{r}(x_0)\subset B_1$, we define
	\be\label{blow.upseq1}
	u_{x_0,r}(x):=\frac1ru(x_0+rx).
	\ee
	Since our functional has variable coefficients, it is more convenient to work with blow-up sequences associated to the change of coordinate induced by the matrix $A$. We define
	\be\label{blow.upseq2}
	\widetilde{u}_{x_0,r}(x) =\frac1r u^{x_0}(rx)= u_{x_0,r}(A^{1/2}(x_0)(x)),
	\ee
	where $u^{x_0}$ is as in \eqref{e:def-of-u_x_0}. 
	In particular, by the Lipschitz continuity of $u$, if $\overline B_{\rho}(x_0)\subset B_1$ we get
	$$\|\nabla u_{x_0,r}\|_{L^\infty(B_{\rho/r})}=\|\nabla u\|_{L^\infty(B_\rho(x_0))}\le L\,,$$
	for some $L>0$, which leads to the compactness of the blow-up sequences. Thus, for every sequence $r_k\to0$, there are a subsequence $r_{n_k}\to0$ and a non-negative $L$-Lipschitz function $$u_0:\R^d\cap\{x_d\ge0\}\to\R\,$$
	such that $u_{x_0,r_{n_k}}$ converges to $u_0$ locally uniformly in $\R^d\cap\{x_d\ge0\}$. Ultimately, the same compactness result holds true for the sequence $\widetilde u_{x_0,r_k}$, where the limit $\widetilde u_0$ satisfies $\widetilde u_0(x) = u_0(A^{1/2}(x_0)(x))$.\\
	
	The main result of the section is the characterization of the blow-up limits in terms of $\widetilde{u}_0$, which turn out to be a one-homogeneous global minimizer of the functional $
	\mathcal{G}^+_{x_0}+\mathcal{G}_{x_0}'$ (see \eqref{e:def-G-inside} and \eqref{e:def-G-boundary}).
	Eventually, up to a rotation, will be not restrictive to consider $\widetilde{u}_0$ a global minimizer of $\mathcal{F}_{x_0}$.\\
	
	Since the strategy is a straightforward adaptation of well-known result for the Alt-Caffarelli functional, we simply sketch the main points of the proof.
	\begin{lemma}\label{l:scal}
		Let $A$, $Q$ and $\beta$ be as in $(\mathcal H1)$, $(\mathcal H2)$ and $(\mathcal H3)$. Given a minimizer $u$ of $\mathcal{F}$ in $B_1$ and a point $x_0 \in \partial\{u>0\}\cap  B_1'$, we consider the rescaled function $\widetilde{u}_{x_0,r}$ from \eqref{blow.upseq2} and the functional  and the functionals $\mathcal G_{x_0}^+$ and $\mathcal G_{x_0}'$ are given by \eqref{e:def-G-inside} and \eqref{e:def-G-boundary}. Then, for every $R>0$ such that $B_{\Lambda_A r R}(x_0)\subset B_1$, we have
		\begin{align}\label{almost.rescaled}
		\begin{aligned}
		\mathcal{G}_{x_0}^+(\widetilde{u}_{x_0,r},B_R)+\mathcal{G}_{x_0}'(\widetilde{u}_{x_0,r},B_R)
		&\leq\, \mathcal{G}_{x_0}^+(w,B_R)+\mathcal{G}_{x_0}'(w,B_R)\\
		&\qquad + C_Q R^{d+\delta_Q} r^{\delta_Q}+ C_A R^{d+\delta_A}\norm{\nabla u}{L^\infty}^2 r^{\delta_A}\\
		&\qquad+ C_\beta R^{\min\{\delta_\beta,\delta_A\}}\norm{w-u^{x_0}_{r}}{L^1(B_{R}\cap H'_{x_0})}r^{\min\{\delta_\beta,\delta_A\}},
		\end{aligned}
		\end{align}
		for every $w\in H^1(B_r^+\cap H^+_{x_0})$ such that $w=u^{x_0}_r$ on $\partial B_R\cap H^+_{x_0}$, where
		$$C_A =C(d,A), C_Q=C(A,d,Q), C_\beta=C(A,d,\beta).$$
	\end{lemma}
	\begin{proof}
		Since the problem is translation invariant on the hyperplane $H'_{x_0}$, let us assume $x_0=0$. Let $w \in H^{1}(\R^{d})$ be such that $w=u^{0}_{r}$ on $(\partial B_R)^+$, for some $R>0$ such that $B_{\Lambda_A rR}\subset B_1$.\\
		Consider now the functions $v= w - u^{0}_{r} $ and $v_{1/r}(x)=r v(x/r), w_{1/r}(x)= r w(x/r)$. We notice that
		$$
		\norm{v_{1/r}}{L^1(B_{rR}\cap H'_{x_0})}= r^{d}\norm{v}{L^1(B_R\cap H'_{x_0})},
		$$
		which allows to exploit the almost-minimality condition \eqref{almost} of $u^0$ with respect to $\tilde{u}:=w_{1/r} = u^0 + v_{1/r}$ in the ball $B_{rR}$. Therefore, by recalling \eqref{e:FeG}, we get
		\begin{align*}
		\mathcal{G}_{x_0}^+(u^{0}_{r},B_R) + \mathcal{G}_{x_0}^+(u^{0}_{r},B_R) =& \,\frac{1}{r^d}\left( \mathcal{G}_{x_0}^+(u^{0},B_{rR}) + \mathcal{G}_{x_0}^+(u^{0},B_{rR})\right)\\
		\leq&\, \frac{1}{r^d}\left( \mathcal{G}_{x_0}^+(w_{1/r},B_{rR}) + \mathcal{G}_{x_0}^+(w_{1/r},B_{rR})\right)\\
		&\quad + C_\beta (rR)^{\min\{\delta_\beta,\delta_A\}}\frac{\norm{w_{1/r}-u^{0}}{L^1(B_{rR}\cap H'_{x_0})}}{r^d}\\
		&\qquad+C_Q R^{d+\delta_Q} r^{\delta_Q}+ C_A R^{d+\delta_A}\norm{\nabla u}{L^\infty}^2 r^{\delta_A}\\
		\leq&\,\mathcal{G}_{x_0}^+(w,B_{R}) + \mathcal{G}_{x_0}^+(w,B_{R})\\
		&\quad+ C_\beta R^{\min\{\delta_\beta,\delta_A\}}\norm{w-u^{0}_{r}}{L^1(B_{R}\cap H'_{x_0})}r^{\min\{\delta_\beta,\delta_A\}}\\
		&\qquad+C_Q R^{d+\delta_Q} r^{\delta_Q}+ C_A R^{d+\delta_A}\norm{\nabla u}{L^\infty}^2 r^{\delta_A},
		\end{align*}
		as we claimed.
	\end{proof}
	\begin{proposition}\label{p:compact}
		Let $A$, $Q$ and $\beta$ be as in $(\mathcal H1)$, $(\mathcal H2)$ and $(\mathcal H3)$. Let $u$ be a local minimizer of $\mathcal{F}$ in $B_1$ and let $x_0 \in \partial\{u>0\}\cap B_1'$ be such that
		\begin{equation}\label{e:hypo-p-compact}\beta(x_0)+a(x_0)\sqrt{Q(x_0)}>0\,.
		\end{equation}
		Then, for every $R > 0$ and for every sequence $r_k \to 0^+$ the following properties hold (up to extracting a subsequence $r_{k_n}$):
		\begin{enumerate}
			\item $\widetilde{u}_{x_0,r_k}$ converges to a blow-up limit $\widetilde{u}_{0} \in H^{1}_{\loc}(\overline{H_{x_0}^+})\cap C^{0,1}_\loc(\overline{H_{x_0}^+})$, uniformly on $B_R\cap H^+_{x_0}$ and strongly in $H^1(B_R^+)$;
			\item the sequence $\ind_{\{\widetilde{u}_{x_0,r_k}>0\}}\to \ind_{\{\widetilde{u}_{0}>0\}}$ strongly in $L^1(B_R^+)$;
			\item the sequence of the closed sets $\overline{B_R^+\cap \{\widetilde{u}_{x_0,r_k}>0\}}$ and its complement in $H_{x_0}^+$, converge in the Hausdorff sense respectively to $\overline{B_R^+\cap \{\widetilde{u}_0>0\}}$ and $H_{x_0}^+ \setminus \overline{B_R^+\cap \{\widetilde{u}_0>0\}}$;
			\item the blow-up limit $\widetilde{u}_0$ is non-degenerate at zero, namely there exists a dimensional constant $c_0 > 0$ such that
			$$
			\sup_{B_r^+ } \widetilde{u}_0 \geq c_0 r\quad\text{for every $r>0$}.
			$$
		\end{enumerate}
	\end{proposition}
	\begin{proof}
		This compactness result for blow-up sequences relies on well known argument of the theory of almost-minimizer of Alt-Caffarelli type functionals (see \cite[Proposition 6.2]{velectures} and \cite[Section 6]{det}). We notice that the hypothesis \eqref{e:hypo-p-compact} is needed only in the proof of (3) and (4).
	\end{proof}

	Finally, by exploiting the H\"{o}lder regularity of $Q$ and $\beta$, we can conclude by showing that every blow-up limit is a global minimizer of a functional of the form \eqref{e:def-J-intro}.
	\begin{proposition}\label{p:limit}
		Let $A$, $Q$ and $\beta$ be as in $(\mathcal H1)$, $(\mathcal H2)$ and $(\mathcal H3)$. Let $u$ be a minimizer of $\mathcal{F}$ in $B_1$ and $x_0 \in \partial\{u>0\}\cap  B_1'$ be fixed. Then, up to a rotation, every blow-up limit $\widetilde{u}_0$ of $u$ at $x_0$ is a global minimizer of the functional $\mathcal{F}_{x_0}$. Indeed, for every $R>0$ we have 
		\begin{multline*}
		\int_{B_R^+}\left(|\nabla \widetilde{u}_0|^2  + {Q(x_0)}\ind_{\{\widetilde{u}_0>0\}}\right)\,\mathrm{d}x + \frac{\beta(x_0)}{a(x_0)}\int_{B_R'}\widetilde{u}_0\,\mathrm{d}x'\\
		\leq
		\int_{B_R^+}\left(|\nabla w|^2 +{Q(x_0)}\ind_{\{\widetilde{u}_0>0\}} \right)\,\mathrm{d}x + \frac{\beta(x_0)}{a(x_0)}\int_{B_R'}w\,\mathrm{d}x',
		\end{multline*}
		for every $w\in H^{1}_\loc(\R^{d}_+)$ such that $w=\widetilde{u}_0$ on $(\partial B_R)^+$.
	\end{proposition}
	\begin{proof}
		For the sake of simplicity, let us set $\widetilde{u}_k=\widetilde{u}_{x_0,r_k}$ as the blow-up sequence \eqref{blow.upseq2} centered at $x_0$ and $\widetilde{u}_0$ its limit. By \cref{l:scal}, for every $R>0$ it holds
		\begin{align}\label{portaqua}
		\begin{aligned}
		\mathcal{G}_{x_0}^+(\widetilde{u}_{k},B_R)+\mathcal{G}_{x_0}'(\widetilde{u}_{k},B_R)
		\leq& \, \mathcal{G}_{x_0}^+(\widetilde{u},B_R)+\mathcal{G}_{x_0}'(\widetilde{u},B_R)
		\\
		&\quad + \widetilde{C}_\beta R^{\min\{\delta_\beta,\delta_A\}}\norm{\widetilde{u}-\widetilde{u}_{k}}{L^1(B_{R}\cap H'_{x_0})}r_k^{\min\{\delta_\beta,\delta_A\}}\\
		&\qquad+\widetilde{C}_Q R^{d+\delta_Q} r_k^{\delta_Q}+ \widetilde{C}_A R^{d+\delta_A}\norm{\nabla u}{L^\infty}^2 r_k^{\delta_A},
		\end{aligned}
		\end{align}
		for every $\widetilde{u} \in H^{1}(H^+_{x_0})$ such that $\widetilde{u}= \widetilde{u}_k$ on $(\partial B_R)^+$.\\
		Let now $w\in H^{1}_\loc(\overline{H^+_{x_0}})\cap L^\infty_\loc(\overline{H^+_{x_0}})$ be such that $w=u^\infty$ on $(\partial B_R)^+$ and let $\eta \in C^\infty_c(B_R)$ be such that $0\leq \eta \leq 1$. Thus, consider the test function
		$$
		w_k = w + (1-\eta)(\widetilde{u}_{k}-\widetilde{u}_0).
		$$
		Since $w = \widetilde{u}_0$ outside of $B_R^+$, we get $w_{k}=\widetilde{u}_{k}$ outside of $B_R^+$. Moreover, since
		$$
		w_{k} - \widetilde{u}_{k}= w  - \widetilde{u}_0 -\eta(\widetilde{u}_{k}-\widetilde{u}_0).
		$$
		and $\widetilde{u}_{k}\to \widetilde{u}_0$ in $L^1(B_R^+)$, there exists $k_0=k_0(d)>0$ such that
		$$
		\frac{|\beta(x_0)|}{|a(x_0)|}\int_{B_R}|\widetilde{u}_k - w_k|\ind_{H'_{x_0}}\mathrm{d}x \leq
		C \norm{\beta}{L^\infty}\norm{\widetilde{u}_0-w}{L^1(B_R)},
		$$
		for $k\geq k_0$ and $C$ depending on the quantities in $(\mathcal H1)$. Hence, by testing \eqref{portaqua} with respect to $w_{k}$ we deduce
		\begin{align*}
		\int_{B_R^+}|\nabla \widetilde{u}_{k}|^2 \,\mathrm{d}x\, +\, & \, Q(x_0)|\{\widetilde{u}_{k}>0\}\cap B_R^+| +  \frac{\beta(x_0)}{a(x_0)}\int_{B_R\cap H'_{x_0}}\widetilde{u}_{k}\,\mathrm{d}x'\\
		\leq&\, \int_{B_R^+}|\nabla w_k|^2 \,\mathrm{d}x + Q(x_0)|\{\widetilde{u}_0>0\}\cap \{\eta =1\}^+| + \frac{\beta(x_0)}{a(x_0)}\int_{B_R\cap H'_{x_0}}w_{k}\,\mathrm{d}x'\\
		&\quad +Q(x_0)|\{0<\eta<1\}^+|+ \widetilde{C}_\beta R^{\min\{\delta_\beta,\delta_A\}}\norm{\widetilde{u}_{k}-w_{k}}{L^1(B_{R}\cap H^{x_k})}r_k^{\min\{\delta_\beta,\delta_A\}}\\
		&\qquad+\widetilde{C}_Q R^{d+\delta_Q} r_k^{\delta_Q}+ \widetilde{C}_A R^{d+\delta_A}\norm{\nabla u}{L^\infty}^2 r_k^{\delta_A}
		\end{align*}
		Finally, since $w_{k}\to w, \widetilde{u}_{k}\to \widetilde{u}_0$ strongly in $H^{1}(B_R^+)$ we get
		\begin{align*}
		\int_{B_R^+}|\nabla \widetilde{u}_0|^2 \,\mathrm{d}x + \,&\, Q(x_0)|\{\widetilde{u}_0>0\}\cap B_R^+| + \frac{\beta(x_0)}{a(x_0)}\int_{B_R\cap H'_{x_0}}\widetilde{u}_0\,\mathrm{d}x'\\
		\leq &
		\int_{B_R^+}|\nabla w|^2 \,\mathrm{d}x + Q(x_0)|\{w>0\}\cap B_R^+|\\
		&\quad + \frac{\beta(x_0)}{a(x_0)}\int_{B_R\cap H'_{x_0}}w\,\mathrm{d}x'+Q(x_0)|\{0<\eta<1\}^+|.
		\end{align*}
		The result follows by choosing $\eta$ such that $|\{\eta =1\}|$ is arbitrarily close to $|B_R|$ and by rotating the coordinate in order to replace $H'_{x_0}$ with $\{x_d=0\}$.
	\end{proof}

	The following is a straightforward application of the Weiss monotonicity formula to the previous characterization of the blow-up limits.
	\begin{corollary}\label{cor:weiss.cor}
		Let $A$, $Q$ and $\beta$ be as in $(\mathcal H1)$, $(\mathcal H2)$ and $(\mathcal H3)$. Let $u$ be a local minimizer of $\mathcal{F}$ in $B_1$ and let $x_0 \in \partial\{u>0\}\cap B_1'$ be such that
		\begin{equation*}
		\beta(x_0)+a(x_0)\sqrt{Q(x_0)}>0\,.
		\end{equation*}
		Then, every blow-up limit  of $u$ at $x_0$ is one-homogeneous.
	\end{corollary}
	\begin{proof}
		By \cref{p:compact} and \cref{p:limit}  we already know that, up to a linear change of coordinates, the blow-up limit $u_0$ is a global minimizer of the functional
		$$
		J(\varphi,B_r):=\int_{B_r^+}|\nabla \varphi|^2\,\mathrm{d}x + q^2\big|\{\varphi>0\}\cap B_r^+\big|+
		2\,m\int_{B_r'}\,|\varphi| \,\mathrm{d}\HH^{d-1}\,,
		$$
		with $q=\sqrt{Q(x_0)}$ and $m=\beta(x_0)/a(x_0)$.
		Therefore, the result is a classical consequence of the monotonicity result \cref{p:weiss}. Indeed, by exploiting the almost-monotonicity of the Weiss type formula and the rescaling of the functional along the blow-up sequence, we can relate the Weiss formula of \cref{p:weiss2} with the one of \cref{p:weiss}, which implies the one-homogeneity of the blow-up limit.
	\end{proof}

	\section{Global homogeneous minimizers in dimension two}\label{s:global2}
	In this section we give a complete classification of one-homogeneous minimizers of the functional
	$$
	J(v,E):=\int_{E^+}|\nabla v|^2\,\mathrm{d}x + q^2\big|\{v>0\}\cap E^+\big|+
	2\,m\int_{E'}\,|v| \,\mathrm{d}\HH^{d-1}\,,q>0\quad\text{and}\quad m\in(-q,q),
	$$
	in the two-dimensional case. We recall that the blow-up limits arising from Section \ref{s:blow.analysis} are global minimizers of the functional
	$$
	J=\mathcal F_{x_0}\quad\text{with}\quad q=\sqrt{Q(x_0)}\quad\text{and}\quad m=\frac{\beta(x_0)}{a(x_0)}\ ,
	$$
	defined in \eqref{e:def-J-intro}. The main result of this section is the following. 
	\begin{proposition}\label{p:classification2D}
		Let $q \in \R, m \in (-q,q)$. Then, the only one-homogeneous global minimizers of the functional $J$ in $\R^2$ are the half-plane solutions 
		$$
		v(x) = \left(\sqrt{q^2-m^2}\,(x\cdot\nu) + m(x\cdot e_2)\right)^+,
		$$
		associated to $\nu=e_1$ and $\nu=-e_1$.
	\end{proposition}
	In order to prove this proposition, in \cref{l:class2}, we start by classifying one-homogeneous two-dimensional solutions to the Euler-Lagrange equations \eqref{e:bernoulli-J} associated to $J$. Below, we will exclude from the list of \cref{l:class2} the solutions which are not minimizing.
	
	\begin{lemma}\label{l:class2}
		Let $v:\R^2\to\R$ be a non-trivial non-negative  continuous and one-homogeneous function such that
		\be\label{el.blow}
		\begin{cases}
			\Delta v=0\quad\text{in}\quad \{x_2>0\}\cap \{v>0\},\medskip\\
			|\nabla v|=q\quad\text{on}\quad \{x_2>0\}\cap \partial\{v>0\},\medskip\\
			e_2\cdot\nabla v=m\quad\text{on}\quad \{x_2=0\}\cap \{v>0\},
		\end{cases}
		\ee
		with $q>0,m \in \R$. Then, up to a reflection with respect to $e_1$, $v$ is one of the functions from following list:
		\begin{enumerate}
			\item \label{1} for $m \in (-q,q)$ the function
			$$
			v(x)=\left(\sqrt{q^2 - m^2 }(x\cdot e_1) + m (x\cdot e_2)\right)^+
			$$
			is a solution such that $\{v=0\}\cap \{x_2=0\}= \{x_1\leq 0, x_2=0\}$;
			\item for $m \in (-q,0)$ the function
			$$
			v(x)=\left(\sqrt{q^2-m^2}|x\cdot e_1| + m(x\cdot e_2)\right)^+,
			$$
			is a solution such that $\{v = 0\} \cap \{x_2=0\} =(0,0)$;
			\item for every $m$, every $q$, and every $C>0$,  the function
			$$
			v(x)=C(x\cdot e_2)^+,
			$$
			is a solution such that $\{v=0\}\equiv \{x_2=0\}$.
		\end{enumerate}
	\end{lemma}
	\begin{proof}
		Since $v$ is one-homogeneous in $\R^2$, there exists a $0$-homogeneous function $w \colon S^{1}\to \R$ such that
		$$
		v(x)=|x|w\left(\frac{x}{|x|}\right)\quad\mbox{in }\R^2.
		$$
		Furthermore, if we set $(r,\theta)$ as the polar coordinates in $\R^2$, with $r>0, \theta \in S^1$, by the first two conditions in \eqref{el.blow} we get $-w''= w$ in $\{w>0\}^+$ and
		$$
		w'(0) = m\quad\mbox{if }\, 0 \in \{w>0\},\qquad\qquad-w'(\pi) = m  \quad\mbox{if }\, \pi \in \{w>0\}.
		$$
		Now, let us split the classifications in three cases:\\\\
		{{Case 1. }}Suppose that $w(0)>0$, then it must be of the form
		$$
		w(\theta) = \left(w(0)\cos\theta + m\sin \theta\right)^+.
		$$
		Let $\theta^*_0 \in (0,2\pi)$ be the first zero-point of $w$. Necessary, we have $\theta^*_0 \in (0,\pi)$ and so, by the last condition in \eqref{el.blow}, we get
		\be \label{theta*}
		\begin{cases}
			w(\theta^*_0)=0\\
			w'(\theta^*_0)= q
		\end{cases}
		\longrightarrow\quad
		w(0)=\sqrt{q^2-m^2},\,\, \theta^*_0=\begin{cases}
			\arctan\left(-\frac{1}{m}\sqrt{q^2-m^2}\right)
			, & \mbox{if } m<0 \\
			\pi+\arctan\left(-\frac{1}{m}\sqrt{q^2-m^2}\right)
			&  \mbox{if } m>0.
		\end{cases}
		\ee
		and so
		$$
		v(x_1,x_2)=\left(\sqrt{q^2-m^2}(x_1) + m(x_2)\right)^+
		$$
		is a solution. By symmetry, we get that if $w(\pi)>0$ then
		\be\label{theta*2}
		v(x_1,x_2)=\left(-\sqrt{q^2-m^2}x_1 + m x_2\right)^+\!,\quad\mbox{with }\theta^*_\pi=
		\begin{cases}
			\pi + \arctan\left(\frac{1}{m}\sqrt{q^2-m^2}\right)
			, & \mbox{if } m<0 \\
			\arctan\left(\frac{1}{m}\sqrt{q^2-m^2}\right)
			&  \mbox{if } m>0.
		\end{cases}
		\ee
		is a solution too.\\\\
		{{Case 2. }}Suppose that both $w(0)>0$ and  $\phi(\pi)>0$. Then, by \eqref{theta*} and \eqref{theta*2} in Case 1, we get
		$$
		\theta^*_\pi > \theta^*_0 \quad\mbox{if and only if}\quad \beta <0
		$$
		and consequently, only for $m<0$ the function
		$$
		v(x_1,x_2)=\left(\sqrt{q^2-m^2}(x_1) + m(x_2)\right)^+ + \left(-\sqrt{q^2-m^2}(x_1) + m(x_2)\right)^+
		$$
		is solution to \eqref{el.blow}.\\\\
		{{Case 3. }}Suppose that $w(0)=0=w(\pi)$, then  $w(\theta)=B\sin \theta$ and by the last condition in \eqref{el.blow} we get $v(x_1,x_2)=q(x_2)^+$ in $\{x_2\geq 0\}$.
	\end{proof}
	For the sake of completeness the following results are already stated in $\R^d$, for $d\geq 2$, and they will be crucial for the analysis in higher dimensions.
	\begin{lemma}\label{l:global.deg.instable}
		Assume that $m\in(-q,q)$ and $C>0$. Then $v(x) = C(x \cdot e_d)^+,
		$ is not a global minimizer of $J$ in $\R^d$.
	\end{lemma}
	\begin{proof}
		By an interior perturbation with a vector field pointing inwards, it is immediate to check that if $C<q$, then $v$ is not a minimizer. We focus on the case $C\ge q$. 
		Suppose by contradiction that $v$ is a global minimizer of $J$ in $\R^d$. Thus, consider $\gamma \in \R$ such that $$
		m-C<\frac{\gamma}{2} <0.
		$$
		Now, let $\varphi$ be the solution to
		$$
		\begin{cases}
		\Delta \varphi =0 & \mbox{in }\{x_d>0\}\cap B_R \\
		\varphi=0 & \mbox{on } \{x_d>0\}\cap \partial B_R\\
		\partial_{x_d} \varphi = \gamma & \mbox{on }\{x_d=0\}\cap B_R ,
		\end{cases}
		$$
		with $R>0$. Since $\gamma <0$, by maximum principle we have $\varphi>0$ in $B_R$ and so
		$$
		J(v + \varphi,B_R) = J(v,B_R) + 2\left(m-C-\frac{\gamma}{2}\right)\int_{B_R\cap H}\varphi \mathrm{d}x' < J(v,B_R) ,
		$$
		in contradiction with the hypothesi of minimality.
	\end{proof}

	\begin{lemma}\label{l:global.deg2}
		Assume that $m\in (-q,0)$. Then, the function
		$$
		v(x)=\left(\sqrt{q^2-m^2}|x \cdot \nu| + m(x \cdot e_d)\right)^+,
		$$
		with $\nu\in \R^d$ a unit vector orthogonal to $e_d$, is not a global minimizer of $J$ in $\R^d$.
	\end{lemma}
	\begin{proof}
		Since $v$ is constant with respect to the variables $(x_3,\dots, x_d)$, it is not restrictive to prove the main result in $\R^d$ with $d=2$ and $\nu=e_1$. Suppose by contradiction that $v$ is a global minimizer of $J$ in $\R^2$.
		Consider the rectangle
		$$
		\mathcal R := [-1,1] \times \left[0,-\frac{1}{m}\sqrt{q^2-m^2}\,\right],
		$$
		and we notice that $v: \mathcal R \to\R$ is symmetric with respect to  $x_1$, that is: 
		$$v(x_1,x_2)=v(-x_1,x_2).$$ 
		For any $t\ge 0$, we define the function $v_t: \mathcal R \to\R$ as
		$$
		v_t(x_1,x_2):=\left(\frac{1}{1+t}\sqrt{q^2-m^2}(|x_1|+t)  + mx_2\right)^+,\quad\mbox{for }t\geq 0.
		$$
		By construction, we have that $v_t\equiv v$ on $\partial  \mathcal R \cap\{x_2>0\}$. Indeed, 
		$$
		\begin{cases}
		v\equiv v_t \equiv 0  & \mbox{on }\  [-1,1]\times \left\{-\frac{1}{m}\sqrt{q^2-m^2}\right\}\,,\\
		v \equiv v_t & \mbox{on }\  \{\pm 1\}\times \left[0,-\frac{1}{m}\sqrt{q^2-m^2}\,\right]\,.
		\end{cases}
		$$
		Since $v_0 = v$, the minimality of $v$ implies that the function
		$$
		f(t):= J(v_t, \mathcal R )\,,
		$$
		has a local minimum at $t=0$. On the other hand, since 
		\begin{align*}
		|\nabla v_t|^2&=q^2-2t(q^2-m^2)+3t^2(q^2-m^2)+o(t^2),\\
		|\{v_t>0\}\cap \mathcal R|&=\frac{\sqrt{q^2-m^2}}{|m|}(1+t-t^2)+o(t^2),\\
		2m\int_{-1}^1v_t(x_1)\,dx_1&=-2|m|\sqrt{q^2-m^2}\Big(1+t-t^2\Big)+o(t^2),
		\end{align*}
		we get that 
		$$f'(0)=0\quad\text{and}\quad f''(0)=2\frac{\sqrt{q^2-m^2}}{|m|}(m^2-q^2)<0,$$
		in contradiction with the minimality of $v$.
	\end{proof}

	\section{Regularity of $\mathrm{Reg}(u)$}\label{s:visco}
	In this section we prove that in a neighborhood of any point $x_0\in{\rm Reg}(u)\subset\partial\{u>0\}\cap\{x_d=0\}$, the  free boundary $\partial\{u>0\}\cap\{x_d\ge 0\}$ is locally a the graph of a $C^{1,\alpha}$ smooth function.
	
	\subsection{Proof of \cref{t:viscosity}} In this section we prove that if $u$ is a minimizer, then it satisfies \eqref{interno} and \eqref{FB} in the sense of \cref{def:solutionnew}. As explained in \cref{sub:open}, the set $\{u>0\}$ is open and the validity of \eqref{interno} weakly in $H^1$, and so, in the classical sense.  
	Therefore, it remains to prove the validity of \eqref{FB}.\medskip

	Suppose now by contradiction that $u$ is touched from below by $\psi^+$ at $x_0 \in B_1'\cap \partial\{u>0\}$ and both the $(i), (ii)$ in \eqref{terza} are not satisfied. Without loss of generality, we can assume that $A(x_0)=\mathrm{Id}, a(x_0)=1$ and we denote $q:=\sqrt{Q(x_0)}, \beta=\beta(x_0)$ for notational convenience.
	Consider now the blow-up sequences of $u$ and $\psi$ centered at the contact point $x_0$, that is
	$$
	u_{x_0, r}(x) = \frac{1}{r}u(x_0 + r x) \quad\mbox{and}\quad \psi_{x_0, r}(x) = \frac{1}{r}\psi(x_0+r x),
	$$
	with $r>0$. Then, up to a subsequence, they converge respectively to some $u_0$ and $\psi_{0}$, uniformly on every
	compact subset of $\{x_d\ge 0\}$. By Proposition \ref{p:limit} and Corollary \ref{cor:weiss.cor} we know that $ u_0$ is a one-homogeneous global minimizer of the functional $\mathcal{F}_{x_0}$. Moreover, since $\psi \in C^2$, we get that $\psi_0(x) = \alpha (x \cdot \nu)$, where $\alpha:=|\nabla \psi|(x_0)$. Recall that since $\psi$ contradicts the condition \rm(VS3) of \cref{def:solutionnew}, we have that
	\be \label{absurd1}
	\alpha \nu \cdot e_d > \beta \qquad \text{ and } \qquad \alpha |1-\nu \cdot e_d| >\sqrt{q^2-\beta^2}.
	\ee
	Thanks to the homogeneity of the blow-ups, up to a dimension reduction argument, it is not restrictive to assume that
	\be\label{viscous:cond1}
	\partial \{ u_0 >0 \} \cap \partial \{ \psi_0 >0 \} \cap \Big(\mathbb{S}^{d-1} \cap \{x_d\ge 0\}\Big) = \emptyset.
	\ee
	Indeed, if there is a second contact point $y_0\in\{x_d=0\}\cap\partial\{u_0>0\}$, by taking the blow-up of $u_0$ at $y_0$, we obtain a function which is invariant in the direction of $y_0-x_0$, and so, it is a minimizer of $\mathcal F$ in dimension $d-1$ which still satisfies \eqref{absurd1}. If \eqref{viscous:cond1} does not hold, then we iterate this procedure up to dimension two. Finally, in dimension two the blow-up limits were classified in \cref{s:global2} and so, we know that \eqref{absurd1} cannot happen. This proves that we can assume \eqref{viscous:cond1}.

	Now, as a consequence of \eqref{viscous:cond1} and the continuity of $u_0$ and $\psi_0$, we get that
	\be\label{viscous:cond2}
	\psi_0 >  u_0 + \delta \quad \text{on } \{ x_d = 0\} \cap \mathbb{S}^{d-1},
	\ee
	for some small constant $\delta>0$.
	
	Thanks to conditions \eqref{viscous:cond1} and \eqref{viscous:cond2} there exists a small parameter $\eta>0$ such that
	\be\label{viscous:wrongPlaneCond1}
	(\alpha - \eta)( \nu \cdot e_d) > \beta, \qquad (\alpha - \eta) |1-\nu \cdot e_d| >\sqrt{q^2-\beta^2}
	\ee
	and
	\be\label{viscous:wrongPlaneCond2}
	h(x) := (\alpha - \eta) (x\cdot\nu + \eta) \leq  u_0\quad \text{ on } \mathbb{S}^{d-1} \cap \{x_d\geq 0\}.
	\ee
	Now we proceed to show that \eqref{viscous:wrongPlaneCond1}-\eqref{viscous:wrongPlaneCond2} lead to a contradiction.
	
	To begin with, let $\Lambda \in (q, \alpha - \eta)$, and notice that the positive part $h_+$ is the unique minimizer of the functional
	\[
	\widetilde{\mathcal{F}}(\phi):=\int_{B_1^+}{|\nabla \phi|^2 + \Lambda \ind_{\{\phi>0\}}\mathrm{d}x} + \int_{B_1'}2 \beta \phi\,\mathrm{d}x'
	\]
	over all $\phi \in H^1(B_1^+)$, with $\phi=h_+ \text{ on }\partial B_1 \cap \{ x_d > 0 \}$ and $\phi\le h_+ \text{ on } B_1^+$. This property follows by slicing $h_+$ with two dimensional planes containing its gradient, and then using a viscous sliding argument in two dimensions. Indeed, since $h$ is invariant in all the directions orthogonal to $\nu$ and $e_d$, by slicing $h_+$ with two dimensional planes spanned by $\nu$ and $e_d$, we only need to show this claim in dimension two. Now, when $d=2$, we can argue by contradiction as follows. Suppose that the minimizer $\phi$ of the above variational problem is different from $h_+$. Let now $t_0$ be the smallest $t>0$ such that $h_t(x):=h_+(x-t\nu)$ remains below $\phi$. If $t_0\neq0$, then $h_{t_0}$ touches $\phi$ from below at a free boundary point, which is impossible by the classification of the two-dimensional blow-ups from \cref{s:global2}.

	Now, let us consider the functions
	\[
	v_1 :=  u_0 \vee h_+ \qquad \text{and} \qquad v_2 := u_0 \wedge h_+,
	\]
	so that, by \eqref{viscous:wrongPlaneCond2}, $v_1 = u_0$ and $v_2 = h$ on $\partial B_1 \cap \{ x_d > 0 \}$, respectively. Hence, since $\widetilde u_0$ minimizes $\mathcal F$ and since $h_+$ minimizes $\widetilde{\mathcal F}$ (in the sense explained above), we get
	\be\label{viscous:MinimalCondition}
	\mathcal{F}_{x_0}\left( \tilde u_0 \right) \le \mathcal{F}_{x_0}\left( v_1 \right) \qquad \text{and} \qquad \widetilde{\mathcal{F}}(h_+) \le  \widetilde{\mathcal{F}}(v_2).
	\ee
	However, since
	\[
	\mathcal{F}_{x_0}\left( v_1 \right) + {\mathcal{F}_{x_0}}(v_2) = \mathcal{F}_{x_0}\left( \tilde u_0 \right) + {\mathcal{F}_{x_0}}(h),
	\]
	and using that
	$\Lambda \ge q$ and $|\{h>0\}|\ge |\{v_2>0\}|$, we get
	\[
	\mathcal{F}_{x_0}\left( v_1 \right) + \widetilde{\mathcal{F}}(v_2) \le \mathcal{F}_{x_0}\left( \tilde u_0 \right) + \widetilde{\mathcal{F}}(h),
	\]
	so that from \eqref{viscous:MinimalCondition} we obtain 
	\[
	\widetilde{\mathcal{F}}(h_+) =  \widetilde{\mathcal{F}}(v_2),
	\]
	which leads to a contradiction since $h_+$ is the unique minimizer of $\widetilde{\mathcal{F}}$.
	When $\psi_+$ touches $u$ from above, the proof is analogous. \qed
	
	\subsection{Harnack type inequality} 
	We start by proving a Harnack type inequality for viscosity solutions to problem \eqref{FB} in the case where the variable coefficients $A, Q$ and $\beta$ have small oscillation. Indeed, in view of assumptions $(\mathcal H1), (\mathcal H2), (\mathcal H3)$, up to rescaling the problem near a regular point of $B_1'\cap \partial\{u>0\}$, such requirement it is fully satisfied. More precisely, we assume that\vspace{0.2cm}
	\begin{enumerate}
		\item[$\qquad(\mathcal I1)$] $A, Q$ and $\beta$ satisfy the assumptions $(\mathcal H1), (\mathcal H2), (\mathcal H3)$ and $$A(0)=\text{\rm Id}\qquad\text{and}\qquad |\beta|<a\sqrt{Q}\quad\text{on}\quad B_1'\,;$$
		\item[$\qquad(\mathcal I2)$] for every couple of indices $1\le i,j\le d$, we have $
		\norm{a_{ij}-\delta_{ij}}{L^\infty(B_1\cap \{x_d\geq 0\})}\leq \eps^{\frac{2}{\sigma}}$, and
		$$ \norm{Q-Q(0)}{L^\infty(B_1\cap \{x_d\geq 0\})}\leq Q(0)\eps^{2/\sigma}\qquad\text{and}\qquad
		\norm{\beta - \beta(0)}{L^\infty(B_1')}\leq |\beta(0)|\eps^{2/\sigma},
		$$
		where $\sigma>0$ is the constant of \cref{c:schauder.easy}, depending only on $d$ and the constants in $(\mathcal{H}_1)$.
	\end{enumerate}
	\begin{proposition}\label{p:holder}
		There exists $\overline{\eps}>0$ such that the following holds.
		Let $0\in \partial\{u>0\}$ and $u$ be a viscosity solution of \eqref{e:interno} and \eqref{FB}, satisfying $(\mathcal I1), (\mathcal I2)$. Suppose that, for some point $x_0 \in (B_1^+\cup B_1')\cap \overline{\{u>0\}}$
		\be\label{flat_2}
		h_{q,m,e_1}(x+t_0 e_1) \leq u(x) \leq h_{q,m,e_1}(x+s_0 e_1)\quad \text{in $B_r^+(x_0)\cup B_r'(x_0)$,}
		\ee
		with $q:=\sqrt{Q(0)}, m:=\beta(0)$ and $s_0 -t_0 \leq  \bar\eps r$. Then, we get
		\be\label{flat_2_improved} h_{q,m,e_1}(x+t_1 e_1)\leq u(x)  \leq h_{q,m,e_1}(x+s_1 e_1) \quad \text{ in $B_{r/20}^+(x_0)\cup B_{r/20}'(x_0)$,}
		\ee
		with $$t_0 \leq t_1 \leq s_1 \leq s_0, \quad  s_1 - t_1= (1-c)(s_0-t_0) ,$$ and $c\in (0,1)$ a universal constant.
	\end{proposition}
We notice that by iterating \cref{p:holder} we obtain the main compactness argument used in the linearization procedure near regular points. We set
	$$\tilde u(x):= \frac{1}{\sqrt{q^2 - m^2 }}
	\frac{u(x)-h_{q,m,e_1}(x)}{\eps}\quad\mbox{in }(B_1^+\cup B_1') \cap \overline{\{u>0\}}.$$
	By iterating the previous result we get the following corollary
	\begin{corollary}\label{cor:AA} Let $0 \in \partial\{u>0\}$ and $u$ be a viscosity solution in $B_1$ (see \cref{def:solutionnew}) satisfying $(\mathcal I1), (\mathcal I2)$. Suppose that
		\be\label{flat*}
		h_{q,m,e_1}(x-\eps e_1) \leq u(x) \leq h_{q,m,e_1}(x+\eps e_1)\quad \text{in}\quad B_1^+\cup B_1',
		\ee
		for some $\eps>0$, with $q:=\sqrt{Q(0)}, m:=\beta(0)$. Then, there exists $\bar \eps>0$ small universal and depending on $q$ and $m$, such that if $\eps \leq \bar \eps$ then $\tilde u$  has a universal H\"older modulus of continuity at $x_0 \in B_{1/2}$ outside a ball of radius $r_\eps,$ with $r_\eps \to 0$ as $\eps \to 0.$
	\end{corollary}
	Finally, we can state the main ingredient in the proof of \cref{p:holder}.
	\begin{lemma}\label{l:fbharnack}
		Let $0\in \partial\{u>0\}, \eps>0$ and $u$ be a viscosity solution in the sense of Definition \ref{def:solutionnew}, satisfying $(\mathcal I1), (\mathcal I2)$. Assume that,
		\be\label{flat_harnack2}  h_{q,m,e_1}(x+ s e_1)\leq u(x) \leq h_{q,m,e_1}(x+(s+\eps)e_1) \quad \text{in $B_1^+\cup B_1'$,}
		\ee
		with
		$$
		|s| < s_0, q=\sqrt{Q(0)}, m=\beta(0),
		$$
		with $s_0$ depending only on $q$ and $m$.\\
		Then, there exists $\bar \eps>0$, such that if $\eps \in (0,\bar \eps]$ then at least one of the following holds true:
		\be\label{flat_harnack3}  h_{q,m,e_1}(x+(s +C\eps)e_1) \leq u(x) \quad \text{in $B_{1/2}^+\cup B_{1/2}'$,}
		\ee
		or
		$$
		u(x) \leq h_{q,m,e_1}\left(x + (s +(1-C)\eps)e_1\right) \quad \text{in $B_{1/2}^+\cup B_{1/2}'$,}
		$$
		for some $C>0$ small universal.
	\end{lemma}
	\begin{proof}
		First, consider $\bar \eps < \eps_0$, with $\eps_0$ the constant of \cref{t:app-schauder} and \cref{c:schauder.easy}. Then, fix $\overline{x} \in B_1^+$ to be such that
		$$
		|\overline{x}|=\frac15\quad\text{and}\quad \mathrm{dist}(\overline{x}, \{x_d=0\})=
		\mathrm{dist}(\overline{x}, \partial\{h_{q,m,e_1}>0\}).
		$$
		Precisely, 
		$$\overline{x}= \frac15 \left(\frac{2q}{q-m}\right)^{-1/2}\left(e_1 + \frac{m+q}{\sqrt{q^2-m^2}}e_d\right).$$
		We also fix
		$$
		r:=\frac{1}{10}\min\left\{1,\frac{m+q}{\sqrt{q^2-m^2}}\right\},\quad R:=\frac34\quad\text{and}\quad
		s_0:= r\frac{\sqrt{q^2-m^2}}{2q}.
		$$
		Therefore, we have that $B_r(\overline{x})\subset \{h_{q,m,e_1}(x+s e_1)>0\}$ and
		\be\label{inclus}
		B_{1/2}^+\cup B_{1/2}'\subset\subset(B_R^+(\overline{x})\cup B_R'(\overline{x}))\subset\subset B_{1}^+\cup B_{1}'.
		\ee
		Finally, the proof is divided in two cases: assume that
		\be\label{bottom}u(\bar x) \leq h_{q,m,e_1}\left(\bar x+\left(s+\frac12\eps\right)e_1\right),\ee
		then, we show that
		\be\label{double} u(x)\leq h_{q,m,e_1}\left(x + (s +(1-C)\eps)e_1\right) \quad \text{in $B_{1/2}^+ \cup B_{1/2}'$ }\ee
		for some $C>0$ universal. On the other hand, if
		\be \label{secondcase} u(\bar x) \geq h_{q,m,e_1}\left(\bar x+\left(s+\frac12\eps\right)e_1\right),\ee we can prove that $$
		h_{q,m,e_1}(x+(s +C\eps)e_1) \leq u(x) \quad \text{in $B_{1/2}^+\cup B_{1/2}'$.}
		$$
		Before going into the details of the proof, let $U(\overline{x}) = B_{R}(\overline{x})\setminus \overline B_{r/2}(\overline{x})$ and consider $w$ defined as
		\be\label{w}
		w(x)=
		\ddfrac{\abs{x-\bar{x}}^{\gamma} - R^{\gamma}}{(r/2)^{\gamma} - R^{\gamma}}\,\mbox{ in }\, U(\overline{x}),\qquad w\equiv 1\,\mbox{ in }\,B_{r/2}(\overline{x}),
		\ee
		with $\gamma <0$ be such that $\Delta w >0$ in $U(\overline{x})$.\medskip
		
		\noindent{\it Case 1.} If \eqref{secondcase} holds true, set
		\be\label{def.p}
		p(x):=\sqrt{q^2-m^2}(x_1+s)+mx_d.
		\ee
		Since $|s|<s_0$ and  by the flatness assumption
		\be\label{flat} u(x) - p(x)\geq 0 \quad \text{in $B_1^+\cup B_1'$},\ee
		we immediately deduce that $B_{r}(\overline{x}) \subset B_1^+\cap \{u>0\}$. Now, consider the solutions $p_A$ and $w_A$ to
		$$
		\begin{cases}
		\text{\rm div}(A\nabla p_A)=0&\text{in}\quad B_1\,,\\
		p_A=  p&\text{on}\quad \partial B_1
		\end{cases}
		\quad\mbox{and}\quad
		\begin{cases}
		\text{\rm div}(A\nabla w_A)=\Delta w&\text{in}\quad U(\overline{x})\,,\\
		w_A= 1&\text{in}\quad B_{r/2}(\overline{x}),\\
		w_A= 0&\text{on}\quad \partial B_R(\overline{x}).
		\end{cases}
		$$
		Clearly, by combining assumption $(\mathcal I2)$ with \cref{t:app-schauder} and \cref{c:schauder.easy}, we get
		\begin{align}\label{e:replace}
		\begin{aligned}
		\norm{p_A-p}{L^\infty(B_1)}  +\norm{\nabla p_A-\nabla p}{L^\infty(B_1)}  &\leq K \eps^2,\\
		\norm{w_A-w}{L^\infty(B_1)}  +\norm{\nabla w_A-\nabla w}{L^\infty(B_1)}  &\leq K \eps^2,
		\end{aligned}
		\end{align}
		with $K>0$ depending only on the dimension and the constants involved in $(\mathcal H1)$. Hence, in view of \eqref{flat}, we have that
		$$
		u(x)-p_A(x)+K\eps^2\geq 0 \quad\mbox{in }B_1^+\cup B_1'
		$$
		and so, by applying the Harnack inequality in $B_{r}(\overline{x})$, we get that
		\begin{align*}\label{first_step}
		u(x) -p_A(x) +K\eps^2 &\geq c\Big(u(\overline{x}) - p_A(\overline{x}) + K\eps^2\Big)\\
		&\geq c(u(\overline{x}) - p(\overline{x})) + c K\eps^2\quad \text{in $B_{r/2}(\overline{x})$}.
		\end{align*}
		Finally, by using the assumption \eqref{secondcase}, we deduce that
		\be\label{first_step}
		u(x) -p_A(x) \geq c_0 \eps \quad \text{in $B_{r/2}(\overline{x})$},
		\ee
		with $c_0>0$ depending only on the dimension and the constants involved in $(\mathcal H1)$. Now, let us set
		\be\label{def.sub}
		v_t(x) := p_A(x) -K\eps^2 + c_0 \eps (w_A(x)-1) +c_0 \eps t  \quad \text{in $\overline{B_{R}(\bar{x})}$}
		\ee
		for $t\geq 0$. Since
		\be\label{>}\text{\rm div}(A\nabla v_t) =  c_0\eps \text{\rm div}(A\nabla w_A) > C\eps\quad\mbox{in }U(\overline{x}),
		\ee
		then, by \eqref{flat}, we deduce that $$v_0\leq p_A - K\eps^2 \leq p\leq u \quad\text{in $\overline{B_{R}^+(\overline{x})}\cap \{x_d\geq0\}$}.$$
		Thus, let $\overline{t}>0$ be the largest $t>0$ such that $v_t\leq u$ in $\overline{B_{R}^+(\overline{x})}\cap \{x_d\geq0\}$. We want to show that $\overline{t}\geq 1$. Indeed, by the definition of $v_t$, we will get
		$$
		u \geq v_{1} \geq p_A -K\eps^2+ c_0 \eps w_A \geq p -2K \eps^2 + c_0 \eps w_A \quad\text{in $\overline{B_{R}^+(\overline{x})}$}.
		$$
		In particular, by \eqref{inclus}, since $w_A \geq c_1$ on $\overline{B_{1/2}^+}$, it implies that
		$$
		u \geq p + C\eps  \quad\text{in $\overline{B_{1/2}^+}$},
		$$
		as we claimed.
		
		Suppose by contradiction that $\bar{t} < 1$ and let $\tilde{x} \in \overline{B_{R}^+(\overline{x})}$ be the touching point between $v_{\overline{t}}$ and $u$ and let us show that it can only occur on $\overline{B_{r/2}(\overline{x})}$. Since $w \equiv 0$ on $\partial B_{R}(\overline{x})$ and $\overline{t}< 1$ we get
		$$
		v_{\overline{t}} = p_A-K\eps^2 - c_0\eps +c_o \eps\overline{t} < p < u \quad \text{on $\partial B_{R}(\overline{x})$},
		$$
		thus it is left to exclude that $\tilde{x}$ belongs to the annulus $U(\overline{x}) = B_{R}(\overline{x})\setminus \overline B_{r/2}(\overline{x})$. In order to exclude this possibility, we list some of the main properties of $w$. 
		
		First, since $w$ is radially symmetric, we notice that for every unit vector $\xi \in \R^d$ we have $\nabla w \cdot \xi = \abs{\nabla w} (\nu_x \cdot \xi)$ in $U^+(\overline{x})$, where $\nu_x$ is the unit direction of $x-\overline{x}$. Thus, we get
		\begin{align*}
		|A^{1/2}\nabla v_{\overline{t}}|
		&\geq
		(1-\eps^2)|\nabla(p + c_0\eps \nabla w)\cdot \xi + \nabla(p-p_A)\cdot \xi + c_0\eps \nabla (w_A-w)\cdot \xi|\\
		&\geq
		|\nabla p\cdot \xi  + c_0\eps |\nabla w|(\nu_x\cdot \xi)| - K\eps^2.
		\end{align*}
		We consider the vectors
		\be\label{vec}
		\xi_1 := \frac{1}{q}\left(\sqrt{q^2-m^2}e_1 + m e_d\right),\quad \xi_2 := e_d\quad\mbox{or}\quad \xi_3:=e_1.
		\ee
		On one side, by the definition \eqref{def.p} of $p$, we get
		$$
		\nabla p\cdot \xi_1 = q,\quad \nabla p\cdot \xi_2 = m \quad\mbox{or}\quad \nabla p\cdot \xi_3=\sqrt{q^2-m^2}.
		$$
		On the other, from the definition of $w$, we get that $\abs{\nabla w} >c $ on $U(\overline{x})$ and $\nu_x \cdot \xi$ is bounded by below in the region $\{v_{\overline{t}} \leq 0\}\cap U(\overline{x})$ for the choices of vector in \eqref{vec}. Indeed, since
		$$\{v_{\overline{t}}\leq 0\} \subset \{p - c_0\eps \leq 0\}\quad\mbox{and}\quad
		\{v_{\overline{t}}> 0\} \subset \{p - c_0\eps > 0\},$$
		we get that
		\begin{align*}
		\{v_{\overline{t}}\leq 0\} \cap U(\overline{x})
		&\subset \left\{q(x\cdot \xi_1)\leq - \sqrt{q^2-m^2}s + c_0\eps \right\} \cap U(\overline{x})\\
		&\subset \left\{q(x\cdot \xi_1)\leq \frac32\sqrt{q^2-m^2}|s_0| \right\} \cap U(\overline{x}),\\
		\{v_{\overline{t}}\leq 0\} \cap U(\overline{x_0})\cap \{x_d=0\}
		&\subset \left\{\sqrt{q^2-m^2}x_1 \leq  - \sqrt{q^2-m^2}s +c_0\eps \right\} \cap U(\overline{x})\\
		&\subset \left\{x_1 \leq \frac32|s_0|\right\} \cap U(\overline{x})\,,
		\end{align*}
		for $\eps$ sufficiently small. Hence, in view of $(\mathcal I2)$, we infer that
		\begin{align}\label{>123}
		\begin{aligned}
		e_d \cdot A\nabla v_{\overline{t}}  > m + C\eps\geq \beta + \frac12C\eps &\quad\mbox{on }\{v_{\overline{t}}>0\}\cap U(\overline{x})\cap \{x_d=0\},\\
		\abs{A^{1/2}\nabla v_{\overline{t}}}  > q+C\eps\geq \sqrt{Q} + \frac12C\eps &\quad\mbox{on } \partial\{v_{\overline{t}}>0\}\cap U(\overline{x}),\\
		\abs{A^{1/2}\nabla_{x'} v_{\overline{t}}} > \sqrt{q^2-m^2}+C\eps \geq \sqrt{Q^2-\frac{\beta^2}{a^2}}+ \frac12C\eps &\quad\mbox{on }\partial\{v_{\overline{t}}>0\}\cap U(\overline{x})\cap \{x_d=0\},
		\end{aligned}
		\end{align}
		for $\eps$ sufficiently small.
		
		We are now in position to exclude the presence of contact points in the annulus $U(\bar x)$. First, in view of \cref{def:solutionnew}, the condition \eqref{>} implies that a touching point cannot occur inside the positivity set $U(\bar x) \cap (B_1^+\cup B_1')\cap \{u>0\}$. Then, by \eqref{>123} we have that it cannot occur on the $(d-1)$ dimensional free boundary $U(\overline{x}) \cap \partial\{u>0\}\cap \{x_d>0\}$ and in the positivity set lying inside the hyperplane $U(\overline{x})\cap \{u>0\}\cap \{x_d=0\}$. Finally, the last viscosity condition from \eqref{FB} implies that a touching point cannot appear on $U(\overline{x})\cap \partial\{u>0\}\cap \{x_d=0\}$. 
		
		Therefore, the touching point $\tilde{x}$ belongs to $\overline{B_{r/2}(\overline{x})}$ and
		$$
		u(\tilde{x}) = v_{\overline{t}}(\tilde{x}) = p(\tilde{x}) -K\eps^2 + c_0 \eps \overline{t} < p_A(\tilde{x}) + \frac12 c_0 \eps,
		$$
		in contradiction with \eqref{first_step}.\medskip

		\noindent{\it Case 2.} If \eqref{bottom} holds true, we proceed as above by considering the family
		$$
		v_t(x) = p_A(x)+K\eps^2 - c_0 \eps (w_A(x)-1) -c_0 \eps t
		$$
		such that $v_0 \geq u$ in $\overline{B_{R}^+(\bar x)} \cap \{p_A>-\frac32\eps\}$. The proof is analogous.
	\end{proof}
	
	\subsection{The improvement-of-flatness lemma}\label{sub:final}
	We can finally conclude the section with an improvement-of-flatness lemma, from which the $C^{1,\alpha}$ regularity result of \cref{t:epsilon-regularity} follows by standard arguments.\\
	
	First, we need to introduce the linearized problem arising from the improvement-of-flatness procedure.
	\begin{definition}\label{def:limitPbSolutionViscous}
		We say that a function $u \in C\left(\overline{\Sigma \cap B_{1/2}} \right)$ is a viscosity solution to the problem
		\be\begin{cases}\label{linearized.problem}
			\Delta u =0 & \text{in }B_{1/2}^+ \cap \Sigma,\\
			\nabla u \cdot \nabla h_{q,m,e_1} = 0 &\text{on }B_{1/2}^+ \cap \partial\Sigma,\\
			\nabla u  \cdot e_d = 0 &\text{on }B_{1/2}' \cap \Sigma,
		\end{cases}\ee
		if the following conditions hold:
		\begin{itemize}
			\item[(i)] suppose that a function $\phi \in C^2\left( B_r(x_0) \right)$ touches $u$ from below (resp. above) at $x_0 \in \Sigma \cap B_{1/2}^+ $. Then $\Delta \phi \le 0$ (resp. $\Delta \phi \ge 0$);
			
			\item[(ii)] suppose that a function $\phi \in C^1\left( B_r(x_0) \right)$ touches $u$ from below (resp. above) at $x_0 \in B_{1/2}^+ \cap \partial\Sigma$. Then $\nabla \phi \cdot \nabla h_{q,m,e_1} \le 0$ (resp. $\nabla \phi \cdot \nabla h_{q,m,e_1} \ge 0$);
			
			\item[(iii)] suppose that a function $\phi \in C^1\left( B_r(x_0) \right)$ touches $u$ from below (resp. above) at $x_0 \in B_{1/2}' \cap \Sigma$. Then $\nabla \phi \cdot e_d \le 0$ (resp. $\nabla \phi \cdot e_d \ge 0$);
			
			\item[(iv)] suppose that a function $\phi \in C^1\left( B_r(x_0) \right)$ touches $u$ from below (resp. above) at $x_0 \in B_{1/2}' \cap \partial\Sigma$. Then, at least one of the following two conditions holds:
			\begin{itemize}
				\item[(1)] $\nabla \phi \cdot \nabla h_{q,m,e_1} \le 0$ (resp. $\nabla \phi \cdot \nabla h_{q,m,e_1} \ge 0$),
				
				\item[(2)]$\nabla \phi \cdot e_d \le 0$ (resp. $\nabla \phi \cdot e_d \ge 0$).
			\end{itemize}
		\end{itemize}
	\end{definition}
	In order to establish a $C^{1, \alpha}$ decay property for solutions to \eqref{e:bernoulli-D}, we will use the regularity for the linearized problem contained in the following lemma.
	\begin{lemma}\label{l:linearized.regularity}
		Let $q\in \R, m\in (-q,q), \Sigma = \{h_{q,m,e_1}>0\}$ and let $u \in C^{0, \alpha}\left( \overline{\Sigma \cap B_1} \right)$ be a viscosity solution to the problem \eqref{linearized.problem} 
		such that 
		\[\|u\|_{L^{\infty}\left( \overline{\Sigma \cap B_1}  \right)} \le 1.
		\]
		Then, there exists $\eps \in (0,1)$ depending on $q$ and $m$ such that, for every $r \in (0,1/4)$, it holds
		\be\label{e:improved-estimate}
		x\cdot \eta - M r^{1+\eps} \leq u(x) \leq x\cdot \eta + M r^{1+\eps} \quad\mbox{in}\quad B_r\cap {\Sigma},
		\ee
		where $\eta\in \R^d$ is a vector (that does not depend on $r$) satisfying
		\be\label{e:condition-vector}
		\eta \cdot e_d = \eta \cdot e_1 = 0,\quad |\eta|\leq M
		\ee
		with $M>0$ a constant depending on $q$, $m$ and the dimension.
	\end{lemma}
	\begin{proof}
		Given a viscosity solution $u\in C(\Sigma \cap B_1)$, let $w \in H^1(B_1\cap \Sigma)$ be the unique minimizer of 
		\be\label{e:variational-solution-linearized}
		\min\Big\{\int_{B_1\cap\Sigma}|\nabla w|^2\,dx\ :\ w\in H^1(B_1\cap\Sigma),\ w=u \text{ on}\ \partial B_1\cap\Sigma\Big\}.
		\ee   
		We proceed by showing first that any variational solution $w$ to \eqref{e:variational-solution-linearized}  satisfies the regularity estimates \eqref{e:improved-estimate}-\eqref{e:condition-vector} and then that the viscosity solution coincides with the variational one in $\overline{B_1 \cap \Sigma}$, thus inheriting the differentiability properties.\\
		
		By exploiting a bi-Lipschitz flattening of the boundary, it is possible to deduce that the variational solution $w$ to \eqref{e:variational-solution-linearized} is continuous up to the boundary $\partial(\Sigma\cap B_1)$. 
		Indeed, by exploiting the bi-Lipschitz map 
		$$\Phi:\Sigma\to(\R^d)^+$$
		of the form $\Phi(x_1,x'',x_d)=(\phi_1(x_1,x_d),x'',\phi_d(x_1,x_d))$, we can rewrite problem \eqref{linearized.problem} into
		\begin{equation}\label{e:spianata}
		\textrm{div}(D\Phi(\Phi^{-1})\nabla \widetilde w)=0\quad\text{in}\quad\{y_d>0\} \cap B_1\,,\qquad e_d\cdot D\Phi(\Phi^{-1})\nabla \widetilde w=0\quad\text{on}\quad\{y_d=0\} \cap B_1,
		\end{equation}
		where $\widetilde w:=w\circ\Phi^{-1}$. Now, consider the even reflection $w_e$ of $\tilde{w}$ defined as
		$$
		w_e(y) = \begin{cases}
		w(y',y_d) &\mbox{if }y_d\geq0\\
		w(y',-y_d) &\mbox{if }y_d<0.
		\end{cases}
		$$
		Then, since
		\begin{equation}\label{e:spianata-riflessa}
		\textrm{div}(D\Phi(\Phi^{-1})\nabla \widetilde w)=0\quad\text{in}\quad\ B_1\,
		\end{equation}
		by applying \cite[Theorem 8.30]{GT} to equation \eqref{e:spianata-riflessa}, we deduce the existence of a variational solution $w_e\in H^1(B_1)$ with is continuous up to the boundary $\partial B_1$. Finally, going back to the original coordinates, it implies that the original variational solution $w$ achieves the boundary data on $\partial (\Sigma \cap B_1)$ continuously.\medskip

		Now let us deduce the regularity properties for the function $u$. Set $x=(x_1,x'',x_d)$ with $x'' \in \R^{d-2}$. We first exploit the translation invariance of the problem with respect to the directions orthogonal to $e_1$ and $e_d$, to show that 
		\begin{equation}\label{linearized:invariance}
		\text{$w$ and $\Delta_{x''}w$ are locally H\"{o}lder continuous in $\overline{\Sigma} \cap B_1$ with  $\|\Delta_{x''}w\|_{L^{\infty}\left( \overline{\Sigma} \cap B_{3/4} \right)} \le C$}
		\end{equation}
		for some constant $C>0$ depending on $d, m$ and $q$. By the De Giorgi-Nash-Moser theorem and the fact that $\|\widetilde w\|_{C^{0,\alpha}(\Phi(B_{1}))}\le 1$, we have that $\widetilde w$ is H\"older continuous in $\Phi(B_{3/4})$ and $\|\widetilde w\|_{C^{0,\alpha}(\Phi(B_{3/4}))}\le C$ for a universal constant $C$. Now, since $\Phi$ is invariant in any direction 
		$\nu$ orthogonal to $e_1$ and $e_d$, we have that 
		$$y\mapsto\widetilde w(y+h\nu)- \widetilde w(y)$$
		still solves \eqref{e:spianata}, for any $h\in\R$. Thus, arguing as in \cite[Corollary 5.7]{cc}, we get that the function
		$$D_{h,\nu}\widetilde w(y):=\frac{\widetilde w(y+\nu h)-\widetilde w(y)}{h}\,$$
		is uniformly bounded and $\alpha$-H\"{o}lder continuous in $\Phi(B_{1/2})$.

		Similarly, since $\Phi$ is invariant in the direction $x''$, we have that every derivative $v=\partial_{j_1}\dots\partial_{j_n}\widetilde w$ with $2\le j_1,j_2,\dots,j_n\le d-1$ satisfies 
		$$\textrm{div}(D\Phi(\Phi^{-1})\nabla v)=0\quad\text{in}\quad\{y_d>0\}\,,\qquad e_d\cdot D\Phi(\Phi^{-1})\nabla v=0\quad\text{on}\quad\{y_d=0\}.$$
		Thus, by the classical $H^2_{\loc}(B_1)$ regularity, all such derivatives are in $H^1_{\loc}(B_1)$. Hence, by using once again \cite[Corollary 5.7]{cc}, we deduce that
		\[
		\|\left( \partial_{j_1j_1} + \dots + \partial_{j_{d-1}j_{d-1}} \right) \widetilde{u}\|_{C^{0, \alpha}(K)} \le C_K.
		\]
		for all compact $K \Subset B_1$, where $C_K>0$ is depending only on $d, m, q$ and $K$. Now, the claim \eqref{linearized:invariance} follows by the directions of invariance of $\Phi$, so that $\left( \partial_{j_1j_1} + \dots + \partial_{j_{d-1}j_{d-1}} \right) \widetilde{w} = \Delta_x'' w$.
		
		To proceed, let us denote with $(r,\theta)\in \R^2_{x_1,x_d}$ the polar coordinates in the $2$-dimensional space generated by $e_1$ and $e_d$, so that $|x|^2 = |x''|^2 + r^2$ and
		$$
		0=\Delta w = \Delta_{x''} w + \frac{\partial^2}{\partial r^2} w + \frac{1}{r}\frac{\partial}{\partial r}w + \frac{1}{r^2}\frac{\partial^2}{\partial \theta^2} w \quad\mbox{ in }\,\Sigma.
		$$
		For any fixed angle $\Theta\in(0,\pi)$, we denote by $\Sigma_\Theta$ the subset of $\R^d_+$ which can be written in the above coordinates as 
		$$\Sigma_\Theta:=\Big\{(x_1,x'',x_d)\ :\ x''\in\R^{d-2},\ x_1=r\cos\theta,\ x_d=r\sin\theta\ \text{ with }\ r>0,\ \theta\in[0,\Theta]\Big\}.$$
		Let $\Theta \in (0, \pi)$ be the angle between the hyperplanes $\partial\Sigma\cap\{x_d>0\}$ and $\{x_d=0\}$; in the above notation, $\Sigma=\Sigma_\Theta$. Let $k \in\N$, $k\ge 1$, be such that $$ \frac{\pi}{k} < \Theta \le \frac{\pi}{k+1}\qquad\mbox{and let}\qquad \gamma := \Theta \frac{k+1}{\pi}. $$
		Let $T\colon \Sigma_{\frac{\pi}{k+1}} \to \Sigma_{\Theta}$ be the transformation defined by
		$$
		T:(r,x'',\theta) \mapsto (r^\gamma,x'',\gamma \theta),
		$$
		and consider the function $\overline{w}=w\circ T$ defined on $\Sigma_{\frac{\pi}{k+1}}$. By a direct computation, we get
		$$
		\begin{cases}
		\Delta_{x_1,x_d} \overline{w}= -\gamma^2 r^{2\gamma-2}\Delta_{x''} \overline{w}&\mbox{ in }\,\Sigma_{\frac{\pi}{k+1}}\cap B_1,\\
		\nu\cdot \nabla \overline{w} = 0&\mbox{ on }\partial\Sigma_{\frac{\pi}{k+1}}\cap \{x_d>0\}\cap B_1,\\
		e_{d}\cdot \nabla \overline{w} = 0&\mbox{ on }\Sigma_{\frac{\pi}{k+1}}\cap\{x_d=0\}\cap B_1,
		\end{cases}$$
		where $\nu \in \mathbb{S}^{d-1}$ is the unit normal vector to $\partial \Sigma_{\frac{\pi}{k+1}}\cap \{x_d>0\}$. Now, by reflecting symmetrically $k$ times $\Sigma_{\frac{\pi}{k+1}}$, we extend $\overline w$ to a function defined on $\{x_d\ge 0\}$. By reflecting evenly with respect to $\{x_d=0\}$, we obtain a function  $\widetilde w:B_1\to\R$ satisfying
		$$
		\Delta_{x_1,x_d} \widetilde{w}= -\gamma^2 r^{2\gamma-2}\widetilde{f}\quad\mbox{in}\quad B_1,
		$$
		%where $\widetilde{f}(x)=\Delta_{x''}\widetilde{w}\in C^{0,\alpha}(\{x_d\geq 0\})$ 
		where $\widetilde f$ is the corresponding extension of $\Delta_{x''}\overline{w}$ to $B_1$. Since 
		$$1\le \gamma<\frac{k+1}{k}\ ,\qquad \widetilde{w}\in L^\infty(B_{1})\qquad\text{and}\qquad\widetilde{f}\in L^\infty(B_{r}),$$
		for every $r<1$, by a standard bootstrap argument and the Sobolev embeddings, $\widetilde w\in C^{1,\alpha}(B_r)$ for every $r<1$ and every $\alpha\in(0,1)$.
		
		Next, we notice that we can find $R_0\in(0,1/2)$ such that 
		$$\|\widetilde{w}\|_{L^\infty(B_{2R_0})}\le 1\qquad\text{and}\qquad \gamma^2 (2R_0)^{2\gamma-2}\|\widetilde{f}\|_{L^\infty(B_{2R_0})}\leq 1.$$
		By classical elliptic regularity estimates, for every $\alpha\in(0,1)$, we can find a constant $M=M(d,R_0,\alpha)$ such that 
		$$
		|\nabla \widetilde{w}(0)|\leq M\qquad\text{and}\qquad\big|\widetilde{w}(x)-\widetilde{w}(0)-x\cdot \nabla \widetilde{w}(0)\big|\le M |x|^{1+\alpha} \quad\text{for}\quad x\in B_{R_0}\ ,
		$$
		which means 
		$$
		|\nabla \overline{w}(0)|\leq M\qquad\text{and}\qquad\big|\overline{w}(x)-\overline{w}(0)-x\cdot \nabla \overline{w}(0)\big|\le M |x|^{1+\alpha} \quad\text{for}\quad x\in \Sigma_{\frac{\pi}{k+1}}\cap B_{R_0}\ .
		$$ 
		Now, setting  
		$$\eta := \nabla \overline{w}(0),$$ 
		we have that $\eta$ is orthogonal to both $e_1$ and $e_d$ and satisfies $\abs{\eta}\leq M$.	Taking $x=(x_1,x'',x_d)$ and $r=\sqrt{x_1^2+x_d^2}$, we can rewrite the above $C^{1,\alpha}$ estimate in zero as
		$$
		x\cdot \eta - M (|x''|^2+ r^2)^{\frac{1+\alpha}{2}}\leq \overline{w}(x) \leq x\cdot \eta+ M(|x''|^2+ r^2)^{\frac{1+\alpha}{2}}\quad\text{for}\quad x\in \Sigma_{\frac{\pi}{k+1}}\cap B_{R_0}\ .
		$$
		We next change back the coordinates with $T$. Since $\eta$ is orthogonal to $e_1$ and $e_d$, we obtain
		$$
		x\cdot \eta - M (|x''|^2+ r^{2/\gamma})^{\frac{1+\alpha}{2}} \leq w(x) \leq x\cdot \eta+ M (|x''|^2+ r^{2/\gamma})^{\frac{1+\alpha}{2}}\quad \text{for}\ x\in {\Sigma}\cap B_{R_0^{1/\gamma}}\ .
		$$
		Finally, choosing $\alpha$ such that $\eps:=\frac{1+\alpha}{\gamma}-1>0$ (which is possible since 
		$\gamma<2$), we obtain 
		$$
		x\cdot \eta - M r^{1+\eps} \leq w(x) \leq x\cdot \eta+ M r^{1+\eps}\quad \text{in}\quad B_{r}\cap {\Sigma},
		$$
		for every $r\in (0,R_0^{1/\gamma})$.
		
		Up to this point, we have shown that the variational solution $w$ is differentiable with continuity in $\overline{\Sigma} \cap B_1$, together with the conditions \eqref{e:improved-estimate} and \eqref{e:condition-vector}. Thus, in order to conclude the proof we need to show that
		\[
		u = w \quad \text{on } \overline{\Sigma \cap B_1},
		\]
		which can be achieved by performing a classical viscous sliding argument.
	\end{proof}
	\begin{lemma}\label{l:IMPF}
		Let $0\in \partial\{u>0\}$ and $u$ be a viscosity solution in $B_1$ satisfying $(\mathcal I1), (\mathcal I2)$. Suppose that
		\be\label{flat1}
		h_{q,m,e_1}(x-\eps e_1)\leq u(x)\leq h_{q,m,e_1}(x+\eps e_1)\quad\text{in $B_1^+\cup B_1'$},
		\ee
		with
		$$
		q= \sqrt{Q(0)},\quad m=\beta(0).
		$$
		If $r\in (0, r_0), \eps \in (0,\eps_0)$, for some $r_0,\eps_0>0$  universal, then
		\be\label{flat_imp}
		h_{q,m,\nu}\left(x-\frac{\eps}{2}r\nu\right)\leq u(x)\leq
		h_{q,m,\nu}\left(x+\frac{\eps}{2}r\nu\right)
		\quad \text{in $B_r^+\cap B_r'$},
		\ee
		with $\nu$ a unit vector orthogonal to $e_d$, such that $|\nu -e_1|\leq C\eps,$ for some universal constant $C>0.$
	\end{lemma}
	\begin{proof}
		The proof follows the classical strategy introduced in \cite{desilva} for the one-phase problem.	
		Let 
		$$L(x):=\sqrt{q^2-m^2}\,(x\cdot e_1) + m\,(x\cdot e_d),$$
		so we have $h_{q,m,e_1}=L^+$, and set $\Sigma=\{h_{q,m,e_1}>0\}$.\medskip
		
		\noindent{\it Step 1 - Compactness.} Set $r \in (0,r_0)$, with $r_0>0$ to be made precise later. Assume by contradiction that there exist $\eps_k \to 0$ such that the following holds true: there exist sequences $A_k, Q_k, \beta_k$ satisfying $(\mathcal I1),(\mathcal I2)$, such that the viscosity solutions $u_k$ of \eqref{e:interno} and \eqref{FB} satisfy $0 \in \partial\{u_k>0\}$ and
		\be\label{contrad1}
		h_{q,m,e_1}(x - \eps_k e_1)\leq u_k(x)\leq h_{q,m,e_1}(x+ \eps_k e_1)\quad \text{in $B_1^+\cap B_1'$},
		\ee
		but the conclusions \eqref{flat_imp} does not hold. Therefore, consider 
		\be\label{tilde}
		\tilde{u}_k(x) := \frac{1}{\sqrt{q^2-m^2}}\frac{u_k(x) - L(x) }{\eps_k}, \qquad \tilde u_k:(B_1^+\cup B_1')\cap \overline{\{u_k>0\}}\to0.
		\ee
		By the flatness assumptions \eqref{contrad1} the sequence $u_k$ is uniformly bounded in $B_1$ and the free boundaries $\partial\{u_k>0\}$ converge to $\partial\Sigma$ in the Hausdorff distance. By \cref{cor:AA}, we have that, up to a subsequence, the graphs of $\tilde{u}_k$ over $(B_{\sfrac12}^+ \cup B_{\sfrac12}')\cap \overline{\{u_k>0\}})$ converge with respect to the Hausdorff distance to the graph of a H\"{o}lder continuous functions $\tilde u$ defined on $(B_{\sfrac12}^+\cup B_{\sfrac12}')\cap \overline{\Sigma}$.	\medskip
		
		\noindent{\it Step 2 - Linearized problem.} We will prove that $\tilde u$ is a viscosity solution (in the sense of \cref{def:limitPbSolutionViscous}) to
		\be
		\Delta \tilde u=0 \quad\text{in $B_{\sfrac12}^+ \cap \Sigma$},\,\quad
		\nabla \tilde u \cdot e_d= 0 \quad\text{on $B_{\sfrac12}'\cap \Sigma$},\,\quad
		\nabla \tilde{u}\cdot \nabla L = 0 \quad\text{on $B_{\sfrac12}^+ \cap \partial \Sigma $}.
		\ee
		Suppose that $P$ is a polynomial  touching $\tilde{u}$ at $\overline{x} \in (B_{1/2}^+\cup B_{1/2}') \cap \overline{\Sigma}$ strictly from below (the case when $P$ touches $\tilde u$ from above is analogous), that is, there is $r>0$ such that
		$$P(\bar x)=\tilde u(\bar x)\qquad\text{and}\qquad \min_{\partial B_r(\bar x)\cap \overline\Sigma}(\tilde u-P)\ge\delta,$$
		for some $\delta>0$.
		In view of \cref{def:limitPbSolutionViscous}, we need to show that:
		\begin{enumerate}[label=(\roman*)]
			\item if $\overline{x}\in B_{\sfrac12}^+\cap \Sigma$, then $\Delta P(\overline{x}) \leq 0$;
			\item if $\overline{x}\in B_{\sfrac12}^+\cap \partial \Sigma$, then $\nabla P(\overline{x})\cdot \nabla L\leq 0$;
			\item if $\overline{x}\in B_{\sfrac12}'\cap \Sigma$, then $\partial_{x_d}P(\overline{x}) \leq 0$;
			\item if $\overline{x}\in B_{\sfrac12}'\cap \partial\Sigma$, then $\partial_{x_d}P(\overline{x}) \leq 0$ or $\nabla P(\overline{x})\cdot \nabla L\leq 0$.
		\end{enumerate}
		For every $k\ge 1$, there exists $x_k \in (B_{\sfrac12}^+\cup B_{\sfrac12}') \cap \overline{\{u_k>0\}})$ and $c_k\in\R$, such that $P+c_k$ touches $\tilde{u}_k$ from below at $x_k$. Moreover, by the Hausdorff convergence of the graphs of $\tilde u_k$, we have that 
		$$P(x_k)+c_k=\frac{u_k(x_k)-L(x_k)}{\eps_k\sqrt{q^2-m^2}}\qquad\text{and}\qquad \min_{\partial B_r(\bar x)\cap \overline{\{u_k>0\}}}\bigg\{\frac{u_k-L}{\eps_k \sqrt{q^2-m^2}}-P-c_k\bigg\}\ge\frac{\delta}{2}.$$
		Consider the solutions $\tilde{P}_k$ and $L_k$ to 
		$$
		\begin{cases}
		\text{\rm div}\big(A_k\nabla L_k\big)=0 & \mbox{in } B_r(\bar x) \\
		L_k=L & \mbox{on }\partial B_r(\bar x),
		\end{cases}\qquad
		\begin{cases}
		\text{\rm div}\big(A_k\nabla \tilde P_k\big)=\Delta P & \mbox{in } B_r(\bar x) \\
		\tilde P_k=P & \mbox{on }\partial B_r(\bar x).
		\end{cases}
		$$
		Then, since by hypothesis $\|A_k - \text{\rm Id}\|_{C^{0, \alpha}\left(B_r\left( \overline{x}\right)\right)}<<\eps_k$, we get that 
		$$\frac\delta4+\tilde P_k(x_k)+c_k>\frac{u_k(x_k)-L_k(x_k)}{\eps_k \sqrt{q^2-m^2}}\qquad\text{and}\qquad \min_{\partial B_r(\bar x)\cap \overline{\{u_k>0\}}}\bigg\{\frac{u_k-L_k}{\eps_k \sqrt{q^2-m^2}}-\tilde P_k-c_k\bigg\}\ge\frac{\delta}{4}.$$
		This, in particular means that $\tilde P_k+c_k$ touches from below $\eps_k^{-1}(u_k-L_k)$ at a point $y_k\in \overline{\{u_k>0\}}\cap B_r(\bar x)$. 
		Therefore, if we set
		$$P_k(x):= L_k(x) +\eps_k \sqrt{q^2-m^2}(\tilde P_k(x)+c_k),$$
		we get that $P_k$ touches $u_k$ from below at $y_k$. \medskip
		
		We now consider the different cases (i), (ii), (iii) and (iv). The cases (i) and (ii) were already discussed in \cite{DeSilvaFerrariSalsa:2PhaseFreeBdDivergenceForm} and \cite{FerreriVelichkov2023:discontinuousBernoulli}. The case (iii) follows by the Schauder estimates in $B_r(\overline x)$ and by using again the assumption $\|A_k -\text{\rm Id}\|_{C^{0, \alpha}\left(B_r\left( \overline{x}\right)\right)}<<\eps_k$. It only remains to treat the case (iv). Indeed, suppose that $P$ touches $\tilde u$ from below at $\bar x\in\partial\Sigma\cap\{x_d=0\}$ and is such that 
		$$e_d\cdot\nabla P(\bar x)>0\qquad\text{and}\qquad \nabla L\cdot \nabla P(\bar x)>0.$$
		Now, for the point $x_k$, we have the following possibilities:
		\begin{itemize}
			\item $x_k\in B_{\sfrac12}^+\cap\{u_k>0\}$. This case leads to contradiction proceeding as in (i) (see \cite{FerreriVelichkov2023:discontinuousBernoulli} or \cite{DeSilvaFerrariSalsa:2PhaseFreeBdDivergenceForm}). 
			%Indeed, since we can suppose that $\Delta P$ is a strictly positive constant and so, also $\text{div}(A_k\nabla P_k)>0$.
			\item $x_k\in B_{\sfrac12}^+\cap\partial\{u_k>0\}$. This case follows as in (ii) (we refer to \cite{FerreriVelichkov2023:discontinuousBernoulli}).  
			\item $x_k\in B_{\sfrac12}'\cap\{u_k>0\}$. This case follows as (iii) by a straightforward adaptation of (ii).
			\item $x_k\in B_{\sfrac12}'\cap\partial\{u_k>0\}$. This is not possible since 
			$$e_d\cdot\nabla P_k(x_k)>0\qquad\text{and}\qquad \nabla L\cdot \nabla P_k(x_k)>0,$$
			and $u_k$ is a viscosity solution in the sense of \cref{def:solutionnew}. 
		\end{itemize}
		
		\noindent{\it Step 3 - Improvement-of-flatness.} We already know that $|\tilde{u}|\leq 1$ in $B_1^+\cup B_1'$ and satisfies $\tilde{u}(0)=0$. First, by \cref{l:linearized.regularity}, we know that there exists $\alpha \in (0,1)$
		$$
		(x\cdot \eta)-M r^{1+\alpha}\leq \tilde{u}(x)\leq (x\cdot \eta) + M r^{1+\alpha} \quad\mbox{in }(B_{r}^+\cup B_r')\cap \overline{\Sigma},
		$$
		for every $r \in (0,1/4)$, where $\eta\in \R^d\setminus \{0\}$ is a vector satisfying
		$$
		\eta \cdot e_d = \eta \cdot e_1 = 0,\quad |\eta|\leq M
		$$
		with $M>0$ a universal constant. Thus, for $k$ sufficiently large there exists $c_1>0$ such that
		$$
		(x\cdot \eta) - c_1 r^{1+\alpha} \leq \tilde{u}_k(x) \leq (x\cdot \eta) + c_1 r^{1+\alpha} \quad \text{in $(B_r^+\cup B'_r) \cap \overline{\{u_k > 0\}}$},
		$$
		and exploiting the definition of $\tilde u_k$, we read
		$$
		\eps_k \sqrt{q^2-m^2}\left((x\cdot \eta) -  c_1 r^{1+\alpha}\right) \leq u_k(x) - L(x) \leq \eps_k \sqrt{q^2-m^2}\left((x\cdot \eta) + c_1 r^{1+\alpha}\right)
		$$
		in $(B_r^+\cup B'_r) \cap \overline{\{u_k > 0\}}$. Since  
		$$1\leq \sqrt{1+\abs{\eta}^2\eps_k^2} \leq 1+ M^2{\eps_k^2}, $$
		by setting
		$$
		\nu = \frac{e_1+\eps_k \eta}{\sqrt{1+\eps_k^2\abs{\eta}^2}} \in \mathbb{S}^{d-1}\cap \{x_d=0\},
		$$
		we get
		\begin{align*}
		\left\{
		\begin{aligned}
		\sqrt{q^2-m^2}\left((x\cdot \nu) -c_1 \eps_k r^{1+\alpha} - M^2{\eps_k^2}r\right) + m(x\cdot e_d) &\leq u_k(x),\\
		\sqrt{q^2-m^2}\left((x\cdot \nu) +c_1 \eps_k r^{1+\alpha} + M^2{\eps_k^2}r\right)  + m(x\cdot e_d)  &\geq u_k(x).
		\end{aligned}
		\right.
		\end{align*}
		Finally, by taking $r \leq r_0$ such that $c_1 r_0 \leq 1/4$ and $k$ large enough so that $M^2\eps_k \leq 1/4$, we obtain
		\be\label{improv.1}
		h_{q,m,\nu}\left(x-\frac{\eps_k}{2}r\nu\right)
		\leq u_k(x) \leq h_{q,m,\nu}\left(x+\frac{\eps_k}{2}r\nu\right)\quad\text{in $(B_r^+\cup B_r')\cap \overline{\{u_k>0\}}$},
		\ee
		that is $u_k$ satisfies the conclusion of the lemma, and we reached the contradiction.
	\end{proof}
	\section{Stable homogeneous solutions}\label{s:stability.sect}
	In this section we prove that there are no stable singular cones in dimensions $d \in \{2,3,4\}$. We prove the Caffarelli-Jerison-Kenig-type stability inequality  from \cref{t:stability-inequality-CJK-type} for one-homogeneous non-negative global minimizers of the functional
	$$
	J(v,E):=\int_{E^+}|\nabla v|^2\,\mathrm{d}x + q^2\big|\{v>0\}\cap E^+\big|+
	2\,m\int_{E'}\,|v| \,\mathrm{d}\HH^{d-1},
	$$
	with $q \in \R$ and $m\in (-q,q)$. Then, by combining the analysis carried out in \cite{js} with the stability inequality of \cref{t:stability-inequality-CJK-type}, we provide the main result concerning existence of global minimizer with isolated singularity (\cref{t:stability}).
	\subsection{Taylor expansion of solutions to PDEs with non-homogeneous Neumann conditions}\label{sub:first-and-second-variation-of-EL}$ $\\
	We start by deriving a second order expansion of solutions to the Euler-Lagrange equations associated to the functional $J$, along smooth internal variations. We follow step-by-step the proof from \cite{BMMTV}, the only difference is that here in the functional $J$ we also have the boundary term $\int_{E'}|v|\,d\HH^{d-1}$. 
	
	Through the section, we denote by $\R^{d\times d}$ the space of $d\times d$ square matrices with real coefficients, and for any real matrix $B=(b_{ij})_{ij}\in \R^{d\times d}$, we define the norm
	$$\|B\|_{\R^{d\times d}}:=\bigg(\sum_{i=1}^d\sum_{j=1}^db_{ij}^2\bigg)^{\sfrac12}.$$
	Given a matrix $B:\O\to\R^{d\times d}$ with coefficients $b_{ij}:\O\to\R$, defined on a measurable set $\Omega\subset\R^d$, we say that $B\in L^\infty(\O;\R^{d\times d}),$
	if $b_{ij}\in L^\infty(\O)$ for every $1\le i,j\le d$. We also define the norm 
	$$\|B\|_{L^\infty(\O;\R^{d\times d})}:=\big\|\|B\|_{\R^{d\times d}}\big\|_{L^\infty(\O)}.$$
	\begin{lemma}[Second order expansion of solutions]\label{l:abstact-second-order-expansion}
		Let $\O$ be a bounded open set in $\R^d$ and let
		$$m:\R\to L^\infty(\O')\quad\text{and}\quad B:\R\to L^\infty(\O;\R^{d\times d}),$$
		where as usual $\Omega'=\Omega\cap\{x_d=0\}$, be such that:
		\begin{enumerate}[\quad\rm(a)]
			\item $B_t(x)$ is a symmetric matrix for every $(t,x)\in\R\times\O$ and there are symmetric matrices\\ $\delta B\in L^\infty(\O;\R^{d\times d})$ and $\delta^2B\in L^\infty(\O;\R^{d\times d})$ such that
			$$B_t=\text{\rm Id}+t\,\delta B+t^2\delta^2B+o(t^2)\quad\text{in }L^\infty(\O;\R^{d\times d})\,;$$
			\item there are functions $\delta m\in L^\infty(\O')$ and $\delta^2 m\in L^\infty(\O')$ such that
			$$m_t=m_0+t\,\delta m+t^2\delta^2m+o(t^2)\quad\text{in }L^\infty(\O').$$
		\end{enumerate}
		Then, given $u_0 \in H^1(\O^+)$, for every $t$ small enough there is a unique solution $u_t$ to the problem
		\be\label{e:abstract-equation-u-t}
		\dive(B_t\nabla u_t)=0\quad\text{in }\O^+,\qquad
		e_d \cdot B_t\nabla u_t = m_t \quad\text{on }\O',\qquad
		u_t=u_0\quad\mbox{on }(\partial \O)^+.
		\ee
		Moreover, $$
		u_t=u_0+t\,\delta u+t^2\delta^2u+o(t^2)\quad\text{in }H^1(\O),
		$$
		where $\delta u, \delta^2 u \in H^1(\O^+)$ are the unique weak solutions to
		\be\label{e:abstract-equations-delta-u}
		\begin{cases}
			-\Delta(\delta u)=\dive((\delta B)\nabla u_0) &\text{in }\O^+,\\
			e_d \cdot \nabla (\delta u) + e_d \cdot \delta B\nabla u_0 =\delta m &\text{on }\O',\\
			\delta u = 0 &\text{on }(\partial\O)^+\,,
		\end{cases}
		\ee
		and
		\be\label{e:abstract-equations-delta-2-u}
		\begin{cases}
			-\Delta (\delta^2 u)=\dive((\delta B)\nabla (\delta u))+\dive((\delta^2 B)\nabla u_0)&\text{in }\O\,,\\
			e_d\cdot \nabla(\delta^2 u) + e_d \cdot\left(\delta B \nabla(\delta u) + \delta^2B\nabla u_0\right) = \delta^2 m  &\text{in }\O'\,,\\
			\delta^2 u = 0 &\text{on }(\partial\O)^+.
		\end{cases}
		\ee
	\end{lemma}
	\begin{proof}
		For the sake of simplicity, we will use the following notation: by  $\mathcal H_0^+(\Omega)$  we will denote the completion of $C^\infty_c(\Omega)$ with respect to the norm $H^1(\Omega^+)$; that is, the functions in $\mathcal H_0^+(\Omega)$ are in $H^1(\Omega^+)$ and are zero on $\partial\Omega\cap\{x_d>0\}$. 
		We split the proof in two parts.
		
		\noindent \it Part 1. \rm Given $w_t:=\frac1t(u_t-u_0)\in \mathcal H_0^+(\Omega)$, we prove that $w_t$ converges to $\delta u$ strongly in $\mathcal H_0^+(\Omega)$. We notice that \eqref{e:abstract-equation-u-t} can be written as
		$$
		\begin{cases}
		\dive((\text{\rm Id}+(B_t-\text{\rm Id}))\nabla (u_0+tw_t))=0&\text{in }\O^+\\
		e_d \cdot (\text{\rm Id} + (B_t-\text{\rm Id}))\nabla (u_0+t w_t) = (m_0+(m_t - m_0)) &\text{on }\O'.
		\end{cases}
		$$
		So, using the equation for $u_0$ and dividing by $t$, we get
		\be\label{e:s3-abstract-equation-w-t}
		\begin{cases}
			-\Delta w_t-\dive\Big(\frac1t(B_t-\text{\rm Id})\nabla u_0\Big)-\dive\Big((B_t-\text{\rm Id})\nabla w_t\Big)=0&\text{in }\O^+\\
			e_d \cdot \nabla w_t + e_d \cdot \frac1t(B_t-\text{\rm Id})\nabla u_0 + e_d \cdot (B_t - \text{\rm Id})\nabla w_t= \frac1t(m_t-m_0) &\text{on }\O'.
		\end{cases}
		\ee
		If we fix $\eps>0$, we can choose $t$ small enough such that
		$$\|B_t-\text{\rm Id}\|_{L^\infty(\O;\R^{d\times d})} + \left\|\frac1{t}(B_t-\text{\rm Id})-\delta B\right\|_{L^\infty(\O;\R^{d\times d})} +
		\left\|\frac1{t}(m_t-m_0)-\delta m\right\|_{L^\infty(\O')}\le \eps.$$
		By testing \eqref{e:s3-abstract-equation-w-t} with $w_t$, we obtain
		\begin{align*}
		\int_{\O^+}|\nabla w_t|^2\,dx=&-\int_{\O^+}\nabla w_t\cdot \frac1t(B_t-\text{\rm Id})\nabla u_0\,dx-\int_{\O^+}\nabla w_t\cdot (B_t-\text{\rm Id})\nabla w_t\,dx- \int_{\O'}w_t \frac{m_t-m_0}{t}\,dx' \\
		\le&\Big(\eps+\|\delta B\|_{L^\infty(\O;\R^{d\times d})}\Big)\|\nabla w_t\|_{L^2 (\O)}\|\nabla u_0\|_{L^2 (\O)}+\eps\|\nabla w_t\|_{L^2(\O^+)}^2\\
		&\qquad + \left(\eps + \norm{\delta m}{L^\infty(\O')}\right)\Big(\HH^{d-1}(\Omega')\Big)^{1/2}\norm{w_t}{L^2(\O')}.
		\end{align*}
		We notice that the following trace inequality holds on $\mathcal H_0^+(\Omega)$
		$$\|\varphi\|_{L^2(\O')}^2\le C_d\|\nabla \varphi\|_{L^2(\O^+)}^2\quad\text{for every}\quad \varphi\in \mathcal H_0^+(\Omega)\,,$$
		where $C_d>0$ is a dimensional constant. Then, 
		$$
		\|\nabla w_t\|_{L^2(\O^+)}\le  \frac{\eps+\|\delta B\|_{L^\infty(\O;\R^{d\times d})}}{1-\eps}\|\nabla u_0\|_{L^2(\O^+)} + C_d\frac{\eps+\norm{\delta m}{L^\infty(\O')}}{1-\eps}\Big(\HH^{d-1}(\Omega')\Big)^{1/2}.
		$$
		Thus, for every sequence $t_n\to0$, there is a subsequence for which $w_{t_n}$ converges as $n\to \infty$, strongly in $L^2(\O^+)$ and weakly in $H^1(\Omega^+)$, to some function $w_\infty$. Passing to the limit the equation \eqref{e:s3-abstract-equation-w-t} we get that $w_\infty$ is also a solution to \eqref{e:abstract-equations-delta-u}. Thus $w_\infty=\delta u$. In particular, this implies that $w_t$ converges, as $t\to0$, strongly in $L^2(\O^+)$ and $L^2(\Omega')$, and weakly in $H^1(\Omega^+)$ to $\delta u$. Finally, in order to prove that the convergence is strong, we test again \eqref{e:s3-abstract-equation-w-t} with $w_t$:
		$$
		\begin{aligned}
		\limsup_{t\to\infty}\int_{\O^+}|\nabla w_t|^2\,dx
		=& -\lim_{t\to0}\int_{\O^+}\nabla w_t\cdot \frac1t(B_t-\text{\rm Id})\nabla u_0\,dx- \int_{\O'}\frac1t(m_t-m_0)w_t\,dx' \\
		=&-\int_{\O^+}\nabla (\delta u)\cdot \delta B\nabla u_0\,dx
		-\int_{\O'} (\delta u)(\delta m)\,dx' \\
		=&\int_{\O^+}|\nabla (\delta u)|^2\,dx\,.
		\end{aligned}
		$$
		Combining this estimate with the lower semi-continuity of the $H^1$ norm, we get
		$$\lim_{t\to\infty}\int_{\O^+}|\nabla w_t|^2\,dx=\int_{\O^+}|\nabla (\delta u)|^2\,dx\,$$
		which implies that $w_t$ converges to $\delta u\in \mathcal H_0^+(\Omega)$ strongly also with respect to the $H^1(\Omega^+)$ norm.\medskip
		
		\noindent \it Part 2. \rm Given $v_t:=\frac1t(w_t-\delta u)\in H^1_0(\O^+),$ we prove that $v_t$ converges strongly in $H^1_0(\O^+)$ to $\delta^2u$. First, from equation \eqref{e:s3-abstract-equation-w-t} for $w_t$, we have
		$$
		\begin{cases}
		-\Delta (\delta u+tv_t) -\dive\Big(\frac1t(B_t-\text{\rm Id})\nabla u_0\Big)-\dive\Big((B_t-\text{\rm Id})\nabla (\delta u+tv_t)\Big)=0&\text{in }\O^+\\
		e_d \cdot \nabla (\delta u+tv_t) + e_d \cdot \frac1t(B_t-\text{\rm Id})\nabla u_0 + e_d \cdot (B_t - \text{\rm Id})\nabla (\delta u+tv_t)= \frac1t(m_t-m_0) &\text{on }\O'.
		\end{cases}
		$$
		Thus, using the equation \eqref{e:abstract-equations-delta-u} for $\delta u$, we get
		$$
		\begin{cases}
		-\Delta v_t-\dive\Big(\frac1{t^2}(B_t-\text{\rm Id}-t\delta B)\nabla u_0\Big)-\dive\Big(\frac1t(B_t-\text{\rm Id})\nabla(\delta u+tv_t)\Big)=0&\text{in }\O^+\\
		e_d \cdot \nabla v_t + e_d \cdot \frac{1}{t^2}(B_t-\text{\rm Id}-t\delta B)\nabla u_0 + e_d \cdot \frac1t(B_t - \text{\rm Id})\nabla (\delta u+tv_t)= \frac{1}{t^2}(m_t-m_0-t\delta m) &\text{on }\O'.
		\end{cases}
		$$
		Now, the proof follows by reasoning as in Part 1 and using that
		$$ \left\|\frac{1}{t^2}(B_t-\text{\rm Id}-t\delta B)-\delta^2 B\right\|_{L^\infty(\O;\R^{d\times d})} +
		\left\|\frac{1}{t^2}(m_t-m_0-t \delta m)-\delta^2 m\right\|_{L^\infty(\O')}\le \eps,$$
		for $t$ sufficiently small.
	\end{proof}
	We next apply the abstract expansion from \cref{l:abstact-second-order-expansion} to the specific case induced from the variation of the original domain $\Omega$ along vector fields.  
	\begin{proposition}\label{l:first-and-second-variation-along-a-flow}
		Let $D$ be a bounded open set in $\R^d$, $\O\subset D$ be an open set, and $\eta \in C^\infty_c(D;\R^d)$ be a compactly supported vector field satisfying the following assumption:
		\be\label{e:assumption-eta}
		\eta(x',0)\perp e_d\quad \mbox{for every}\quad (x',0) \in D':=D\cap\{x_d=0\}.
		\ee
		Let $\Phi:\R\times\R^d\to\R^d$ be the flow associated to $\eta$. Precisely, for every $x\in D$, $\Phi_t(x)$ solves the ODE
		\be\label{e:flow-of-eta}
		\begin{cases}
			\partial_t\Phi_t(x)=\eta\big(\Phi_t(x)\big)\quad\text{for every }t\in\R\\
			\Phi_0(x)=x.
		\end{cases}
		\ee
		We consider	the family of open sets $\O_t := \Phi_t(\O)$ and the associated state variable $u_t$ satisfying $u_t=u_0$ on $(\partial D)^+$ and
		$$
		\Delta u_t=0 \quad\mbox{in } \O_t\cap D^+,\qquad
		u_t=0 \quad\mbox{on } \partial \O_t \cap D^+,\qquad
		e_d \cdot \nabla u_t = m \quad\mbox{on }\O_t \cap D'
		$$
		for some $m \in \R$.
		We define $\delta u_\O$ and $\delta^2 u_\O$ to be the weak solutions to the PDEs
		\be\label{e:deltau}
		\begin{cases}
			-\Delta(\delta u)=\dive((\delta B)\nabla u_0) &\text{in }\O^+,\\
			e_d \cdot \nabla (\delta u) + e_d \cdot \delta B\nabla u_0 =\delta m &\text{on }\O',\\
			\delta u = 0 &\text{on }(\partial\O)^+\,,
		\end{cases}
		\ee
		and
		\be\label{e:delta2u}
		\begin{cases}
			-\Delta (\delta^2 u)=\dive((\delta B)\nabla (\delta u))+\dive((\delta^2 B)\nabla u_0)&\text{in }\O\,,\\
			e_d\cdot \nabla(\delta^2 u) + e_d \cdot\left(\delta B \nabla(\delta u) + \nabla(\delta^2 u)\right) = \delta^2 m  &\text{in }\O'\,,\\
			\delta^2 u = 0 &\text{on }(\partial\O)^+,
		\end{cases}
		\ee
		where the matrices $\delta B\in L^\infty(D;\R^{d\times d})$ and $\delta^2B\in L^\infty(D;\R^{d\times d})$ are given by
		\be\label{e:first-and-second-variation-B}
		\begin{aligned}
			\delta B &:=-D\eta-\nabla\eta+(\dive\eta)\text{\rm Id}\,,\\
			\delta^2 B &:= (D\eta)\,(\nabla\eta)+\frac12\big(\nabla\eta\big)^2+\frac12(D\eta)^2-\frac12(\eta\cdot\nabla)\big[\nabla\eta+D\eta\big]\\
			&\qquad-\big(\nabla\eta+D\eta\big)\dive\eta+\text{\rm Id}\frac{(\dive\eta)^2+\eta\cdot\nabla(\dive\eta)}{2}\,,
		\end{aligned}
		\ee
		while the variations $\delta m\in L^\infty(D)$ and $\delta^2 m\in L^\infty(D)$ of the Neumann condition $m$ are defined as
		\be\label{e:first-and-second-variation-m}
		\delta m :=m\dive \eta\ ,\qquad
		\delta^2 m := m\frac{(\dive\eta)^2+\eta\cdot\nabla[\dive\,\eta]}{2} \,.
		\ee
		Then,
		\[
		u_t\circ \Phi_t = u_0 + t (\delta u) + t^2 (\delta^2 u) + o(t^2)\quad\mbox{in}\quad H^1(\Omega^+)\,.
		\]
	\end{proposition}
	\begin{proof}
		Since the proof coincides, up to minor changes, with the one of \cite[Proposition 2.5]{BMMTV}, we simply sketch the main steps. Setting $\tilde{u}_t:=u_{t}\circ \Phi_t$, we have that $\tilde{u}_t\in H^1(\O^+)$ and $u_{t}=\tilde{u}_t\circ \Phi_t^{-1}$. Moreover, $\tilde{u}_t$ is zero on $\partial\Omega\cap D^+$, that is, following the notation from the proof of \cref{l:abstact-second-order-expansion}, $\tilde u_t\in\mathcal H^+_0(\Omega)$. By a change of variables, we have that $\tilde{u}_t$ is satisfies the system in \eqref{e:abstract-equation-u-t} where the matrix $B_t$ and the function $m_t$ are defined as
		\be\label{e:Bm}
		B_t:=(D\Phi_t)^{-1}(D\Phi_t)^{-T}|\text{\rm det}(D\Phi_t)|,\qquad m_t:= m|\text{\rm det}(D\Phi_t)|.
		\ee
		Now, by differentiating the equation for the flow $\Phi_t$, we get $$D\Phi_t=\text{\rm Id}+tD\eta+\frac{t^2}{2}\Big((\eta\cdot\nabla)[D\eta]+(D\eta)^2\Big)+o(t^2).$$ Finally, by exploiting the computations in \cite[Remark 2.4]{BMMTV}, we get 
		\[
		\begin{aligned}
		B_t&=\text{\rm Id}+ t (\delta B) + t^2 (\delta^2 B) + o(t^2)\quad\text{in }L^\infty(D;\R^{d\times d})\,,\\
		m_t&= m + t (\delta m) + t^2 (\delta^2 m) + o(t^2)\quad\text{in }L^\infty(D)\,,
		\end{aligned}
		\]
		where $\delta B$, $\delta^2B$, $\delta m$, $\delta^2m$ are given by \eqref{e:first-and-second-variation-B} and \eqref{e:first-and-second-variation-m}. Thus, the claim follows from \cref{l:abstact-second-order-expansion}.
	\end{proof}
	\subsection{First and second variation of $J$}\label{sub:first-and-second-variation-of-J} We next use the Taylor expansion from \cref{l:first-and-second-variation-along-a-flow} in order to compute the first derivative of the functional $J$ along inner variations with compact support in $D$.
	
	\begin{lemma}[First variation of $ J$ along vector fields]\label{l:firstv}
		Let $D$ be a bounded open set in $\R^d, \O\subset D$ be open and $\eta\in C^\infty_c(D;\R^d)$ be a vector field with compact support in $D$ satisfying \eqref{e:assumption-eta}. Let $\Phi_t$ be the flow of the vector field $\eta$ defined by \eqref{e:flow-of-eta} and set $\O_t := \Phi_t(\O)$ and $u_t$ to be  as in \cref{l:first-and-second-variation-along-a-flow}.  Then
		\be\label{e:first-derivative-J-along-vector-field}
		\frac{\partial}{\partial t}\bigg|_{t=0} J(u_t,D)= \int_{\O\cap D^+} \dive\eta\left(|\nabla u|^2 +q^2\right)- 2 \nabla u \cdot D\eta \nabla u\,dx +2m\int_{\O\cap D'} u\dive\eta\, dx'.
		\ee
		Moreover, if $\partial\O\cap D^+$ is $C^2$-regular in a neighborhood of the support of $\eta$, then
		\be\label{e:statiosmooth}
		\frac{\partial}{\partial t}\bigg|_{t=0} J(u_t,D) = \int_{\partial\O\cap D^+} (\nu \cdot \eta)(q^2-|\nabla u|^2)\, d\HH^{d-1},
		\ee
		where $\nu$ is the outer unit normal to $\partial\O$.
	\end{lemma}
	\begin{proof}
		By applying \cref{l:first-and-second-variation-along-a-flow} to $\tilde{u}_t=u_t\circ\Phi_t$ we get that
		$\tilde{u}_t=u+t(\delta u)+o(t)$, where $\delta u$ is the solution to \eqref{e:deltau}. Therefore, setting $B_t$ and $m_t$ as in \eqref{e:Bm} and $Q_t=q^2|\text{\rm det}(D\Phi_t)|$, we get
		\be
		\begin{aligned}
			J(u_t,D)=&\,\int_{\O_t\cap D^+}|\nabla u_t|^2 +q^2 \,dz + 2m\int_{\O_t \cap D'} u_t\,dz'\\
			=&\int_{\O\cap D^+}\Big(\nabla \tilde{u}_t\cdot B_t \nabla \tilde{u}_t+ Q_t\Big)\,dx + 2\int_{\O \cap D'}m_t \tilde{u}_t\,dx'\\
			=& J(u,D) + t \int_{\O\cap D^+}2\nabla (\delta u) \cdot \nabla u + \nabla u\cdot(\delta B)\nabla u + q^2\dive\eta \,dx\\
			&\quad+t\int_{\O\cap D'}2m \left(u \dive \eta + \delta u\right)\,dx'+o(t)\\
			=& J(u,D) + t \int_{\O\cap D^+} \Big(\nabla u\cdot (\delta B)\nabla u + q^2\dive\eta\Big)\,dx + 2m\int_{\O\cap D'}u\dive\eta\,dx' + o(t)
		\end{aligned}
		\ee
		where in the first equality we applied the change of variables $z=\Phi_t(x)$ and in the last one we use the harmonicity of $u$ in $\O\cap D^+$. Substituting with the expression for $\delta B$ from \eqref{e:first-and-second-variation-B}, we obtain \eqref{e:first-derivative-J-along-vector-field}. Finally, \eqref{e:statiosmooth} follows, as in \cite{BMMTV}, by an integration by parts.
	\end{proof}
	
	\begin{proposition}[Second variation of $ J$ along vector fields]\label{l:secondstv}
		Consider a bounded open set $D$ in $\R^d$ and an open set $\O\subset D$. Let $\eta\in C^\infty_c(D;\R^d)$ be a vector field with compact support in $D$ satisfying \eqref{e:assumption-eta}. Let $\Phi_t$ be the flow of the vector field $\eta$ defined by \eqref{e:flow-of-eta} and set $\O_t := \Phi_t(\O)$ and $u_t$ as in \cref{l:first-and-second-variation-along-a-flow}. Then
		\be\label{e:second-derivative-J-along-vector-field}
		\frac12	 \frac{\partial^2}{\partial t^2}\bigg|_{t=0} J(u_t,D) = \int_{\O\cap D^+} \nabla u\cdot(\delta^2 B)\nabla u -|\nabla (\delta u)|^2 + \delta^2 Q\, dx + 2\int_{\O\cap D'}u\delta^2 m  \,dx'
		\ee
		where $\delta^2 B, \delta^2 m, \delta u$  are the ones defined in \cref{l:first-and-second-variation-along-a-flow} and
		\be\label{e:delta2Q}\delta^2 Q:= \frac12 q^2 \dive\left(\eta(\dive\eta)\right).\ee
	\end{proposition}
	
	\begin{proof}
		By applying \cref{l:first-and-second-variation-along-a-flow} to $\tilde{u}_t:=u_t\circ \Phi_t$, we get that
		$$
		\tilde{u}_t = u + t (\delta u) + t^2 (\delta^2 u)+o(t^2)
		$$
		with $\delta u$ and $\delta u^2$ satisfying \eqref{e:deltau} and \eqref{e:delta2u}. Hence, by computing the second order expansion in $t$ of
		\begin{align*}
		J(u_t,D) =\, \int_{\O\cap D^+}\Big(\nabla \tilde{u}_t\cdot B_t \nabla \tilde{u}_t+ Q_t\Big)\,dx + 2\int_{\O \cap D'}m_t \tilde{u}_t\,dx'\,,
		%&= J(\O,D)+t\delta J(\O,D)[\eta]+t^2 \partial^2 J(\O,D)[\eta]+o(t^2)
		\end{align*}
		we get that
		\begin{align*}
		\frac12	 \frac{\partial^2}{\partial t^2}\bigg|_{t=0} J(u_t,D) =&\, \int_{\O\cap D^+} 2\nabla (\delta^2 u )\cdot \nabla u +\nabla u\cdot(\delta^2 B)\nabla u +|\nabla (\delta u)|^2+2\nabla (\delta u)\cdot (\delta B)\nabla u\, dx\\
		&+\int_{\O\cap D^+}\delta^2 Q\,dx + 2\int_{\O\cap D'}u\delta^2 m  + m\delta^2 u + \delta m\delta u\,dx'.
		\end{align*}
		Now, by using $\delta u$ as a test function in the equation for $\delta u$ we obtain
		$$\int_{\O\cap D^+}\nabla (\delta u)\cdot (\delta B)\nabla u\,dx=-\int_{\O\cap D^+} |\nabla (\delta u)|^2\,dx - \int_{\O\cap D'}\delta u\,\delta m\,dx'.$$
		Then, by testing the equations for $u$ with $\delta^2 u$, we get
		$$\int_{\O\cap D^+}\nabla (\delta^2 u )\cdot \nabla u\,dx=-m\int_{\O\cap D'} \delta^2 u\,dx'.$$
		Using this identities in the expression of the second derivative, we get precisely \eqref{e:second-derivative-J-along-vector-field}.
	\end{proof}

	\subsection{Proof of \cref{t:stability-inequality-CJK-type}}
	\noindent Let $u:\R^d\cap \{x_d\geq0\}\to\R$ be a non-negative $1$-homogeneous global minimizer of the functional $J$ and set $\O:=\{u>0\}$. Suppose that the free boundary $\partial\Omega$ is smooth away from the origin, 
	%$$\partial\Omega\cap\Big(\R^{d-1} \setminus \{0\}\Big)\subset \text{\rm Reg}(u).$$
	that is, $\partial\O\cap \{x_d>0\}$ is a $C^\infty$ smooth $(d-1)$-dimensional manifold whose boundary is given by the regular part 
	$$\text{\rm Reg}(u):=\left(\partial\O\cap \{x_d=0\}\right)\setminus\{0\}\,,$$
	which is $C^{1,\alpha}$ smooth by \cref{t:epsilon-regularity}.	Thus, $u$ is a classical solution to the PDE
	\be\label{e:smooth-one-phase}
	\Delta u=0\quad\text{in }\O^+\,,\qquad e_d\cdot \nabla u=m\quad\text{on }\O',\qquad |\nabla u|=q\quad\text{on }(\partial\O)^+\,.
	\ee
	We proceed perturbing $\Omega$ far from the origin. For any $R>1$, we set
	$$E_R:=B_R\setminus \overline B_{{1}/{R}}\,,$$
	and we fix a smooth vector field $\eta\in C^\infty_c(E_R,\R^d)$ satisfying \eqref{e:assumption-eta}. Let 
	$$\O_t:=\Phi_t(\O)\quad\text{for every}\quad t\in\R\,,$$
	where $\Phi_t$ is the flow of $\eta$ defined by \eqref{e:flow-of-eta}. Let $u_t:E_R\to\R$ be the solution of the PDE
	$$
	\Delta u_t=0\quad\text{in }\O_t\cap E_R^+\,,\qquad
	u_t=0\quad\text{on }\partial\O_t\cap E_R^+\,,\qquad
	e_d \cdot \nabla u_t = m\quad\text{on }\O_t \cap E_R'\,,
	$$
	satisfying $u_t=u$ on $(\partial E_R)^+.$ The rest of proof is divided in four steps.\medskip
	
	\noindent{Step 1.}  By \cref{l:secondstv} we already know that
	\begin{equation}\label{e:second-variation-with-w}
	\frac12	 \frac{\partial^2}{\partial t^2}\bigg|_{t=0} J(u_t,E_R) = \int_{\O\cap E_R^+} \nabla u\cdot(\delta^2 B)\nabla u -|\nabla w|^2 + \delta^2 Q\, dx + 2\int_{\O\cap E_R'}u\delta^2 m  \,dx',
	\end{equation}
	where $\delta B, \delta^2B$ are defined in \eqref{e:first-and-second-variation-B}, $\delta m, \delta^2 m$ in \eqref{e:first-and-second-variation-m}, $\delta^2 Q$ in \eqref{e:delta2Q} and $w=\delta u$ is the solution to
	\be\label{e:global-stable-solutions-equation-w-R}
	\begin{cases}
		-\Delta w=\dive((\delta B)\nabla u) &\text{in }\O\cap E_R^+\,,\\
		e_d \cdot \nabla w + e_d \cdot \delta B\nabla u =\delta m &\text{on }\O \cap E_R',\\
		w = 0 &\text{on }(\partial(\O\cap E_R))^+\,.
	\end{cases}
	\ee
	Clearly, since $u_t=u$ on $\R^d\setminus E_R$, we have that $J(u,E_R)\leq J(u_t,E_R)$ for $t$ sufficiently small. 
	By exploiting \cref{l:firstv} and \cref{l:secondstv}, we get that
	$$
	\frac{d}{dt}\bigg|_{t=0} J(u_t,E_R)=0,\quad\mbox{and}\quad\frac12	 \frac{\partial^2}{\partial t^2}\bigg|_{t=0} J(u_t,E_R)\geq0.
	$$
	\noindent{Step 2.} Fixed the vector field $\eta$, we define the function $u'$ to be such that $u'=0$ on $\O\cap (\partial E_R)^+$ and 
	\be\label{e:sty}
	\Delta u' = 0 \quad\mbox{in}\quad\O\cap E_{R}^+\,,\qquad
	e_d \cdot \nabla u' = 0 \quad\mbox{on}\quad \O\cap E_{R}'\,,\qquad
	u' = q(\eta\cdot \nu)\quad\mbox{on}\quad\partial\O\cap E_{R}^+\,.
	\ee
	where $\nu$ is the outer normal to $\partial \O$. We will show that $u'=w-\eta\cdot\nabla u$. Reasoning as in \cite{BMMTV}, we have that, since $\eta$ is supported in $E_R$ and since $\nabla u=-q\nu$ on $\partial\O\cap E_R$, we have:
	$$u'=w-\eta\cdot\nabla u\quad\text{on}\quad (\partial(\O\cap E_R))^+,$$
	while, a direct computation gives%in $\O\cap E_R^+$ %(using the repeated index summation convention)
	$$
	\dive\big((\delta B)\nabla u\big) = -\Delta(\eta\cdot\nabla u)\qquad \mbox{in }\O\cap E_R^+,
	$$
	which proves that $w-\eta\cdot\nabla u$ and $u'$ satisfy the same PDE inside $\Omega$. In order to prove that $w-\eta\cdot\nabla u$ and $u'$ satisfy the same boundary conditions on $\Omega'$, we compute
	\begin{align*}
	e_d\cdot\nabla(w-\eta\cdot\nabla u)&=\delta m-e_d \cdot \delta B\nabla u-e_d\cdot\nabla(\eta\cdot\nabla u)\\
	&= \delta m-\delta_{jd}\big[-\partial_i\eta_j\partial_iu-\partial_j\eta_i\partial_iu+\partial_i\eta_i\partial_ju\big] -e_d\cdot\nabla(\eta\cdot\nabla u)\\
	&=\delta m-\Big(m\dive\eta-\partial_i\eta_d\partial_iu-\partial_d\eta_i\partial_iu\Big) -e_d\cdot\nabla(\eta\cdot\nabla u)\\
	&= \partial_d\eta_i \partial_i u-e_d\cdot\nabla(\eta\cdot\nabla u)=0.
	\end{align*}
	\noindent{Step 3.} We next rewrite the second variation from \eqref{e:second-variation-with-w} in terms of $u'$ as in \cite[Proposition 7.12]{BMMTV}. Since $w=u'+\eta\cdot\nabla u$, an integration by parts gives:
	\begin{align*}
	-\int_{\O\cap E_R^+}|\nabla w|^2\,dx
	&=\int_{\O\cap E_R^+}\Big(|\nabla u'|^2-|\nabla(\eta\cdot\nabla u)|^2\Big)\,dx\,.
	\end{align*}
	On the other hand, the definition of $\delta^2B$ gives
	\begin{align*}
	\int_{\O\cap E_R^+}\nabla u\cdot(\delta^2 B)\nabla u\,dx
	&=\int_{\O\cap E_R^+}|D\eta(\nabla u)|^2+\nabla u\cdot(D\eta)^2(\nabla u)\,dx\\
	&\qquad-\int_{\O\cap E_R^+}\nabla u\cdot\Big((\eta\cdot\nabla)[\nabla\eta]\Big)\nabla u\,dx+ 2\,\dive\,\eta\,\nabla u\cdot \nabla\eta\nabla u\,dx\\
	&\qquad\qquad-\int_{\O\cap E_R^+}D^2 u(\nabla u)\cdot\eta (\dive\,\eta)\,dx+\int_{\partial(\O\cap E_R^+)}\frac12|\nabla u|^2\dive\,\eta (\eta\cdot\nu)\,d\HH^{d-1}.
	\end{align*}
	Now, since we have the identity
	\begin{align*}
	|\nabla\eta&[\nabla u]|^2+\nabla u\cdot(\nabla\eta)^2(\nabla u)-|\nabla(\eta\cdot\nabla u)|^2\\
	&-\nabla u\cdot\Big((\eta\cdot\nabla)[\nabla\eta]\Big)\nabla u-2\dive\,\eta\,\nabla u\cdot \nabla\eta(\nabla u)-D^2 u(\nabla u)\cdot\eta (\dive\,\eta)\\
	&=\dive\Big(-(\eta\cdot\nabla u)D^2u(\eta)+(\eta\cdot\nabla u)D\eta(\nabla u)-(\nabla u\cdot \nabla\eta(\nabla u))\eta-(\eta\cdot\nabla u)(\text{div}\,\eta)\nabla u\Big),
	\end{align*}
	substituting in \eqref{e:second-variation-with-w}, we obtain
	\begin{align*}
	\frac12&\frac{d^2}{dt^2}\Big|_{t=0} J(u_t,E_R) %=\int_{\O\cap E_R^+} \nabla u\cdot(\delta^2 B)\nabla u -|\nabla w|^2 + \delta^2 Q\, dx + 2\int_{\O\cap E_R'}u\delta^2 m  \,dx'\\
	=\int_{\O\cap E_R^+}|\nabla u'|^2 + \delta^2 Q\,dx + 2\int_{\O\cap E_R'}u\delta^2 m  \,dx'+\frac12\int_{\partial(\O\cap E_R^+)}|\nabla u|^2\dive \eta (\eta\cdot \nu)\,d\HH^{d-1}\\
	&\qquad+\int_{\partial(\O\cap E_R^+)}\left((\eta\cdot\nabla u)D\eta(\nabla u)-(\eta\cdot\nabla u)D^2u(\eta)-(\nabla u\cdot \nabla\eta(\nabla u))\eta-(\eta\cdot\nabla u)(\text{div}\,\eta)\nabla u\right)\cdot \nu\,d\HH^{d-1}
	\end{align*}
	Then, by combining the definition of $\delta^2m, \delta^2Q$ in \eqref{e:first-and-second-variation-m} and \eqref{e:delta2Q} with
	$$
	\nabla u= - q\nu\quad\mbox{ on }\partial \O\cap E_R^+,\qquad e_d \cdot \nabla u= m \quad\mbox{ on }\O\cap E_R',
	$$
	we get
	\begin{align*}
	\frac12&\frac{d^2}{dt^2}\Big|_{t=0} J(u_t,E_R)	=\int_{\O\cap E_R^+}|\nabla u'|^2\,dx +m\cancel{\int_{\O\cap E_R'}u\dive(\eta\dive\eta)\,dx'}+\frac12\cancel{\int_{\partial\O\cap E_R^+}(|\nabla u|^2 + q^2)\dive \eta (\eta\cdot \nu)}\,d\HH^{d-1}\\
	&\qquad+\int_{\partial\O\cap E_R^+}\left((\eta\cdot\nabla u)D\eta(\nabla u)-(\eta\cdot\nabla u)D^2u(\eta)-(\nabla u\cdot \nabla\eta(\nabla u))\eta-\cancel{(\eta \cdot \nabla u)(\dive \eta)\nabla u}\right)\cdot \nu\,d\HH^{d-1}\\
	&\qquad +\cancel{\int_{\O\cap E_R'}\text{div}\,\eta(\eta\cdot\nabla u)(e_d \cdot \nabla u)}\,d\HH^{d-1}\\
	%&=\int_{\O\cap E_R^+}|\nabla u'|^2\,dx +\int_{\partial\O\cap E_R^+}q(\eta\cdot \nu)\Big(\nu \cdot D^2 u (\eta) +\cancel{\nabla u\cdot D\eta(\nabla u)}-\cancel{\nabla u\cdot \nabla\eta(\nabla u)}\Big)\,d\HH^{d-1}\\
	&=\int_{\O\cap E_R^+}|\nabla u'|^2\,dx+\int_{\partial\O\cap E_R^+}q(\eta\cdot\nu)^2 (\nu \cdot D^2u (\nu))\,d\HH^{d-1}\\
	&=\int_{\O\cap E_R^+}|\nabla u'|^2\,dx-\int_{\partial\O\cap E_R^+}q^2(\eta\cdot\nu)^2 H\,d\HH^{d-1},%\\
	%&=\int_{\O\cap E_R^+}|\nabla u'|^2\,dx-\int_{\partial\O\cap E_R^+}(u')^2|H|\,d\HH^{d-1},
	\end{align*}
	where in the last line we use that $u'=q(\eta\cdot \nu)$ on $\partial\O\cap E_R$ and where $H$ is the mean curvature of $\partial \O\cap E_R^+$. Indeed, since $\Delta u=0$, we have that 
	\begin{equation}\label{e:mean-curvature-definition}
	H=-\frac{\nabla u\cdot\nabla^2u(\nabla u)}{|\nabla u|^3}\quad\text{on}\quad \partial\Omega\cap\{x_d>0\},
	\end{equation}
	which together with $\nabla u=-q\nu$, gives $\displaystyle H=-\frac1q\nu \cdot \nabla^2u (\nu)$.

	\noindent{Step 4.} We next prove that the vector field
	$$\nabla_{x'}u:=\nabla u-(e_d\cdot\nabla u)e_d$$
	is never zero on $\partial\Omega\setminus \{0\}$, where we recall that $\Omega=\{u>0\}$. Indeed, suppose by contradiction that there exists $z_0 \in \partial\Omega\setminus\{0\}$ such that $|\nabla_{x'} u |(z_0)=0$. In particular, this means that
	$$\nabla u(z_0)=(e_d\cdot\nabla u(z_0))e_d.$$
	Then we have three possibilities:
	\begin{enumerate}
		\item Suppose that $z_0 \in \partial \O\cap \{x_d>0\}$.	Since $z_0$ is on the boundary of $\Omega$, we have that $|\nabla u(\lambda z_0)|\neq 0$ and thus also $e_d\cdot\nabla u(z_0)\neq 0$. Now, since the line $\{\lambda z_0\ :\ \lambda>0\}$ remains in the boundary $\partial\Omega$, we have that $z_0$ is orthogonal to $\nabla u$ and so $z_0\cdot e_d=0$, which is a contradiction with the assumption $z_0\in\{x_d>0\}$.\smallskip
		\item Suppose that $z_0 \in \partial\Omega\cap \{x_d=0\}$. Since 
		$$e_d \cdot \nabla u(z_0)=m\qquad\text{and}\qquad |\nabla_{x'} u (z_0)|=\sqrt{q^2-m^2},$$
		we have a contradiction with the choice $m\in(-q,q)$.
		\item Suppose that $z_0 \in \Omega\cap \{x_d=0\}$. Since 
		$$|\nabla_{x'}u(\lambda z_0)|=|\nabla_{x'}u(z_0)|=0\qquad\text{for every}\qquad \lambda>0,$$
		we have that $u$ is constant on the line $\lambda z_0$. But then, $u(z_0)=0$, which is a contradiction with the choice $z_0\in\Omega$.\\
	\end{enumerate}
	\noindent{Step 5.} 
	Let $\xi:\{x_d\ge 0\}\to\R^d$ be a smooth vector field $\xi\in C^\infty_c(\{x_d\ge 0\}\cap E_R)$ such that: 
	$$\xi\cdot e_d=0\ \text{ on }\ \{x_d=0\}\qquad\text{and}\qquad \xi=-\frac{\nabla_{x'}u}{|\nabla_{x'}u|^2}\ \text{ on }\ \partial\Omega\cap\{x_d\ge 0\}\cap E_R.$$
	%Indeed, since $\O\cap E_R^+$ has $\mathcal{H}^{d-1}$-almost $C^2$-boundary, there exists an extension of the unit normal vector $\nu$ to the whole space. In particular let $\nu =(\nu',e_d\cdot \nu)$.
	Given any $\varphi\in C^\infty_c(E_R)$, we consider the vector field
	$$
	\eta:=\varphi\xi,
	$$
	and the function $u'$ defined in \eqref{e:sty}. Then, on the boundary $\partial\Omega\cap E_R^+$ we have 
	$$u' =-\eta\cdot\nabla u= -\varphi \xi\cdot\nabla u=\varphi.$$
	Thus, by the minimality of $u'$ in $\O\cap E_R^+$, we get
	\begin{align*}
	\int_{\O\cap E_R^+}|\nabla \varphi|^2\,dx-\int_{\partial\O\cap E_R^+}\varphi^2 H\,d\HH^{d-1}&\ge \int_{\O\cap E_R^+}|\nabla u'|^2\,dx-\int_{\partial\O\cap E_R^+}(u')^2H\,d\HH^{d-1}\\
	&=\frac12\frac{d^2}{dt^2}\Big|_{t=0}\mathcal J(u_t,E_R)\geq 0.
	\end{align*}
	This concludes the proof of \cref{t:stability-inequality-CJK-type}.\qed
	
	\subsection{Positivity of the mean curvature $H$}
	In the next lemma we show that the mean curvature $H$ (defined in \eqref{e:mean-curvature-definition}) on non-flat smooth minimizing cones is positive on the free boundary.
	\begin{lemma}\label{l:curvature}
		Let $u:\R^d\cap \{x_d\geq 0\}\to\R$ be a non-negative one-homogeneous function and $\O=\{u>0\}$. Then 
		$$
		H\geq 0\quad\mbox{on}\quad\partial\O\cap\{x_d>0\}\qquad\mbox{and}\qquad|\nabla u|\leq q\quad\mbox{in}\quad \overline{\O}, $$
		where $H$ is the mean curvature given by \eqref{e:mean-curvature-definition}. Moreover, if $H=0$ at some point $x_0\in\partial\Omega\cap\{x_d>0\}$, then $u$ is a half-plane solution in the sense of \cref{def:half-plane}.
	\end{lemma}
	\begin{proof}
		We first prove that $|\nabla u|\le q$ on $\overline\Omega\setminus\{0\}$. Since $|\nabla u|^2$ is $0$-homogeneous, it achieves its maximum at some point $x_0\in\overline\Omega\setminus\{0\}$. Since by the Bochner's identity 
		$$\Delta (|\nabla u|^2)=2\|\nabla^2u\|_2^2+2\nabla u\cdot\nabla(\Delta u)\ge 0\quad\text{in}\quad\Omega\cap\{x_d>0\},$$
		we have that $x_0\notin \Omega\cap\{x_d>0\}$. In order to exclude that $x_0\in\Omega\cap\{x_d=0\}$,
		we notice that since we have the Neumann boundary condition on $\{x_d=0\}$, $x_0$ needs to be a maximum for both $|\nabla u|^2$ and 
		$$|\nabla_{x'}u|^2=\sum_{j=1}^{d-1}(\partial_ju)^2=|\nabla u|^2-(\partial_du)^2$$
		on the hyperplane $\{x_d=0\}$. 
		Hence, let us compute the tangential Laplacian of $|\nabla_{x'}u|^2$. Using again the Bochner's identity, this time in $\R^{d-1}$, we get
		\begin{align*}
		\Delta_{x'} (|\nabla_{x'} u|^2)&=2\|\nabla_{x'}^2u\|_2^2+2\nabla_{x'} u\cdot\nabla_{x'}(\Delta_{x'} u)\ge 2\nabla_{x'} u\cdot\nabla_{x'}(\Delta_{x'} u)\quad\text{in}\quad\Omega\cap\{x_d=0\}.
		\end{align*}
		Now, since $u$ is harmonic, we have that 
		$$\Delta_{x'}u=\Delta u-\partial_{dd}u=-\partial_{dd}u.$$
		Then, using the Neumann condition $\partial_du\equiv m$ on $\{x_d=0\}\cap\Omega$, we obtain that $$\partial_{jdd}u=\partial_d(\partial_j(\partial_du))=\partial_d(\partial_jm)=0\quad\text{on}\quad\{x_d=0\}\cap\Omega,$$ 
		for every $j=1,\dots,{d-1}$. Thus,
		\begin{align*}
		\Delta_{x'} (|\nabla_{x'} u|^2)&\ge 2\nabla_{x'} u\cdot\nabla_{x'}(\Delta_{x'} u)=0\quad\text{in}\quad\Omega\cap\{x_d=0\},
		\end{align*}
		which means that the maximum of $|\nabla u|^2$ cannot be achieved in $\Omega\cap\{x_d=0\}$. Finally, this implies that the maximum of $|\nabla u|^2$ is achieved on the free boundary, which proves the desired inequality. \smallskip
		
		Let now $x_0\in\partial\Omega\cap\{x_d>0\}$. Since $|\nabla u|^2$ is subharmonic in $\Omega\cap\{x_d>0\}$ and achieves its maximum at $x_0$, we have that
		$$\nu\cdot\nabla(|\nabla u|^2)\ge 0,$$
		at $x_0$, where $\nu=-\frac{1}{|\nabla u|}\nabla u$ is the exterior normal, an equality being achieved if and only if $|\nabla u|^2\equiv q$ (which in particular implies that $u$ is linear by the Bochner's formula). Now, the claim follows since the mean curvature $H$ is given by \eqref{e:mean-curvature-definition}. 
	\end{proof}

	\subsection{Non-variational formulation of the stability inequality} 
	Following \cite{js}, in this subsection we rewrite the criterion for stability \eqref{e:stability-inequality} in the form of non-existence of homogeneous strict subsolutions of the linearized problem
	\be\label{e:stability-linearized}
	\Delta \varphi= 0\quad\mbox{in }\,\{u>0\}^+,\,\quad\nu \cdot \nabla \varphi - H\varphi = 0\quad\mbox{in }\,(\partial\{u>0\})^+,\,\quad e_d \cdot \nabla \varphi = 0 \quad\mbox{in }\{u>0\}',
	\ee
	where $\nu$ and $H$ are respectively the exterior unit normal  and the mean curvature of the free boundary.

	\begin{lemma}\label{l:stability-hardy}
		Let $u:\R^d\cap \{x_d\geq 0\}\to\R$ be a non-negative one-homogeneous function and $\O=\{u>0\}$. Let $q>0, m \in (-q,q)$ and suppose that $\R^{d-1}\setminus \{0\}\subset \text{\rm Reg}(u)$ and
		$$
		\Delta u =0 \quad\mbox{in }\,\O^+,\qquad |\nabla u|=q \quad\mbox{in }\,(\partial\O)^+,\qquad e_d\cdot \nabla u = m\quad\mbox{in } \,\O'
		$$
		% Let $u \colon \R^d \cap \{x_d\geq 0\}\to \R$ be as in \cref{p:equivalence-stability} and let $\Omega=\{u>0\}$.
		Suppose that there are $\alpha>0$  and a non-negative function $\phi\in C^2(\Omega\setminus\{0\})\cap C^1(\overline\Omega\setminus\{0\})$, homogeneous of degree $-\alpha$, such that
		\be\label{e:stability-hardy}
		\Delta \phi > \gamma \frac{\phi}{|x|^2}\quad\mbox{in }\,\O^+,\qquad
		\nu \cdot \nabla \phi - H\phi \leq 0\quad\mbox{on }\,(\partial\O)^+,\qquad
		e_d \cdot \nabla \phi\geq 0\quad\mbox{on }\,\O',
		\ee
		where $\gamma$ is such that
		\be\label{e:assumption-gamma}
		\gamma \geq \left(\frac{d}{2}-1-\alpha\right)^2.
		\ee
		Then, the stability inequality \eqref{e:stability-inequality} does not hold for every $v \in C^\infty_c(\R^d\setminus \{0\})$.
	\end{lemma}
	\begin{proof}
		The proof is an adaptation of \cite[Proposition 2.1]{js} and \cite[Proposition 2.2]{js} to the linearized problem \eqref{e:stability-linearized}. By exploiting the homogeneity of $\phi$, we get that its restriction on $\mathbb{S}^{d-1}$ satisfies
		\be\label{e:1}
		\Delta_{\mathbb{S}^{d-1}} \phi \geq \left(\gamma + \alpha(d-2-\alpha)\right)\phi \quad\mbox{in }\O_{\mathbb{S}}^+,\qquad
		e_{d}\cdot \nabla_{
			\mathbb{S}^{d-1}}\phi \geq 0 \quad\mbox{on }\O_\mathbb{S}',
		\ee
		where $\Delta_{\mathbb{S}^{d-1}}$ is the Laplace-Beltrami operator on $\mathbb{S}^{d-1}$ and $\O_S=\O\cap \mathbb{S}^{d-1}, \partial\O_\mathbb{S}=\partial\O\cap \mathbb{S}^{d-1}$. Moreover, if we set $\lambda = \gamma + \alpha(d-2-\alpha)$, then \eqref{e:assumption-gamma} implies
		$$	\lambda \geq \frac14(d-2)^2.$$
		By testing \eqref{e:1} with $\phi$ and integrating by parts, we get
		$$
		\int_{\O_\mathbb{S}^+}|\nabla_{\mathbb{S}^{d-1}} \phi|^2\,d\HH^{d-1} -
		\int_{(\partial\O_\mathbb{S})^+}H \phi^2\,d\HH^{d-2}< -\lambda \int_{\O_\mathbb{S}^+} \phi^2\,d\HH^{d-1},
		$$
		which gives the following bound from above on the first spherical Robin eigenvalue $\Lambda$:  
		\be\label{e:2}
		-\Lambda =\min_{v\in H^1(\Omega_\mathbb{S})}\dfrac{\int_{\O_\mathbb{S}^+}|\nabla_{\mathbb{S}^{d-1}} v|^2\,d\HH^{d-1} - \int_{(\partial\O_\mathbb{S})^+}H v^2\,d\HH^{d-2}}{\int_{\O_\mathbb{S}^+} v^2\,d\HH^{d-1}}< -\lambda\,,
		\ee
		where, by \cite{Robin-reference}, the minimum is realized in a non-negative smooth function $v$ satisfying
		$$
		\Delta_{\mathbb{S}^{d-1}} v = \Lambda v\quad\mbox{on }(\O_\mathbb{S})^+, \qquad \nu \cdot \nabla_{\mathbb{S}^{d-1}} v -H v =0 \quad\mbox{on }(\partial \O_\mathbb{S})^+,\qquad e_d \cdot \nabla_{\mathbb{S}^{d-1}} v =0\quad\mbox{on }\O_\mathbb{S}'.
		$$
		Now, since by construction $\Lambda >\lambda \geq (d-2)^2/4$, we can find $\beta >0$ such that
		$$
		\Lambda >\beta >\frac14 (d-2)^2.
		$$
		Then, let $g \colon \R^+\to \R$ be a solution to the following ODE
		$$
		f''(r) + \frac{d-1}{r}f'(r) + \frac{\beta}{r^2}f(r) = 0.
		$$
		Precisely, we can take 
		$$f(r)=r^{\sigma}\cos\big(\zeta\ln r\big),$$
		where the constants $\sigma$ and $\zeta$ are given by
		$$\sigma=-\frac{d-2}{2}\qquad\text{and}\qquad \zeta=\sqrt{\beta-\frac{(d-2)^2}{4}}.$$
		Since $f$ oscillates in a neighborhood of $r=0$, we can find two consecutive zeros $r_0$ and $r_1$ such that $r_0<1<r_1$, and we define 
		$$g(r):=|f(r)|\ind_{[r_0,r_1]}(r)$$
		We next choose $R>1$ such that $1/R<r_0<1<r_1<R$ and we consider the function 
		$$\varphi(x):=g(|x|) v\left(\frac{x}{|x|}\right).$$ 
		By construction, $\varphi$ is non-negative and supported in $E_R^+$, and we have 
		$$
		\Delta \varphi =  v\left(g''+ \frac{d-1}{|x|}g' + \frac{\Lambda}{|x|^2}g\right) = \frac{\Lambda -\beta}{|x|^2} \varphi>0 \quad\mbox{in}\quad\O\cap E_R^+.
		$$
		Moreover, since $g$ is radial, the following boundary conditions are satisfied:
		\begin{align*}     
		\begin{cases}
		\nu \cdot \nabla \varphi - H\varphi= 0 &\mbox{on}\quad \partial\O\cap E_R^+,\\
		e_d \cdot \nabla \varphi = 0 & \mbox{on}\quad\O\cap E_R'.    \end{cases}
		\end{align*}
		Therefore, integrating by parts, we obtain 
		\begin{align*}
		\int_{\O\cap E_R^+}|\nabla \varphi|^2\,dx &=
		-\int_{\O\cap E_R^+}\varphi\Delta \varphi\,dx +
		\int_{\partial\O\cap E_R^+}\varphi (\nu \cdot \nabla \varphi)\,d\HH^{d-1} - \int_{\O\cap E_R'}\varphi (e_d \cdot \nabla \varphi)\,dx' \\
		&\leq
		\int_{\partial\O\cap E_R^+}H\varphi^2\,d\HH^{d-1},
		\end{align*}
		which violates the validity of the stability inequality \eqref{e:stability-inequality}. 
	\end{proof}

	\subsection{Proof of \cref{t:stability}} Let $u\colon \R^d\cap \{x_d\geq0\} \to \R$ be a non-negative one-homogeneous global minimizer of the functional $J$ from \eqref{e:def-J-intro} with $q>0$ and $m \in (-q,q).$ In view of \cref{l:curvature}, we can focus on the case 
	$$
	H>0 \quad \mbox{on }(\partial \O)^+.
	$$
	Since the classification in \cref{s:global2} implies that in two dimensions the half-plane solution $h_{q,m,e_1}$ is the  unique global one-homogeneous minimizer, 	in this section we address the remaining cases $d\in \{3,4\}$ by showing the existence of strict subsolutions to \eqref{e:stability-hardy} satisfying the assumptions of \cref{l:stability-hardy}.\medskip
	
	By comparing the linearized problem \eqref{e:stability-linearized} with the one arising in the one-phase problem (see \cite[Section 2.3]{js}), we can easily see that the strict subsolutions constructed in \cite[Section 3]{js} and \cite[Section 4]{js} still satisfy the first two equations in \eqref{e:stability-hardy}. Indeed, it remains to prove that their subsolutions satisfy the third inequality in \eqref{e:stability-hardy}, which arises from the capillarity term in $J$. We do this in the following remark. 

	\begin{remark}\label{r:ddd}
		Given the Hessian matrix $D^2 u\colon \{x_d\geq 0\}\to \R^{d\times d}$ of $u$, we denote with $\lambda_k(x)$, with $k=1,\dots,d$, its eigenvalues at a point $x \in \{x_d\geq 0\}$. It is worth noticing that the spectrum of $(D^2 u)^2$ consists of the squared eigenvalues of $(D^2 u)$. Then, for every $j=1,\dots,d$, 
		$$\partial_d\lambda_j(x)=0\quad\text{for every}\quad x\in\Omega'.$$ 
		Indeed, by the Neumann boundary condition $\partial_du=m$ on $\Omega'$, we have that $\partial_d(\partial_{ij}u)=0$ whenever $\{i,j\}\neq\{d,d\}$. On the other hand, when $\{i,j\}=\{d,d\}$, by the harmonicity of $u$ we have 
		$$\partial_{ddd}u=\partial_d\Big(-\sum_{j=1}^{d-1}\partial_{jj}u\Big)=-\sum_{j=1}^{d-1}\partial_{jj}(\partial_du)=0.$$
	\end{remark}
	
	Now, we split the proof in two cases.\medskip
	
	\noindent Case $d=3$. Let $\alpha>0$ to be chosen later and consider
	$$
	\phi^2(x) :=\sum_{\lambda_k>0}\lambda_k(x)^2 + \sum_{\lambda_k<0}\lambda_k(x)^2
	\qquad\mbox{and}\qquad \varphi(x)= \phi(x)^\alpha.
	$$
	First, assume that $\phi$ is not identically zero, that is $u$ is not an half-plane solution in the sense of \cref{def:half-plane}. Since $u$ is one-homogeneous we have that $\varphi$ is homogeneous of degree $-\alpha$ and, by \cite[Corollary 3.2]{js}, we already know that 
	\be\label{e:competitor-interior}
	%\Delta \varphi \geq 0\quad\mbox{in }\O \cap \{\varphi=0\}\qquad \mbox{and}\qquad 
	\Delta \varphi \geq \alpha(\alpha+1)\frac{\varphi}{|x|^2}\quad\mbox{in}\quad\Omega.
	\ee
	On the other hand, by \cite[Section 3.2]{js} and the computations in \cite[Corollary 3.4]{js}, we have
	\be\label{e:competitor-fb}
	\nu\cdot \nabla \varphi - H \varphi \leq 0 \quad\mbox{on }(\partial\O)^+, \qquad\mbox{if }\alpha\leq\frac12.
	\ee
	Moreover, by \cref{r:ddd}, we have 
	\be\label{e:competitor-neum}
	e_3 \cdot \nabla \varphi =  0 \quad\mbox{on }\O'.
	\ee
	Finally, in order to have the condition \eqref{e:assumption-gamma}  from \cref{l:stability-hardy} satisfied, we must have
	\be\label{e:competitor-gamma}
	\gamma:=\alpha(\alpha+1)\geq \left(\frac12-\alpha\right)^2,\qquad\mbox{that is}\qquad \alpha \geq \frac18.
	\ee
	By choosing $\alpha \in (\sfrac18,\sfrac12)$, all the conditions \eqref{e:competitor-interior}, \eqref{e:competitor-fb}, %\eqref{e:competitor-neum} 
	and \eqref{e:competitor-gamma} are fulfilled  and so, by \cref{l:stability-hardy}, the stability inequality \eqref{e:stability-inequality} must be violated for some test function $v \in C^\infty_c(\R^d\setminus \{0\})$.\medskip

	\noindent Case $d=4$. Suppose that $u$ is not an half-plane solution and set
	$$
	\phi^2(x) := \sum_{\lambda_k>0}\lambda_k(x)^2 + 4\sum_{\lambda_k<0}\lambda_k(x)^2 \qquad\mbox{and}\qquad \varphi=\phi^{\sfrac13}. %where $\lambda_k$ with $k=1,\dots,4$ are the eigenvalues of the Hessian $D^2 u$. 
	$$
	By \cite[Theorem 4.1]{js} and the proof of \cite[Corollary 3.2]{js} we know that 
	\be\label{e:competitor-interior-4}
	\Delta \varphi \geq \frac49\frac{\varphi}{|x|^2}\quad\mbox{in}\quad\Omega,
	\ee
	where the inequalities are understood in a viscosity sense. On the other hand, by \cite[Section 4.3]{js} (in particular by (4.8) in \cite{js}), we have
	\be\label{e:competitor-fb-4}
	\nu\cdot \nabla \varphi - H \varphi \leq 0 \quad\mbox{on }(\partial\O)^+.
	\ee
	Actually, as pointed out in \cite{js}, \eqref{e:competitor-fb-4} is a strict inequality. Otherwise, by choosing a system of coordinates at $x_0 \in (\partial\O)^+$ such that $e_1=x_0/|x_0|$ and $e_4 = \nu(x_0)$ we would have
	$$
	\lambda_1=0,\,\, \lambda_2 >0,\,\, \lambda_3 = \lambda_4 = - H <0.
	$$
	In view of the Remark in the proof of \cite[Theorem 4.1]{js}, such equality can be improved by exploiting the strict convexity of the function $\phi$ with respect to the eigenvalues $\lambda_k$.\\
	Finally, using again \cref{r:ddd}, we get that $e_4 \cdot \nabla \varphi =  0$ on $\O'$. Thus,  $\varphi$ fulfills all the assumptions of \cref{l:stability-hardy}. More precisely, since by \eqref{e:competitor-interior-4} we have $\gamma=4/9$, we get
	$$
	\frac49=\gamma =\alpha(\alpha+1)\geq \left(1-\alpha\right)^2 = \frac49,
	$$
	which implies the non validity of the stability inequality \eqref{e:stability-inequality} by \cref{l:stability-hardy}.

	\appendix
	
	\section{On the distance between $A$-harmonic functions}\label{s:schauder}
	In this section we recall some useful regularity result for $A$-harmonic functions. Indeed, given $\lambda\ge 1$, $C>0, \alpha>0$ and a bounded open domain $D\subset \R^d$, we consider family $\mathcal A(D,\lambda,C,\alpha)$ of real symmetric matrices $A=(a_{ij})_{i,j}$ with variable coefficients $a_{ij}:\overline D\to\R$ for which: 
	\begin{enumerate}
		\item[($\mathcal A$1)] $\lambda^{-1}\text{\rm Id}\le A(x)\le \lambda\text{\rm Id}\ $ for every $\ x\in\overline D\,;$\medskip
		\item[($\mathcal A$2)]  for every couple of indices $1\le i,j\le d$, we have
		$$|a_{ij}(y)-a_{ij}(x)|\le C|x-y|^\alpha\qquad\text{for every}\qquad x,y\in\overline D\,.$$
	\end{enumerate}
	We recall that by the $C^{1,\alpha}$ Schauder estimates we have the following uniform bounds:
	\begin{enumerate}[(S1)]
		\item There is a constant $L$ (depending on $d$, $\rho$, $C$, $\lambda$, $\alpha$) such that
		$$\|\nabla u_A\|_{L^\infty(D)}\le L\qquad\text{for every matrix}\qquad A\in \mathcal A(D,\lambda,C,\alpha).$$
		\item There are a constant $M>0$, an exponent $\gamma>0$ and a radius $R>0$ (depending on $d$, $\rho$, $C$, $\lambda$, $\alpha$) such that for every $A\in \mathcal A(D,\lambda,C,\alpha)$ and every $x_0\in\overline D$ we have:
		$$\Big|u_A(x+x_0)-u_A(x_0)-x\cdot\nabla u_A(x_0)\Big|\le M |x|^{1+\gamma}\quad\text{whenever}\quad |x|\le R.$$
	\end{enumerate}	
	\begin{theorem}\label{t:app-schauder}
		Given $\rho \in (0,1]$ and $D=B_{1+\rho}\setminus \overline B_1$, we consider the solution $u_A:\overline D\to\R$
		to the problem
		$$
		\text{\rm div}(A\nabla u_A)=0\quad\text{in}\quad D\,,\qquad
		u_A=0\quad\text{on}\quad \partial B_1\,,\qquad
		u_A=\rho\quad\text{on}\quad \partial B_{1+\rho}\,.
		$$
		for some $A$ in $\mathcal A(D,\lambda,C,\alpha)$.
		Then there are constants $K$, $\eps_0>0$ and $\sigma\in(0,1)$, depending on $\rho$, $C$, $d$, $\alpha$ and $\lambda$, such that if $$A,B\in\mathcal A(D,\lambda,C,\alpha)$$ are two symmetric matrices satisfying the bounds
		\begin{equation}\label{e:app-bounds-A-B}
		\frac1{1+\eps}B(x)\le A(x)\le (1+\eps)B(x)\qquad\text{for every}\qquad x\in\overline D\,,
		\end{equation}
		for some $\eps\in(0,\eps_0)$, then
		$$\|u_A-u_B\|_{L^\infty(D)}  + \|\nabla u_A-\nabla u_B\|_{L^\infty(D)}\le K\eps^\sigma\,.$$
	\end{theorem}	
	\begin{proof}
		Testing the optimality of $u_A$ with $u_B$ we have
		\begin{align*}
		\int_{D}\nabla u_A\cdot B(x)\nabla u_A\,dx&\le (1+\eps)\int_{D}\nabla u_A\cdot A(x)\nabla u_A\,dx\\
		&\le (1+\eps)\int_{D}\nabla u_B\cdot A(x)\nabla u_B\,dx\le (1+\eps)^2\int_{D}\nabla u_B\cdot B(x)\nabla u_B\,dx\,,
		\end{align*}	
		which implies
		\begin{align*}
		\int_{D}|\nabla (u_A-u_B)|^2\,dx&\le \lambda	\int_{D}\nabla (u_A-u_B)\cdot B(x)\nabla (u_A-u_B)\,dx\\
		&\le \lambda 3\eps\int_{D}\nabla u_B\cdot B(x)\nabla u_B\,dx\\
		&\le \lambda 3\eps\int_{D}\nabla h\cdot B(x)\nabla h\,dx\le \lambda^2 3\eps \int_{D}|\nabla h|^2\,dx\le C_d\lambda^2\eps\,,
		\end{align*}	
		where $h$ is the harmonic function 		
		$$
		\Delta h=0\quad\text{in}\quad D,\qquad
		h=0\quad\text{on}\quad \partial B_1,\qquad	h=\rho\quad\text{on}\quad \partial B_{1+\rho},
		$$
		and in particular there is a  dimensional constant $C_d>0$ such that
		$$\int_{D}|\nabla h|^2\,dx\le C_d\,.$$
		Now, since $u_A-u_B\in H^1_0(D)$, by the Poincaré inequality in $D$, we have
		$$\int_{D} |u_A-u_B|^2\,dx\le \lambda^2C_d\eps\,.$$
		Using (S1), we get
		\begin{equation}\label{e:app-infty-est-old}
		\frac{C_d}{L^{d}}	\|u_A-u_B\|_{L^\infty(D)}^{d+2}\le \int_{D} |u_A-u_B|^2\,dx\le \lambda^2C_d\eps\,,
		\end{equation}
		and so
		\begin{equation}\label{e:app-infty-est}
		\|u_A-u_B\|_{L^\infty(D)}\le C_dL^{\frac{d}{d+2}}\lambda^{\frac{2}{d+2}}\eps^{\frac1{d+2}}\,.
		\end{equation}
		We now pick a point
		$x_0\in \partial B_{1}$ and we aim to estimate $|\nabla u(x_0)-\nabla h(x_0)|$, the cases $x_0\in\partial B_{1+\rho}$ and $x_0\in B_{1+\rho}\setminus\overline B_1$ being analogous. Using (S2) we have that, for any radius $r\in(0,R)$,
		$$\Big|u_A(x+x_0)-x\cdot\nabla u_A(x_0)\Big|\le M r^{1+\gamma}\quad\text{and}\quad \Big|u_B(x+x_0)-x\cdot\nabla u_B(x_0)\Big|\le M r^{1+\gamma}\quad\text{for every}\quad |x|=r.$$
		Thus, by the triangular inequality and \eqref{e:app-infty-est}, we get
		$$\Big|x\cdot\nabla u_A(x_0)-x\cdot\nabla u_B(x_0)\Big|\le 2Mr^{1+\gamma}+C_dL^{\frac{d}{d+2}}\lambda^{\frac{2}{d+2}}\eps^{\frac1{d+2}}\,.$$	
		Taking $x$ to be the inwards normal times $r$, we obtain
		$$\Big|\nabla u_A(x_0)-\nabla u_B(x_0)\Big|\le 2Mr^{\gamma}+\frac1rC_dL^{\frac{d}{d+2}}\lambda^{\frac{2}{d+2}}\eps^{\frac1{d+2}}\,.$$	
		Choosing $r=\eps^{\frac1{(d+2)(1+\gamma)}}$ and $\eps_0=\rho^{(d+2)(1+\gamma)},$
		we finally get
		$$\Big|\nabla u_A(x_0)-\nabla u_B(x_0)\Big|\le \Big(2M+C_dL^{\frac{d}{d+2}}\lambda^{\frac{2}{d+2}}\Big)\eps^{\frac{\gamma}{(d+2)(1+\gamma)}}\,,$$	
		which concludes the proof.
	\end{proof}
	Similarly, we also have the following result in a half-ball.

	\begin{corollary}\label{t:app-schauder-half-ball}
		Given $\rho\in(0,1]$ we consider the annulus $D=B_{1+\rho}\setminus \overline B_1$ in $\R^d$. and, for every matrix $A\in \mathcal A(D,\lambda,C,\alpha)$, we denote by $v_A:\overline D\to\R$
		the solution to the problem
		$$\begin{cases}
		\text{\rm div}(A\nabla v_A)=0\quad\text{in}\quad D\cap\{x_d>0\}\,,\\
		v_A=0\quad\text{on}\quad \partial B_1\cap\{x_d>0\}\,,\\
		v_A=\rho\quad\text{on}\quad \partial B_{1+\rho}\cap\{x_d>0\}\,\\
		e_d\cdot A\nabla v_A=0\quad\text{on}\quad D\cap\{x_d=0\}\,.
		\end{cases}$$
		Then there are constants $K$, $\eps_0>0$ and $\sigma\in(0,1)$, depending on $\rho$, $C$, $d$, $\alpha$ and $\lambda$, such that if $A$ and $B$ are two symmetric matrices in $\mathcal A(D,\lambda,C,\alpha)$ satisfying the bounds \eqref{e:app-bounds-A-B} for some $\eps\in(0,\eps_0)$, then
		$$\| v_A-v_B\|_{L^\infty(D\cap \{x_d\ge 0\})} +\|\nabla v_A-\nabla v_B\|_{L^\infty(D\cap \{x_d\ge 0\})}\le K\eps^\sigma\,.$$
	\end{corollary}	
	\begin{corollary}\label{c:schauder.easy}
		Given $r>0$, we consider the solution $u_A:\overline B_r\to\R$ to the problem
		$$
		\text{\rm div}(A\nabla p_A)=f\quad\text{in}\quad B_r\,,\qquad
		p_A=g\quad\text{on}\quad \partial B_r\,,
		$$
		for some $A \in \mathcal{A}(B_r,\lambda,C,\alpha), g \in H^1(B)$ and $f \in C^{0,\alpha}(\overline{B_r})$ such that
		$$|f(y)-f(x)|\le C|x-y|^\alpha\qquad\text{for every }\, x,y\in\overline B_r\,.
		$$
		Then there are constants $K$, $\eps_0>0$ and $\sigma\in(0,1)$, depending on $\rho$, $C$, $d$, $\alpha$ and $\lambda$, such that if $A$ and $B$ are two symmetric matrices in $A(B_r,\lambda,C,\alpha)$ satisfying the bounds \eqref{e:app-bounds-A-B} for some $\eps\in(0,\eps_0)$, then
		$$\| p_A-p_B\|_{L^\infty(B_r)} +\|\nabla p_A-\nabla p_B\|_{L^\infty(B_r)}\le K\eps^\sigma\,.$$
	\end{corollary}

\end{document}